\documentclass[a4paper,12pt]{amsart}

\usepackage{amssymb, amscd}
\usepackage{latexsym,epsfig}
\usepackage[titletoc,title]{appendix}
\usepackage[all]{xy}
\usepackage{pst-all}
\numberwithin{equation}{section}
\def\today{\ifcase\month\or Jan\or Febr\or  Mar\or  Apr\or May\or Jun\or  Jul\or Aug\or  Sep\or  Oct\or Nov\or  Dec\or\fi \space\number\day, \number\year}

\newcommand{\CC}{\mathbb C}
\newcommand{\EE}{\mathbb E}

\newcommand{\PP}{\mathbb P}
\newcommand{\QQ}{\mathbb Q}
\newcommand{\RR}{\mathbb R}
\newcommand{\VV}{\mathbb V}

\newcommand{\ZZ}{\mathbb Z}

\numberwithin{equation}{section}

\newtheorem{theorem}{Theorem}[section]
\newtheorem{lemma}[theorem]{Lemma}
\newtheorem{proposition}[theorem]{Proposition}
\newtheorem{corollary}[theorem]{Corollary}

\newtheorem{definition-lemma}[theorem]{Definition-Lemma}

\theoremstyle{definition}

\newtheorem{example}[theorem]{Example}

\theoremstyle{remark}
\newtheorem{remark}[theorem]{Remark}

\newtheorem{question}[theorem]{Question}

\newtheorem{construction}[theorem]{Construction}
\newtheorem{notation}[theorem]{Notation}

\begin{document}

\title[]{Siegel Modular Forms of Genus $2$ and Level $2$}
\author{Fabien Cl\'ery}
\address{Department Mathematik,
Universit\"at Siegen,
Emmy-Noether-Campus,
Walter-Flex-Strasse 3, 
57068 Siegen, Germany}
\email{cleryfabien@gmail.com}

\author{Gerard van der Geer}
\address{Korteweg-de Vries Instituut, Universiteit van
Amsterdam, Postbus 94248,
1090 GE  Amsterdam, The Netherlands.}
\email{geer@science.uva.nl}

\author{Samuel Grushevsky}
\address{Mathematics Department, Stony Brook University, Stony Brook, NY 11790-3651, USA}
\email{sam@math.sunysb.edu}
\thanks{Research of the first author was supported by NWO under grant 
nwo 613.000.901; research of the 
third author is supported in part by National Science Foundation under the grant DMS-12-01369.}


\maketitle
\centerline{(with an appendix by Shigeru Mukai)}

\begin{abstract}
We study vector-valued Siegel modular forms of genus $2$ 
on the three level $2$ groups  $\Gamma[2]  \lhd \Gamma_1[2]  \lhd \Gamma_0[2]\subset {\rm Sp}(4,{\ZZ})$.
We give generating functions 
for the dimension of spaces of vector-valued modular forms,
construct various vector-valued modular forms by using theta functions and describe
the structure of certain modules of vector-valued modular forms 
over rings of scalar-valued Siegel modular forms.
\end{abstract}

\begin{section}{Introduction}
Vector-valued Siegel modular forms are the natural generalization of
elliptic modular forms and in recent years there has been an increasing 
interest in these modular forms. One of the attractive aspects of the
theory of elliptic modular forms is the presence of easily accessible examples.
By contrast easily accessible examples in theory of vector-valued 
Siegel modular forms have been very few.  Vector-valued Siegel modular
forms of genus $2$ and level $1$ have been considered by Satoh, Ibukiyama 
and others, cf.\ \cite{Aoki,Satoh, Ibukiyama3, Ibukiyama4, vanDorp}.

The study of local systems and point 
counting of curves over finite fields has made it possible to calculate
Hecke eigenvalues for eigenforms of the Hecke algebra, first for 
vector-valued forms of genus $2$ and level~$1$, later under some assumptions 
also for genus~$2$ and level~$2$ and even for genus~$3$ and level~$1$, see 
\cite{FvdG, BFvdG, BFvdG2}. These methods do not require nor provide
an explicit description of these modular forms. Describing explicitly these
modular forms and the generators for the modules of such modular forms 
is thus a natural question.

The focus of this paper is genus $2$ and level $2$:
more precisely, we will study vector-valued modular forms on the
full congruence subgroup $\Gamma[2]$ of ${\rm Sp}(2,{\ZZ})$ of level $2$ together with
the action of $\frak{S}_6\cong {\rm Sp}(2,{\ZZ}/2{\ZZ})$ on these.
This will lead to a wealth of results on modular forms on the congruence
subgroups $\Gamma_0[2]$ and $\Gamma_1[2]$ too.
We will construct many such modular forms by taking Rankin-Cohen brackets 
of polynomials in theta constants with even characteristics, 
and by using gradients of theta functions with odd characteristics. 
We will furthermore describe some modules of vector-valued modular forms. 
One major tool is studying the representations of $\frak{S}_6$, 
the Galois group of the level $2$ cover of the moduli space of 
principally polarized abelian surfaces, on the spaces of modular forms. The 
methods of \cite{BFvdG}  allow one to compute these actions 
assuming the conjectures made in \cite{BFvdG} --- 
and these give a heuristic tool to detect where one has to search for 
modular forms or relations among them. We apply these to get bounds 
on the weights of generators and relations of the modules of 
vector-valued forms --- but note that our final results on the module
structure are not conditional on the conjectures of \cite{BFvdG}.

More precisely, our results are as follows. 
In Theorems \ref{scalarformsonGamma0} and \ref{RingGamma12}  
we compute the rings of 
scalar-valued modular forms on $\Gamma_1[2]$ and $\Gamma_0[2]$. 
This computation uses Igusa's determination of the ring of 
scalar-valued modular forms on $\Gamma[2]$, and the result for 
$\Gamma_0[2]$ was already known by Ibukiyama \cite{AokiIbukiyama}. 
By analyzing the action of $\frak{S}_6$ on the spaces of 
vector-valued modular forms on $\Gamma[2]$, in theorem \ref{dimGamma12}
we give the generating functions for the dimensions of the spaces 
$M_{j,k}(\Gamma_1[2])$ of modular forms on $\Gamma_1[2]$. 
These results are based on Wakatsuki's \cite{wakatsuki} 
 computation of the generating functions for $M_{j,k}(\Gamma[2])$, 
and their derivation uses the conjectures made in \cite{BFvdG} --- 
but the result fits all available data, e.g.\ 
Tsushima's calculations (cf.\ references in \cite{BFvdG}). 
In Sections 15--18 we construct vector-valued modular forms in two ways: 
using a variant of the Rankin-Cohen bracket applied to even
theta constants and by using gradients of odd theta functions 
multiplied by suitable even theta constants in order to get 
modular forms of the desired level. Using these results and 
suitable Castelnuovo-Mumford regularity established in Section 
\ref{ModuleStructure}, 
in Theorem \ref{ThmSigma2} 
we determine the generators for the 
module $\Sigma_2= \oplus_{k,\, \rm odd} S_{2,k}(\Gamma[2])$ of cusp forms 
of `weight' ${\rm Sym}^2 \otimes \det^k$. 
In theorems \ref{ThmSigma1Gamma1}   and \ref{ThmM4}
we determine the generators for the modules 
${\mathcal M}_j^\epsilon= \oplus_{k, k\equiv \epsilon\mod 2} M_{j,k}(\Gamma[2])$ for $\epsilon =0,1$ and $j=2,4$. 
In some cases we also determine the submodule of relations.

We conclude the paper by constructing 
an explicit generator 
for many cases where the space of cusp forms 
$S_{j,k}({\rm Sp}(4,{\ZZ}))$ is
$1$-dimensional and by giving the Fourier coefficients of the module 
generators for
$\Sigma_2^1=\oplus_{k, \, \rm odd} S_{2,k}(\Gamma[2])$ and of certain
generators of a module of modular forms of weight $(4,*)$.

The fact that we have two different ways of constructing vector-valued modular forms
naturally leads to many identities between modular forms, 
some of them quite pretty. We have restricted ourselves to just giving a few samples,
inviting the reader to find many more. 

\begin{remark} 
One intriguing feature of the situation is as follows. Mukai \cite{Mukai} 
recently showed that the Satake compactification of the moduli space of principally polarized abelian
surfaces with a $\Gamma_1[2]$-level structure is given by the 
Igusa quartic --- which by the results of Igusa is the Satake 
compactification of the moduli space of principally polarized abelian surfaces with a full level 2 structure. 
We will see how this remarkable fact is reflected in 
the structure of rings and modules of scalar-valued and 
vector-valued modular forms on $\Gamma[2]$ and $\Gamma_1[2]$.
In an appendix to this paper Mukai makes a very minor correction 
to a statement about the Fricke involution in \cite{Mukai}  to guarantee
the peaceful coexistence of his paper with the present one.
\end{remark}
\begin{remark} 
Another interesting feature is that the  modules of vector-valued 
modular forms that we consider are not of finite presentation 
over the ring of scalar-valued modular forms. Indeed, recall that 
the ring of even weight scalar-valued modular forms on $\Gamma[2]$ 
is a quotient of a polynomial ring in five variables by a principal ideal --- and the modules of vector-valued modular forms like 
$\oplus_k M_{j,k}(\Gamma[2])$ are of finite presentation 
only over this polynomial ring.
\end{remark}
\noindent
{\sl Acknowledgements}
The first two authors thank NWO for support. The first author thanks the Riemann Center and the Institute of Algebraic Geometry of the Leibniz University Hannover
for their support. The second author thanks the Mathematical Sciences Center of Tsinghua University
for hospitality enjoyed there. The second author also thanks Jonas Bergstr\"om
for useful comments. This paper builds on the work done with Jonas Bergstr\"om 
and Carel Faber in \cite{BFvdG}.
\end{section}
\begin{section}{Preliminaries}\label{Preliminaries}
Let $\Gamma={\rm Sp}(4,{\ZZ})$ be the Siegel modular group.
The following level $2$ congruence subgroups 
$\Gamma[2]  \lhd \Gamma_1[2]  \lhd \Gamma_0[2]\subset \Gamma$
defined by
$$
\Gamma[2]=\{ 
M\in \Gamma:  M\equiv 1_4   \bmod 2 \} \, ,
\qquad
\Gamma_1[2]= \{ M\in \Gamma: 
 M\equiv \left(\begin{matrix} 1_2 & * \\ 0 & 1_2\end{matrix} \right) \bmod 2\}
$$
and 
$$
\Gamma_0[2]= \{
M\in \Gamma: M\equiv \left(\begin{matrix} * & * \\ 0 & *
\end{matrix} \right) \bmod 2
\}
$$
will play a central role here. 

The successive quotients can be identified
as follows
$$\Gamma_1[2]/ \Gamma[2]  \simeq (\ZZ/2\ZZ)^3, \quad
\Gamma_0[2]/ \Gamma[2]  \simeq \ZZ/2\ZZ \times \mathfrak{S}_4, \quad
\Gamma_0[2]/ \Gamma_1[2]  \simeq \mathfrak{S}_3, \quad
\Gamma/ \Gamma[2]  \simeq \mathfrak{S}_6\, ,
$$
with $\mathfrak{S}_n$ the symmetric group on $n$ letters; see Section \ref{ThetaChar} for an explicit identification.

These groups act on the Siegel upper half space 
$$
\frak{H}_2=\{ \tau=\left(\begin{smallmatrix} 
\tau_{11} & \tau_{12} \\ \tau_{12}& \tau_{22}\end{smallmatrix} \right) 
\in {\rm Mat}(2\times 2, {\CC}): 
\tau^t=\tau, {\rm Im}(\tau) > 0 \}
$$
in the usual way 
($\tau \mapsto M\langle \tau \rangle =(a\tau+b)(c\tau+d)^{-1}$)
and the quotient orbifolds of the action of 
$\Gamma, \Gamma_0[2], \Gamma_1[2]$ and $\Gamma[2]$ will be denoted by 
${\mathcal A}_2$, ${\mathcal A}_2[\Gamma_0[2]]$,
${\mathcal A}_2[\Gamma_1[2]]$ and ${\mathcal A}_2[\Gamma[2]]$. 
We have a diagram of coverings

\begin{displaymath}
\begin{xy}
\xymatrix{
{\mathcal A}_2[\Gamma[2]] \ar[rr] 
\ar@/^2.5pc/ @{-}[rrrr]^{\frak{S}_6}
\ar@/^1pc/ @{-}[rrr]^{{\ZZ}/2{\ZZ} \times \mathfrak{S}_4}
\ar@/_1pc/ @{-}[rr]_{({\ZZ}/2{\ZZ})^3}
&&{\mathcal A}_2[\Gamma_1[2]] \ar[r]  \ar@/_1pc/ @{-}[r]_{\mathfrak{S}_3}
&{\mathcal A}_2[\Gamma_0[2]] \ar[r] 
&{\mathcal A}_2 \\
}
\end{xy}
\end{displaymath}

Recall that we have a so-called Fricke involution induced by the element
\begin{equation}\label{Fricke}
\left( \begin{matrix} 0 & 1_2/\sqrt{2} \\ -\sqrt{2}\, 1_2 & 0 \\ \end{matrix} \right)
\end{equation}
of ${\rm Sp}(4,{\RR})$ 
that normalizes $\Gamma_1[2]$ and $\Gamma_0[2]$ and thus 
induces an involution $W_2$ on
${\mathcal A}_2[\Gamma_1[2]]$ and ${\mathcal A}_{2}[\Gamma_0[2]]$.

These quotients admit a Satake (or Baily-Borel) compactification
obtained by adding $1$-dimensional and $0$-dimensional boundary components.

The Satake compactification ${\mathcal A}_2[\Gamma[2]]^*$ of ${\mathcal A}_2[\Gamma[2]]$ is obtained
by adding fifteen $1$-dimensional boundary components each isomorphic to
${\mathcal A}_1[2]=\Gamma(2)\backslash \frak{H}_1$, where $\Gamma(2)$
denotes the principal congruence subgroup\footnote{We denote the congruence 
subgroups of ${\rm SL}(2,{\ZZ})$ by round brackets, those of ${\rm Sp}(4,{\ZZ})$ by square brackets}   
of level $2$ of ${\rm SL}(2,{\ZZ})$
and fifteen points forming a $(15_3,15_3)$-configuration.
The group $\mathfrak{S}_6=\Gamma/\Gamma[2]$ acts on it.
One can assign to each $1$-dimensional boundary component a pair $\{ i,j\} 
\subset \{1,2,\ldots,6\}$ with $i\neq j$ such that any 
$\sigma \in \mathfrak{S}_6$
sends the component $B_{ij}$
corresponding to $\{ i,j\}$ to $B_{ \sigma(i) \sigma(j)}$;
similarly
one can assign to each $0$-dimensional cusp a 
partition $(ij)(kl)(mn)$ of $\{i,j,k,l,m,n\}=\{1,2,\ldots,6\}$ into three pairs
on which $\mathfrak{S}_6$ acts in the natural way such that the
cusp given by $(ij)(kl)(mn)$ is a cusp of the boundary 
components $B_{ij}, B_{kl}$ and $B_{mn}$, cf.\ Lemma \ref{thetacharlemma}
and Remark \ref{boundaryvsoddpair} below.
Note that $\Gamma_0[2]$ is the inverse image
of a Siegel parabolic group (fixing a $0$-dimensional boundary component)
under the reduction $\bmod \, 2$ map 
${\rm Sp}(4,{\ZZ})\to {\rm Sp}(4,{\ZZ}/2{\ZZ})$ and $\Gamma_1[2]$ is 
the subgroup fixing each of the three $1$-dimensional boundary 
components passing through this $0$-dimensional cusp.

The Satake compactification of ${\mathcal A}_2[\Gamma_1[2]]$ is obtained
by adding six $1$-dimensional boundary components (each isomorphic to
$\Gamma_0(2)\backslash \frak{H}_1$ and denoted $A,\ldots,F$) 
and five $0$-dimensional 
boundary components (denoted $\alpha,\ldots, \epsilon$) 
as in the following configuration.
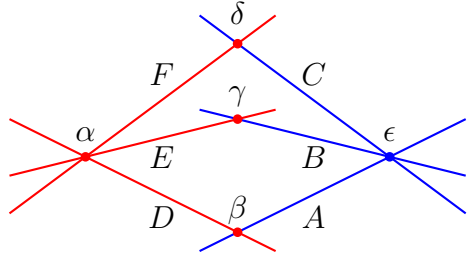
\begin{center}\begin{pspicture}(-0,-2)(4,3)
\psline[linecolor=red](-1,0.5)(2.5,-1.25)
\psline[linecolor=red](-1,-0.75)(2.5,1.875)
\psline[linecolor=red](-1,-0.25)(2.5,0.625)
\psline[linecolor=blue](1.5,-1.25)(5,0.5)
\psline[linecolor=blue](1.5,0.625)(5,-0.25)
\psline[linecolor=blue](1.5,1.875)(5,-0.75)
\psdot[linecolor=blue](4,0)
\psdot[linecolor=red](0,0)
\psdot[linecolor=red](2,0.5)
\psdot[linecolor=red](2,1.5)
\psdot[linecolor=red](2,-1)
\rput(3,-0.8){$A$}
\rput(3,0){$B$}
\rput(3,1.1){$C$}
\rput(1,-0.8){$D$}
\rput(1,0){$E$}
\rput(1,1.1){$F$}
\rput(0,0.3){$\alpha$}
\rput(2,-0.7){$\beta$}
\rput(2,0.8){$\gamma$}
\rput(2,1.9){$\delta$}
\rput(4,0.3){$\epsilon$}
\end{pspicture}
\end{center}
The normal subgroup $\Gamma_1[2]/\Gamma[2]$ of $\Gamma_0[2]/\Gamma[2]$
acts trivially on this configuration and the induced action of the quotient 
$\mathfrak{S}_3$
permutes the $1$-dimensional boundary cusps  $A,B,C$ and $D,E,F$
and permutes the three $0$-dimensional cusps $\beta, \gamma, \delta$
and fixes $\alpha$ and $\epsilon$. The Fricke involution $W_2$
interchanges $\alpha$ and $\epsilon$,  fixes $\gamma$ and interchanges
$\beta$ and $\delta$ as we shall see later (Corollary \ref{W2oncusps}).

The Satake compactification of ${\mathcal A}_2[\Gamma_0[2]]$ is obtained
by adding to ${\mathcal A}_2[\Gamma_0[2]]$ 
two $1$-dimensional boundary components (the images of $D$ and $A$)
each isomorphic to 
$\Gamma_0(2)\backslash \frak{H}_1$ and three $0$-dimensional cusps
(the images of $\alpha, \beta$ and $\epsilon$).
\smallskip

We let $V$ be the standard $2$-dimensional representation space 
of ${\rm GL}(2,{\CC})$
and let $\rho_{j,k}: {\rm GL}(2,{\CC}) \to 
{\rm GL}({\rm Sym}^j(V)\otimes \det(V)^{\otimes k})$ be the irreducible
representation of highest weight $(j+k,k)$.
By a Siegel modular form of weight $(j,k)$ on $\Gamma$ 
(resp.\ $\Gamma_0[2], \Gamma_1[2], \Gamma[2]$) we mean a holomorphic
map $f: \frak{H}_2 \to {\rm Sym}^j(V)\otimes \det(V)^{\otimes k}$
such that 
$$
f(M\langle \tau \rangle)= \rho_{j,k}(c\tau+d) f(\tau)\quad
\hbox{\rm for all $M= \left( \begin{smallmatrix} a & b \\ c & d \\
\end{smallmatrix} \right)  \in \Gamma$ 
(resp.\ $\Gamma_0[2], \Gamma_1[2], \Gamma[2]$).}
$$
We refer to \cite{BGHZ} and the references given there for 
background on Siegel modular forms. 
Let ${\EE}$ be the Hodge bundle on ${\mathcal A}_2$
(or its pull back to ${\mathcal A}_2[\Gamma']$ for $\Gamma'$ 
a finite index subgroup of $\Gamma$). 
It corresponds to the standard representation of 
${\rm GL}(2,{\CC})$. The bundle ${\EE}$ extends to `good' 
toroidal compactifications of ${\mathcal A}_2[\Gamma']$. 
Then scalar-valued modular forms of weight $k$ on $\Gamma'$ can be interpreted as 
sections of $L^{\otimes k}$ with $L=\det({\EE})$ on
${\mathcal A}_2[\Gamma']$. By the well-known Koecher principle such sections
extend automatically to these toroidal compactifications. Similarly,
if ${\EE}_{\rho}={\rm Sym}^j(\EE)\otimes \det(\EE)^{\otimes k}$ 
is the vector bundle on ${\mathcal A}_2[\Gamma']$ corresponding to 
the irreducible representation $\rho=\rho_{j,k}$ 
then modular forms of weight $(j,k)$ on $\Gamma'$ are the sections of this vector bundle
and by the Koecher principle these extend to sections over `good' toroidal
compactifications. Again, we refer to  \cite{BGHZ}  and the references 
given there for more details.

\smallskip
We close this section by explaining our notation for the irreducible representations
of $\frak{S}_6$. The irreducible representations of $\frak{S}_6$ correspond bijectively
to the partitions of $6$. The representation corresponding to the partition $P$ will be denoted
by $s[P]$, with $s[6]$ the trivial one and $s[1,1,1,1,1,1]=s[1^6]$ the alternating representation. Their dimensions are recalled for convenience.

\begin{footnotesize}
\smallskip
\vbox{
\bigskip\centerline{\def\quad{\hskip 0.6em\relax}\def\quod{\hskip 0.5em\relax }
\vbox{\offinterlineskip
\hrule\halign{&\vrule#&\strut\quod\hfil#\quad\cr
height2pt&\omit&&\omit&&\omit&&\omit&&\omit&&\omit&&\omit&&\omit&&\omit&&\omit &
& \omit && \omit &\cr
&$P$ && $[6]$ && $[5,1]$ && $[4,2]$ && $[4,1^2]$ && $[3^2]$ && $[3,2,1]$ && $[3,1^3]$ && $[2^3]$ && $[2^2,1^2]$ && $[2,1^4]$&& $[1^6]$ &\cr
\noalign{\hrule}&$\dim$ && $1$ && $5$ && $9$ && $10$ && $5$ && $16$ && $10$ && $5$ && $9$ && $5$ && $1$ &\cr
} \hrule}
}}
\end{footnotesize}

\end{section}
\begin{section}{Theta Characteristics}\label{ThetaChar}
\bigskip
In this paper a theta characteristic is an element of $\{ 0, 1\}^4$
written as a row vector $(\mu_1,\mu_2,\nu_1,\nu_2)$ or as a $2\times 2$
matrix $\left[ \begin{smallmatrix}
\mu_1 & \mu_2 \\ \nu_1 & \nu_2 \end{smallmatrix} \right]$. It is called even or
odd depending on the parity of $\mu_1\nu_1+\mu_2\nu_2$.

We order the six odd theta characteristics $m_1,\ldots,m_6$
 lexicographically:
$$
\begin{aligned}
m_1=\left[\begin{matrix} 0 & 1 \cr 0 & 1 \cr\end{matrix}\right], \quad
m_2=\left[\begin{matrix} 0 & 1 \cr 1 & 1 \cr\end{matrix}\right], \quad
m_3=\left[\begin{matrix} 1 & 0 \cr 1 & 0 \cr\end{matrix}\right], &\\
m_4=\left[\begin{matrix} 1 & 0 \cr 1 & 1 \cr\end{matrix}\right], \quad
m_5=\left[\begin{matrix} 1 & 1 \cr 0 & 1 \cr\end{matrix}\right], \quad
m_6=\left[\begin{matrix} 1 & 1 \cr 1 & 0 \cr\end{matrix}\right]. &\\
\end{aligned}
$$
Note that the sum $\sum_{i=1}^6 m_i$ is zero $\bmod \, 2$ 
and  each of the ten even 
theta characteristics is a sum of three different
odd theta characteristics in two ways; e.g.,
$$
n_1=\left[\begin{matrix} 0 & 0 \cr 0 & 0 \cr\end{matrix}\right]=
m_1+m_4+m_6=m_2+m_3+m_5\, .
$$
In this way each even theta characteristic is associated to a partition
of $\{1,2,3,4,5,6\}$ in two triples.
We use the following (lexicographic)
ordering for the ten even theta characteristics
$$
\begin{aligned}
& n_1=\left[\begin{matrix} 0 & 0 \cr 0 & 0 \cr\end{matrix}\right],
n_2=\left[\begin{matrix} 0 & 0 \cr 0 & 1 \cr\end{matrix}\right],
n_3=\left[\begin{matrix} 0 & 0 \cr 1 & 0 \cr\end{matrix}\right],
n_4=\left[\begin{matrix} 0 & 0 \cr 1 & 1 \cr\end{matrix}\right],
n_5=\left[\begin{matrix} 0 & 1 \cr 0 & 0 \cr\end{matrix}\right], \\
& n_6=\left[\begin{matrix} 0 & 1 \cr 1 & 0 \cr\end{matrix}\right],
n_7=\left[\begin{matrix} 1 & 0 \cr 0 & 0 \cr\end{matrix}\right],
n_8=\left[\begin{matrix} 1 & 0 \cr 0 & 1 \cr\end{matrix}\right],
n_9=\left[\begin{matrix} 1 & 1 \cr 0 & 0 \cr\end{matrix}\right],
n_{10}=\left[\begin{matrix} 1 & 1 \cr 1 & 1 \cr\end{matrix}\right]. \\
\end{aligned}
$$
For the ease of the reader we give the correspondence between the even
$n_i$ and triples of odd ones.
$$
\begin{matrix} 
n_1 & (146)(235) && n_6 & (156)(234) \\
n_2 & (136)(245) && n_7 & (123)(456) \\
n_3 & (135)(246) && n_8 & (124)(356) \\
n_4 & (145)(236) && n_9 & (126)(345) \\
n_5 & (134)(256) && n_{10} & (125)(346) \\
\end{matrix}
$$
\begin{lemma}\label{thetacharlemma}
i) An unordered pair $\{ m_i,m_j\}$ of different odd theta characteristics 
determines uniquely an unordered  quadruple of even theta characteristics, 
namely the $n_k$ corresponding to the four ways of writing
$n_k=m_i+m_j+a=b+c+d$ with $\{m_1,\ldots,m_6\}=\{m_i,m_j,a,b,c,d\}$.
ii) A partition of the set of odd theta characteristics 
$ \{m_{i_1},m_{i_2}\} \sqcup \{ m_{i_3},m_{i_4}\} \sqcup \{ m_{i_5},m_{i_6}\}$ 
in three pairs determines uniquely a quadruple of
even theta characteristics 
such that $n=a+b+c$ with $a \in \{m_{i_1},m_{i_2}\}$, $b\in \{m_{i_3},m_{i_3}\}$
and $c \in \{m_{i_5},m_{i_6}\}$.  
\end{lemma}
For example $\{ m_1,m_2\}$ corresponds to $\{ n_7,n_8,n_9,n_{10}\}$ and
$\{ m_1,m_2\}\sqcup \{ m_3,m_4\} \sqcup \{m_5,m_6\}$ corresponds to
$\{ n_1,n_2,n_3,n_4\}$.
\smallskip

An element $M=\left(\begin{smallmatrix} A & B \\ C & D \\ \end{smallmatrix}
\right)$
of $\Gamma$ acts on ${\ZZ}^4$ by
\begin{equation}\label{actonchar}
M\cdot \left[\begin{smallmatrix} \mu_1 \\ \mu_2 \\ \nu_1 \\ \nu_2 \\
 \end{smallmatrix}\right]=
\left(\begin{smallmatrix} D & -C\\ -B & A \end{smallmatrix}\right)
\left(\begin{smallmatrix} \mu_1 \\ \mu_2 \\ \nu_1 \\ \nu_2 \\ \end{smallmatrix}\right)
+\left(\begin{smallmatrix}(CD^t)_0 \\ (AB^t)_0 \end{smallmatrix}\right), 
\end{equation}
where for a matrix $X$ the symbol $X_{0}$ 
denotes the diagonal vector (in its natural order).
The quotient group $\Gamma/\Gamma[2]\cong {\rm Sp}(4, {\ZZ}/2{\ZZ})$ 
is identified with the symmetric group $\mathfrak{S}_6$ via its action on the six 
odd theta characteristics. 
Recall that the group $\mathfrak{S}_6$ is generated by the two 
elements $(12)$ and $(123456)$ represented by elements of $\Gamma$
\begin{equation}\label{XandY}
X=
\left(\begin{matrix}
1 & 0 & 1 & 0 \\
0 & 1 & 0 & 0 \\
0 & 0 & 1 & 0\\
0 & 0 & 0 & 1
\end{matrix}\right)
\quad
{\rm and}
\quad
Y=
\left(\begin{matrix}
0 & 1 & 0 & 1 \\
1 & 0 & 1 & 0 \\
1 & 0 & 1 & 1\\
-1 & 1 & 0 & 1
\end{matrix}\right) \, .
\end{equation}

The partition of the six odd theta characteristics into three pairs
defines a conjugacy class of $\Gamma_0[2]$:
let $C\cong \mathfrak{S}_3 \ltimes ({\ZZ}/2{\ZZ})^3$ be the subgroup of $\mathfrak{S}_6$ that 
stabilizes the partition $\{m_1,m_2\} \sqcup \{m_3,m_4\} \sqcup\{m_5,m_6\}$. 
Then the inverse image of $C$ under the quotient map 
$\Gamma \to \Gamma/\Gamma[2]$ equals  $\Gamma_0[2]$. 

Since by Lemma \ref{thetacharlemma} this partition of the six odd theta characteristics in three disjoint
pairs defines a quadruple of even ones, the group 
$\Gamma_0[2]/\Gamma[2]\cong C$ acts on this set 
$\{n_1,n_2, n_3, n_4 \}$ and this defines a surjective map 
$C \to \mathfrak{S}_4$
with kernel generated by $(12)(34)(56)$ that gives an isomorphism
$C \cong \mathfrak{S}_4 \times {\ZZ}/2{\ZZ}$.

Representatives of the generators of $\Gamma_1[2]/\Gamma[2]=({\ZZ}/2{\ZZ})^3$
are given by the transformations $\tau_{11}\mapsto\tau_{11}+1$, 
$\tau_{22}\mapsto \tau_{22}+1$ and $\tau\mapsto \tau+1_2$ corresponding to
$(12),(34)$ and $(56)$. 
Generators of $\mathfrak{S}_3=\Gamma_0[2]/\Gamma_1[2]$ are given by
\begin{equation}\label{AprimeBprime}
X'=\left(\begin{matrix} A & 0 \\ 0 & A^{-t} \\ \end{matrix} \right), \qquad
Y'=\left(\begin{matrix} B & 0 \\ 0 & B^{-t} \\ \end{matrix} \right)
\end{equation}
with $A=\left(\begin{smallmatrix} 1 & 1 \\ 0 & 1 \\ \end{smallmatrix} \right)$
and $B=\left(\begin{smallmatrix} 0 & 1 \\ 1 & 1 \\ \end{smallmatrix}\right)$.
\end{section}
\begin{section}{Theta Series}
For $(\tau,z) \in \frak{H}_2 \times {\CC}^2$ and  $\left[
\begin{smallmatrix} \mu \\  \nu  \\ \end{smallmatrix} \right]
=\left[ \begin{smallmatrix} \mu_1 & \mu_2 \\ \nu_1 & \nu_2 \\
\end{smallmatrix} \right] $
with $\mu =(\mu_1,\mu_2)$ and $\nu =(\nu_1,\nu_2)$ in ${\ZZ}^2$
we consider the standard theta series with characteristics
$$
\vartheta_{\left[\begin{smallmatrix}\mu  \\ \nu
\end{smallmatrix}\right]}(\tau,z)=
\sum_{n=(n_1,n_2) \in \ZZ^2}
e^{\pi i
\left((n+\mu /2)\left(\tau(n+\mu/2)^t+2(z+\nu/2)^t\right)\right)}.
$$
Usually the $\mu_i,\nu_i$ will be equal to $0$ or $1$; in fact we will be
mainly interested in the theta constants and the formula
$$
\vartheta_{\left[ \begin{smallmatrix} \mu+ 2 m\\
\nu+ 2n \\ \end{smallmatrix} \right]}(\tau,0)= (-1)^{\mu \cdot n^t}
\vartheta_{\left[ \begin{smallmatrix} \mu \\ \nu \\
\end{smallmatrix} \right]}(\tau,0)
$$
allows us to reduce the characteristic modulo $2$.
The transformation behavior of the theta series under $\Gamma$ is known,
cf.\ \cite{Igusa1}.

\begin{lemma}\label{thetatrans}
For $M=\left(\begin{smallmatrix} A & B\\ C & D \\
\end{smallmatrix}\right) \in \Gamma$,
we have the transformation behavior
$$
\begin{aligned}
\vartheta_{M\cdot\left[\begin{smallmatrix}\mu  \\ \nu
\end{smallmatrix}\right]}(M
\left\langle \tau \right\rangle,&(C\tau+D)^{-t}z)=\\
&\kappa(M)\,
e^{2\pi i\phi(\left[\begin{smallmatrix} \mu \\ \nu \end{smallmatrix}\right],M)}
\cdot \det(C\tau+D)^{\frac{1}{2}}
e^{\pi i z(C\tau+D)^{-1}Cz^t}
\vartheta_{\left[
\begin{smallmatrix}\mu \\ \nu
\end{smallmatrix}\right]}
(\tau,z)\, ,
\end{aligned}
$$
where $\phi(\left[\begin{smallmatrix} \mu \\ \nu \end{smallmatrix}\right],M)$
is given by
$$
(2\,  \mu B^tC  \nu^t  +2\, (AB^t)_0(D\mu^t-C\nu^t)
-\mu B^tD\mu^t- \nu A^tC  \nu^t)/8,
$$
and with the action on the characteristics given by (\ref{actonchar}).
Moreover, $\kappa(M)$ is an $8$th root of unity (depending only on $M$
and not on $\mu, \nu$).
\end{lemma}

For the theta constants $\vartheta(\tau)=\vartheta(\tau,0)$
the transformation under
$M=\left(\begin{smallmatrix} A & B\\ C & D \end{smallmatrix}\right)\in
 \Gamma$ reduces to
$$
\vartheta_{M\cdot\left[\begin{smallmatrix}\mu \\ \nu \end{smallmatrix}\right]}(M\left\langle \tau \right\rangle)=
\kappa(M)\,
e^{2\pi i \phi(\left[\begin{smallmatrix}\mu \\ \nu \end{smallmatrix}\right],M)}
\det(C\tau+D)^{\frac{1}{2}}\,
\vartheta_{\left[\begin{smallmatrix}\mu  \\ \nu \end{smallmatrix}\right]}(\tau) \, .
$$
\smallskip
It is convenient to introduce the {\sl slash operator}. For $M\in \Gamma$,
$k$ half-integral  and a function $F$ on $\frak{H}_2$ we put
$$
(F|_{0,k}M)(\tau)=\det(C\tau+D)^{-k} F(M\langle \tau\rangle).
$$
(Here $\sqrt{\det(C\tau+D)}$ is chosen to have positive imaginary part.)
Invariance of $F$ under the slash operator 
expresses the fact that a function transforms 
like a scalar-valued modular form of weight $k$.

The action of the matrices $M=X$ and $M=Y$ (defined in \ref{XandY}) on the (column) vector of the
ten even theta constants by the slash operator
$\vartheta_{n_i} \mapsto \vartheta_{n_i}\vert_{0,\frac{1}{2}}M$ is
given by the unitary matrices
\begin{equation}\label{rhoXandrhoY}
\rho(X)=
\left(\begin{smallmatrix}
0 & 0 & 1 & 0 & 0 & 0 & 0 & 0 & 0 & 0\\
0 & 0 & 0 & 1 & 0 & 0 & 0 & 0 & 0 & 0\\
1 & 0 & 0 & 0 & 0 & 0 & 0 & 0 & 0 & 0\\
0 & 1 & 0 & 0 & 0 & 0 & 0 & 0 & 0 & 0\\
0 & 0 & 0 & 0 & 0 & 1 & 0 & 0 & 0 & 0\\
0 & 0 & 0 & 0 & 1 & 0 & 0 & 0 & 0 & 0\\
0 & 0 & 0 & 0 & 0 & 0 &\zeta & 0 & 0 & 0\\
0 & 0 & 0 & 0 & 0 & 0 & 0 & \zeta & 0 & 0\\
0 & 0 & 0 & 0 & 0 & 0 & 0 & 0 &  \zeta & 0\\
0 & 0 & 0 & 0 & 0 & 0 & 0 & 0 & 0 &  \zeta\\
\end{smallmatrix}\right)
\quad {\rm and} \quad
\rho(Y)=
\left(\begin{smallmatrix}
0 & 0 & 0 & 0 & 0 & 0 & 0 & \zeta^{7} & 0 & 0\\
0 & 0 & 0 & 0 & 1 & 0 & 0 & 0 & 0 & 0\\
0 & 0 & \zeta^{6} & 0 & 0 & 0 & 0 & 0 & 0 & 0\\
0 & 0 & 0 & 0 & 0 & 0 & 0 & 0 & 0 & \zeta^{7}\\
0 & 0 & 0 & \zeta^{6} & 0 & 0 & 0 & 0 & 0 & 0\\
0 & 0 & 0 & 0 & 0 & 0 & \zeta^{7} & 0 & 0 & 0\\
0 & 0 & 0 & 0 & 0 & 0 & 0 & 0 & 1 & 0\\
0 & \zeta^{7} & 0 & 0 & 0 & 0 & 0 &  0 & 0 & 0\\
0 & 0 & 0 & 0 & 0 & \zeta^{7} & 0 & 0 & 0 & 0\\
\zeta^5 & 0 & 0 & 0 & 0 & 0 & 0 & 0 & 0 &  0\\
\end{smallmatrix}\right)
\end{equation}
with $\zeta= e^{\pi i /4}$.

By formula \ref{actonchar} for the action of $M$ on the set of characteristics 
it follows that $M$ 
acts trivially on the set $\{0,1\}^4$ of characteristics 
if and only if $M\in\Gamma[2]$. 
Recall the (Igusa) theta groups
$$
 \Gamma[n,2n]:=\lbrace M\in\Gamma: M\equiv 1_4\bmod n, (AB^t)_0\equiv (CD^t)_0\equiv 0 \bmod 2n\rbrace.
$$
It turns out \cite{Igusa1} that theta constants are scalar-valued 
modular forms 
(with a multiplier) only on the subgroup $\Gamma[4,8]$.
The transformation formula for theta constants implies that the squares of theta constants are scalar-valued modular forms of weight $1$
 on $\Gamma[2,4]$, while the fourth powers of theta constants are 
modular forms of weight $2$ on $\Gamma[2]$. 
In fact, it is known, see \cite{Igusa1}, that the ring of scalar-valued modular forms of integral weight on $\Gamma[2,4]$ 
is generated by squares of theta constants, while the ring of scalar-valued 
modular forms of even weight on $\Gamma[2]$ 
is generated by the fourth powers of theta constants. The squares of the theta constants and fourth powers of theta constants 
satisfy many polynomial relations, which we will describe explicitly below
for genus $2$. 
All these polynomials identities follow from Riemann's bilinear relation, which we now recall.

We define the theta functions of the second order to be
$$
 \Theta[\mu](\tau,z)=\vartheta_{\left[\begin{smallmatrix}\mu\\ 0\end{smallmatrix}\right]}(2\tau,2z),
$$
and call their evaluations at $z=0$ theta constants of  the second order. These are modular forms of weight $1/2$ on
$\Gamma[2,4]$, and generate the ring of scalar-valued modular forms of half-integral weight on  $\Gamma[2,4]$. 
In particular, the squares of theta constants (with characteristics) are expressible in terms of theta constants of the 
second order by using Riemann's bilinear relation
\begin{equation}\label{bilinear}
 \vartheta_{\left[\begin{smallmatrix}\mu  \\ \nu\end{smallmatrix}\right]}^2(\tau,z)=
 \sum\limits_{\sigma\in (\ZZ/2\ZZ)^2} (-1)^{\sigma\cdot\nu}\Theta[\sigma](\tau)\Theta[\sigma+\mu](\tau,z),
\end{equation}
evaluated at $z=0$.
Moreover, it is known that (in genus $2$) 
theta constants of the second order are algebraically 
independent, and determine a birational morphism of the Satake compactification of 
$\Gamma[2,4]\backslash \frak{H}_2$ onto $\PP^3$.
Thus the squares of theta constants of the second 
order are simply the coordinates on $\PP^9$ restricted to the Veronese 
image of $\PP^3\to\PP^9$ given by the Riemann bilinear relations, and as 
such satisfy polynomial relations given by  
Igusa \cite[pp.\ 393, 396]{Igusa2}, which will be described 
explicitly in the next section, where we also 
 explicitly write down the action of 
$\Gamma[2]/\Gamma[2,4]=(\ZZ/2\ZZ)^4$ on the squares of theta constants.

\end{section}
\begin{section}{The Squares of the Theta Constants}\label{ThetaSquares}
To construct modular forms we shall use the squares and the
fourth powers of the ten even theta constants. Therefore
we summarize the behavior of the squares of the theta constants
under $\Gamma[2]$, cf.\ \cite{Igusa1,SM}.  
From the transformation formula of Lemma \ref{thetatrans}
we obtain
$$
\begin{aligned}
(\vartheta_{n_j}^2\vert_{0,1}\, M)(\tau)&=
\det(C\tau+D)^{-1}\vartheta_{n_j}^2(M\left\langle \tau \right\rangle)\\
&=(-1)^{{\rm{Tr}}(D-I_2)/2}\,
e^{4\pi i \phi(n_j,M)}\, \vartheta_{n_j}^2(\tau).
\end{aligned}
$$
Here we have to compute  the expression $4\,\phi(n_j,M)$
modulo $2$ in order to get the transformation formula.
Letting $M=\left(\begin{smallmatrix} A & B\\ C & D \end{smallmatrix}\right)\in \Gamma[2]$ and thus
$B=2\left(\begin{smallmatrix} b_1 & b_2\\ b_3 & b_4 \end{smallmatrix}\right)$
and
$C=2\left(\begin{smallmatrix} c_1 & c_2\\ c_3 & c_4 \end{smallmatrix}\right)$,
we get for $4\,\phi(\left[\begin{smallmatrix} \mu \\
 \nu \end{smallmatrix}\right],M) $
the expression
$$
\mu_1b_1+\mu_2b_4+\nu_1c_1+\nu_2 c_4+\mu_1\mu_2(b_2+b_3)+\nu_1 \nu_2 
(c_2+c_3) \,  \bmod 2
$$
and the fact that $M\in \Gamma[2]\subset \Gamma$ implies
$c_2+c_3\equiv 0 \bmod 2$ and $b_2+b_3\equiv 0 \bmod 2,$
so
$$
4 \, \phi(\left[\begin{smallmatrix} \mu \\ \nu \end{smallmatrix} \right],M)
\equiv
\mu_1b_1+\mu_2b_4+\nu_1 c_1+\nu_2 c_4 \bmod 2,
$$
and writing
$$
e^{4\,  \pi i \phi(n_j,M)}= (-1)^{\alpha(n_j,M)},
$$
we see that
the $\alpha(n_j,M)$ are given for $j=1,\ldots,10$ by the following table

\begin{footnotesize}
\smallskip
\vbox{
\bigskip\centerline{\def\quad{\hskip 0.6em\relax}
\def\quod{\hskip 0.5em\relax }
\vbox{\offinterlineskip
\hrule
\halign{&\vrule#&\strut\quod\hfil#\quad\cr
height2pt&\omit&&\omit&&\omit&&\omit&&\omit&&\omit&&\omit&&\omit&&\omit && \omit && \omit &\cr
&$j$ && $1$ && $2$ && $3$ && $4$ && $5$ && $6$ && $7$ && $8$ && $9$ && $10$&\cr
\noalign{\hrule}
&$\alpha$ && $0$ && $c_4$ && $c_1$ && $c_1+c_4$ && $b_4$ && $b_4+c_1$ && $b_1$ && $b_1+c_4$ && $b_1+b_4$ && $b_1+b_4+c_1+c_4$ &\cr
} \hrule}
}}
\end{footnotesize}

\bigskip
The squares of the theta constants satisfy many quadratic relations.
A pair of
odd theta characteristics $\{ m_{j_1}, m_{j_2}\} $ determines six
even theta characteristics $n_{i}$, namely the six complementary to the four
given by Lemma \ref{thetacharlemma}.
These come in pairs 
such that the sum of $\alpha$ is the same for each pair, see the table
above. For example, $m_1$ and $m_2$ determine the three pairs $(n_1,n_3)$,
$(n_2,n_4)$ and $(n_5,n_6)$ (that give $c_1 \bmod 2$). This gives the relation
\begin{equation}\label{thetasquarerelation}
\vartheta_{1}^2 \vartheta_{3}^2 -
\vartheta_{2}^2\vartheta_{4}^2-\vartheta_{5}^2\vartheta_{6}^2=0, 
\end{equation}
where we write $\vartheta_i$ for $\vartheta_{n_i}$.
These relations form an orbit under the action of $\mathfrak{S}_6$.

\end{section}
\begin{section}{The Ring of Scalar-valued Modular Forms on $\Gamma[2]$}\label{scalarformsonGamma2}
We review the structure of the ring $\oplus_k M_{0,k}(\Gamma[2])$ 
of scalar-valued modular forms on $\Gamma[2]$. We have graded rings
$$
  R= \oplus_k M_{0,k}(\Gamma[2]) \quad 
{\rm and} \quad R^{\rm ev}=\oplus_k M_{0,2k}(\Gamma[2])\, .
$$
The group $\mathfrak{S}_6={\rm Sp}(4,{\ZZ}/2{\ZZ})$ 
acts on $R$ and $R^{\rm ev}$.
The structure of these rings was determined by Igusa, cf.\ \cite{Igusa2, IgusaZ}. 
The ring $R^{\rm ev}$
is generated by the fourth powers of the ten even theta characteristics. 
We shall use the following notation.

\begin{notation}
We denote $\vartheta_{n_i}$ by $\vartheta_{i}$ and $\vartheta_{n_i}^4$ by $x_i$
for $i=1,\ldots, 10$.
\end{notation}
Each $x_i$ is a modular form of weight $2$ on $\Gamma[2]$. 
Formally the ten elements 
$x_i$ span a ten-dimensional representation of $\frak{S}_6$.
The matrices
$\rho(X)$ and $\rho(Y)$ given in \ref{rhoXandrhoY}
imply that the $\mathfrak{S}_6$-representation is $s[2^3]+s[2,1^4]$.
However, the forms $x_i$  are not linearly independent,
but generate the vector space $M_{0,2}(\Gamma[2])$ of dimension
$5$; in fact these satisfy relations like
$$
\vartheta_{1}^4-\vartheta_{4}^4-\vartheta_{6}^4-\vartheta_{7}^4=0
$$
and these relations form a representation $s[2,1^4]$ of $\mathfrak{S}_6$.
The four $n_j$ occurring in such a relation correspond to a pair of 
odd theta characteristics; this gives fifteen such relations, see Lemma 
\ref{thetacharlemma}.
So $M_{0,2}(\Gamma[2])$ equals $s[2^3]$ as a representation space and
$x_i$ for $i=1,\ldots,5$ form a basis.
The $x_i$ define a morphism
$$
\varphi: {\mathcal A}_2[\Gamma[2]] \longrightarrow {\PP}^4 \subset {\PP}^{9}
$$
that extends to an embedding of the Satake compactification
${\mathcal A}_2[\Gamma[2]]^*$ into projective space ${\PP}^4\subset {\PP}^9$.
The ${\PP}^4 \subset {\PP}^9$ is given by the linear relations satisfied
by the $x_i$, a basis of which can be given by
\begin{equation}\label{linearrelationxi}
\begin{aligned}
&x_6=x_1-x_2+x_3-x_4-x_5, \, x_7=x_2-x_3+x_5, \\
&x_8=x_1-x_4-x_5,\, x_9=-x_3+x_4+x_5, \,  x_{10}=x_1-x_2-x_5. \\
\end{aligned}
\end{equation}
The closure of the image of $\varphi$ is then the quartic threefold (the Igusa quartic)
within this linear subspace given by the equation
\begin{equation}\label{Igusaquartic1}
(\sum_{i=1}^{10} x_i^2)^2 -4\, \sum_{i=1}^{10} x_i^4 =0. 
\end{equation}
It follows that $R^{\rm ev}$ is generated by five $4$th powers of the
theta constants
$u_0=x_1, u_1=x_2,\ldots, u_4=x_5$ and that
$$
R^{\rm ev} \cong {\CC}[u_0,\ldots,u_4]/(f)
$$
with $f$ a homogeneous polynomial of degree $4$ in the $u_i$. The full ring
$R$ is a degree $2$ extension $R^{\rm ev}[\chi_5]/(\chi_5^2+2^{14}\chi_{10})$
generated by the modular form $\chi_5$ of weight $5$
$$
\chi_5= \prod_{i=1}^{10} \vartheta_{i}\, .
$$
This form is anti-invariant under $\mathfrak{S}_6$ {} 
(i.e.\ it generates the sign representation $s[1^6]$ of $\mathfrak{S}_6$)
and so its square is a form of level $1$ and satisfies the equation
$\chi_5^2= - 2^{14}\chi_{10}$,
where $\chi_{10}$ is Igusa's cusp form of weight $10$
and level $1$, cf.\ \cite{IgusaZ}.

As a virtual representation of $\mathfrak{S}_6$ we thus have for even $k\geq 0$
$$
M_{0,k}(\Gamma[2]) = {\rm Sym}^{k/2} s[2^3] - \begin{cases}
0 & 0\leq k \leq 6 \\
{\rm Sym}^{k/2-4} s[2^3] & k \geq 8 \\
\end{cases}
$$
and
$M_{0,k+5}(\Gamma[2])= s[1^6] \otimes M_{0,k}(\Gamma[2])$ for even $k\geq 0$.
Igusa calculated the character of $\frak{S}_6$ on the spaces 
$M_{0,k}(\Gamma[2])$, see \cite{Igusa2}, pp.\ 399--402. From his results we can deduce generating
functions 
$ \sum_{k \geq 0} m_{s[P],k} \, t^k $
for the multiplicities $m_{s[P],k}$ of the irreducible representations 
$s[P]$ (with $P$ a partition of $6$) 
of $\frak{S}_6$ in $M_{0,k}(\Gamma[2])$. We give the result in a table.

\begin{footnotesize}
\smallskip
\vbox{
\bigskip\centerline{\def\quad{\hskip 0.6em\relax}\def\quod{\hskip 0.5em\relax }
\vbox{\offinterlineskip
\hrule\halign{&\vrule#&\strut\quod\hfil#\quad\cr
height2pt&\omit&&\omit&&\omit&&\omit &\cr
&$s[6]$ && $\frac{1+t^{35}}{(1-t^4)(1-t^6)(1-t^{10})(1-t^{12})}$ &&
$s[1^6]$ && $\frac{t^5(1+t^{25})}{(1-t^4)(1-t^6)(1-t^{10})(1-t^{12})}$ &\cr
\noalign{\hrule}
height2pt&\omit&&\omit&&\omit&&\omit &\cr
&$s[5,1]$ && $\frac{t^{11}(1+t)}{((1-t^4)(1-t^6))^2}$ &&
$s[2,1^4]$ && $\frac{t^6(1+t^{11})}{((1-t^4)(1-t^6))^2}$ & \cr
\noalign{\hrule}
height2pt&\omit&&\omit&&\omit&&\omit &\cr
& $s[4,2]$ && $\frac{t^4(1+t^{15})}{(1-t^2)(1-t^4)^2(1-t^{10})}$ &&
$s[2,1^2]$ && $\frac{t^9}{(1-t^2)(1-t^4)^2(1-t^5)}$ & \cr
\noalign{\hrule}
height2pt&\omit&&\omit&&\omit&&\omit &\cr
& $s[4,1^2]$ && $\frac{t^{11}(1+t^4)}{(1-t)(1-t^4)(1-t^6)(1-t^{12})}$ &&
$s[3,1^3]$ && $\frac{t^6(1+t^4+t^{11}+t^{15})}{(1-t^2)(1-t^4)(1-t^6)(1-t^{12})}$& \cr
\noalign{\hrule}
height2pt&\omit&&\omit&&\omit&&\omit &\cr
&$s[3,3]$ && $\frac{t^7(1+t^{13})}{(1-t^2)(1-t^4)(1-t^6)(1-t^{12})}$ &&
$s[2^3]$ && $\frac{t^2(1+t^{23})}{(1-t^2)(1-t^4)(1-t^6)(1-t^{12})}$ & \cr
\noalign{\hrule}
height2pt&\omit&&\omit&&\omit&&\omit &\cr
& $s[3,2,1]$ && $\frac{t^8(1-t^8)}{(1-t^2)^2(1-t^5)(1-t^6)^2}$ && &&& \cr
height2pt&\omit&&\omit&&\omit&&\omit &\cr
}  \hrule}
}}
\end{footnotesize}

For the convenience of the reader we give 
the representation type of $M_{0,k}(\Gamma[2])$ for even
$k$ with $2\leq k \leq 12$.

\begin{footnotesize}
\smallskip
\vbox{
\bigskip\centerline{\def\quad{\hskip 0.6em\relax}\def\quod{\hskip 0.5em\relax }
\vbox{\offinterlineskip
\hrule\halign{&\vrule#&\strut\quod\hfil#\quad\cr
height2pt&\omit&&\omit&&\omit&&\omit&&\omit&&\omit&&\omit&&\omit&&\omit&&\omit &
& \omit && \omit &\cr
&$k\backslash P$ && $[6]$ && $[5,1]$ && $[4,2]$ && $[4,1^2]$ && $[3^2]$ && $[3,2,1]$ && $[3,1^3]$ && $[2^3]$ && $[2^2,1^2]$ && $[2,1^4]$&& $[1^6]$ &\cr
\noalign{\hrule}&$2$ && $0$ && $0$ && $0$ && $0$ && $0$ && $0$ && $0$ && $1$ && $0$ && $0$ && $0$ &\cr
\noalign{\hrule}&$4$ && $1$ && $0$ && $1$ && $0$ && $0$ && $0$ && $0$ && $1$ && $0$ && $0$ && $0$ &\cr
\noalign{\hrule}&$6$ && $1$ && $0$ && $1$ && $0$ && $0$ && $0$ && $1$ && $2$ && $0$ && $1$ && $0$ &\cr
\noalign{\hrule}&$8$ && $1$ && $0$ && $3$ && $0$ && $0$ && $1$ && $1$ && $3$ && $0$ && $0$ && $0$ &\cr
\noalign{\hrule}&$10$ && $2$ && $0$ && $3$ && $0$ && $0$ && $2$ && $3$ && $4$ && $0$ && $2$ && $0$ &\cr
\noalign{\hrule}&$12$ && $3$ && $1$ && $6$ && $1$ && $0$ && $3$ && $4$ && $5$ && $0$ && $2$ && $0$ &\cr
} \hrule}
}}
\end{footnotesize}
\end{section}
\begin{section}{The Igusa Quartic}\label{IgusaQuartic}
In this section we give three models of the Igusa quartic. 
The first is the one given above as the image of the Satake compactification 
${\rm Proj}(\oplus_k M_{0,2k}(\Gamma[2]))$ under the morphism
$\varphi$ above which is the variety in ${\PP}^4\subset {\PP}^{9}$ given 
by the linear equations  (representing an irreducible representation $s[2,1^4]$
of $\frak{S}_6$)
\begin{equation}\label{Igusaquartic00}
\begin{aligned}
&x_6=x_1-x_2+x_3-x_4-x_5, \, x_7=x_2-x_3+x_5, \\
&x_8=x_1-x_4-x_5,\, x_9=-x_3+x_4+x_5, \,  x_{10}=x_1-x_2-x_5 \\
\end{aligned}
\end{equation}
and the quartic equation
\begin{equation}\label{Igusaquartic0}
(\sum_{i=1}^{10} x_i^2)^2 -4\, \sum_{i=1}^{10} x_i^4 =0. 
\end{equation}
This variety admits an action of $\frak{S}_6$ induced by the action on the 
$x_i$ given by the irreducible $5$-dimensional representation $s[2^3]$.
It has exactly $15$ singular lines given as the $\frak{S}_6$-orbit of
$\{(a:a-b:a:a-b:b:b:0:0:0:0): (a:b)\in {\PP}^1\}$. The intersection points
of such lines
form the $\frak{S}_6$-orbit of $(1:1:1:1:0:0:0:0:0:0)$ of length $15$. Together
these form a $(15_3,15_3)$ configuration and are the images of the boundary components. Using Lemma
\ref{thetacharlemma} we get:

\begin{lemma}\label{geometryvsthetachar}
The fifteen $1$-dimensional boundary components of ${\mathcal A}_2[\Gamma[2]]^*$ correspond 1-to-1 to the
fifteen pairs of distinct odd theta characteristics. The fifteen $0$-dimensional boundary components
correspond 1-to-1 to the fifteen partitions of $\{m_1,\ldots,m_6\}$ 
into three pairs.
\end{lemma}

There is another model of the Igusa quartic given in 
${\PP}^4 \subset {\PP^5}$ by the equations (cf.\ \cite{vdG})
\begin{equation}\label{Igusaquartic2}
\sigma_1=0, \, \sigma_2^2-4\sigma_4=0
\end{equation}
with $\sigma_i$ the $i$th elementary symmetric function in the $6$ coordinates $y_1,\ldots, y_6$. We let the group
$\frak{S}_6$ act by $y_i \mapsto y_{\pi(i)}$ for $\pi \in \frak{S}_6$. The representation on the space of the
$y_i$ is $s[6]+s[5,1]$ with $\sigma_1$ representing the $s[6]$-part. We can connect the two models by using
the outer automorphism 
$$
\psi: \frak{S}_6\to \frak{S}_6\qquad \hbox{\rm with $\psi(12)=(16)(34)(25)$ and $\psi(123456)=(134)(26)(5)$}
$$
and the coordinate change $x_i=y_{a(i)}+y_{b(i)}+y_{c(i)}$ with $(a(i),b(i),c(i))$ given by
$$
\begin{matrix}
x_1 & x_2 & x_3 & x_4& x_5 & x_6 & x_7 & x_8 & x_9 & x_{10} \cr
(125) & (245) & (256) & (235) & (156) & (126) & (145) & (124) & (135) & (123) \cr
\end{matrix}
$$
or conversely $y_1=(2\, x_1-x_2-x_3-x_4)/3$, etc.\ (use the $\frak{S}_6$-actions).
In the model given by (\ref{Igusaquartic2}) the $1$-dimensional boundary components of ${\mathcal A}_2[\Gamma[2]]^*$
form the orbit of $\{ (x:x:y:y:-(x+y):-(x+y)): (x:y)\in {\PP}^1\}$. Under our conventions
the boundary component $B_{ij}$ is given by $y_a=y_b$, $y_c=y_d$, $y_e=y_f$ if
$\psi(ij)=(ab)(cd)(ef)$.

\smallskip
Yet another way to describe the Igusa quartic as a hypersurface in ${\PP}^4$ 
that we shall also use later
is by taking $x_1,\ldots,x_4,x_5-x_6$ as the generators
of $M_{0,2}(\Gamma[2])$. Then equation (\ref{Igusaquartic1}) reads
\begin{equation}\label{Igusaquartic3}
(s_1^2-4\, s_2-(x_5-x_6)^2)^2-64\, s_4=0
\end{equation}
where $s_i$ is the $i$th elementary symmetric function of $x_1,x_2,x_3,x_4$. The involution $\iota=(12)(34)(56)\in \frak{S}_6$
acts by sending $x_5-x_6$ to its negative and the fixed point locus is the Steiner surface $(s_1-4\, s_2)^2=64\, s_4$ in
${\PP}^3$ and it displays the quotient by $\iota$ as a double cover of ${\PP}^3$ branched along the four
planes given by $x_i=0$, $i=1,\ldots,4$, cf.\ Mukai \cite{Mukai}.
\end{section}
\begin{section}{Humbert Surfaces}
A Humbert surface in ${\mathcal A}_2$ (or ${\mathcal A}_2[G]$ for $G=\Gamma[2], \Gamma_1[2]$ or $\Gamma_0[2]$) is a divisor
parametrizing principally polarized abelian surfaces with multiplication 
by an order in a real quadratic field, or abelian surfaces that are isogenous to a 
product of elliptic curves. Some of these Humbert surfaces play a role in
the story of our modular forms.

The Humbert surface of invariant $\Delta$ in ${\mathcal A}_2[G]$ 
is defined in $\frak{H}_2$ 
by all equations of the form
$$
a \, \tau_{11}+b\, \tau_{12}+ c \, \tau_{22}+d(\tau_{12}^2-\tau_{11}\tau_{22})+e=0\, ,
$$
with primitive vector $(a,b,c,d,e) \in {\ZZ}^5$ satisfying
 $\Delta= b^2-4\, ac-4\, de$, cf.\ \cite{vdG}.
We can take their closures in the Satake compactifications 
${\mathcal A}_2[G]^*$. 
A Humbert surface of invariant $\Delta$ with $\Delta$ not a square intersects
the boundary only in the $0$-dimensional boundary components, while those
with $\Delta$ a square contain $1$-dimensional components.
 
In this paper the Humbert surfaces of invariant $1$, $4$ and $8$ 
will play a role.
The Humbert surface of invariant $1$ is the locus of principally polarized
abelian surfaces in ${\mathcal A}_2$ (resp.\ in ${\mathcal A}_2[\Gamma[2]]$ etc.)
that are products of elliptic curves. In ${\mathcal A}_2[\Gamma[2]]$ this 
locus consists of ten irreducible components, each isomorphic to 
$\Gamma(2) \backslash \frak{H}_1 \times 
\Gamma(2) \backslash \frak{H}_1$ and corresponding to the vanishing of
one even theta characteristic. 
In ${\mathcal A}_2[\Gamma_1[2]]$ this locus consists of
$4$ irreducible components, three of which are isomorphic to
$\Gamma_0(2) \backslash \frak{H}_1 \times 
\Gamma_0(2) \backslash \frak{H}_1$
and one is isomorphic to 
${\rm Sym}^2(\Gamma(2) \backslash \frak{H}_1)$. 

In ${\mathcal A}_2[\Gamma_0[2]]$ the Humbert surface of invariant $1$ has two
irreducible components. One is isomorphic to $(\Gamma_0(2)\backslash
\frak{H}_1)^2 $ and the other one to  ${\rm Sym}^2(\Gamma_0(2)\backslash
\frak{H}_1)$.  

The Humbert surface of invariant $4$ in ${\mathcal A}_2[\Gamma[2]]^*$ 
consists of $15$ components. 
In the model of the Igusa quartic given by (\ref{Igusaquartic2}) 
these components are given by $y_i=y_j$ with $1\leq i , j \leq 6$.
The product $\prod (y_i-y_j)$ defines the $\frak{S}_6$-anti-invariant 
modular form 
$$\chi_{30}=
(x_2-x_3)(x_2-x_4)(x_3-x_4)(x_3-x_5)(x_3-x_6)(x_5-x_6)
\prod_{i=2}^{10} (x_1-x_i)
$$ 
of weight $30$. The zero locus of $\chi_{35}=\chi_{30}\chi_5$ 
is supported on $H_1+H_4$.

A $0$-dimensional boundary component of ${\mathcal A}_2[\Gamma[2]]^*$ 
has as its stabilizer a (non-normal) subgroup
$\Gamma_0[2]$ in $\Gamma[2]$, hence determines a subgroup 
$\frak{S}_4\times {\ZZ}/2{\ZZ}$ in $\frak{S}_6$. 
The central involution of this group fixes a component of the Humbert
surface $H_4$. For our choice of $\Gamma_0[2]$ this is the surface
given in the Igusa quartic by
$$
x_5-x_6=0, \quad \hbox{\rm equivalently given by $x_7-x_8=0$ or $x_9-x_{10}=0$}
$$
or in the model with the $y$-coordinates by $y_2=y_5$.

The fixed point set of the Fricke involution on the Igusa quartic
consists of two curves and two isolated points as we 
shall see in Section \ref{Gamma1andIgusaQuartic}.
\end{section}
\begin{section}{The Ring of Scalar-Valued Modular Forms on $\Gamma_1[2]$ and $\Gamma_0[2]$}\label{gamma1andgamma0}
We now consider modular forms on $\Gamma_1[2]$ and $\Gamma_0[2]$. Note that
$M_{0,k}(\Gamma_1[2])$ is the invariant subspace of
$M_{0,k}(\Gamma[2])$ under the action of
$({\ZZ}/2{\ZZ})^3=\Gamma_1[2]/\Gamma[2]$. The space 
$M_{0,k}(\Gamma_1[2])$ is a representation space for 
$\mathfrak{S}_3=\Gamma_0[2]/\Gamma_1[2]$. Representation theory tells us that
a virtual $\mathfrak{S}_6$-representation 
$a_{s[6]}s[6]+a_{s[5,1]}s[5,1]+\ldots +a_{s[1^6]}s[1^6]$  
in $M_{j,k}(\Gamma[2])$ 
contributes a virtual $\mathfrak{S}_3$-representation
\begin{equation}\label{S6toS3reps}
(a_{s[6]}+a_{s[4,2]}+a_{s[2^3]})s[3]+(a_{s[5,1]}+a_{s[4,2]}+a_{s[3,2,1]})s[2,1]+
(a_{s[4,1^2]}+a_{s[3^2]})s[1^3]
\end{equation}
to $M_{j,k}(\Gamma_1[2])$, and hence a contribution 
$a_{s[6]}+a_{s[4,2]}+a_{s[2^3]}$ to the dimension of $M_{j,k}(\Gamma_0[2])$. 

\begin{proposition}\label{Multipl}
The generating function $\sum_{k \geq 0} m_{s[P],k} \, t^k$ of the 
irreducible $\frak{S}_3$ representations in $M_{0,k}(\Gamma_1[2])$ 
is given by
$$
\sum_{k \geq 0} m_{s[P],k} \, t^k = \frac{N_{s[P]}}{(1-t^2)(1-t^4)^2(1-t^6)}
$$
with $N_{s[3]}=1+t^{19}$, $N_{s[2,1]}=t^4+t^8+t^{11}+t^{15}$ 
and $N_{s[1^3]}=t^7+t^{12}$. 
\end{proposition} 

We thus find the 
following table of representations for $M_{0,k}(\Gamma_1[2])$ for even
$k\leq 12$.

\begin{footnotesize}
\smallskip
\vbox{
\bigskip\centerline{\def\quad{\hskip 0.6em\relax}\def\quod{\hskip 0.5em\relax }
\vbox{\offinterlineskip
\hrule\halign{&\vrule#&\strut\quod\hfil#\quad\cr
height2pt&\omit&&\omit&&\omit&&\omit &\cr
&$k\backslash P$ && $[3]$ && $[2,1]$ && $[1^3]$ & \cr
\noalign{\hrule}&$2$ && $1$ && $0$ && $0$ & \cr
\noalign{\hrule}&$4$ && $3$ && $1$ && $0$ & \cr
\noalign{\hrule}&$6$ && $4$ && $1$ && $0$ & \cr
\noalign{\hrule}&$8$ && $7$ && $4$ && $0$ & \cr
\noalign{\hrule}&$10$ && $9$ && $5$ && $0$ & \cr
\noalign{\hrule}&$12$ && $14$ && $10$ && $1$ & \cr
} \hrule}
}}
\end{footnotesize}

The structure of these rings of modular forms is as follows.

\begin{theorem}\label{scalarformsonGamma0}
The ring of scalar-valued modular forms on $\Gamma_0[2]$ is generated by
forms $s_1, s_2, \alpha, s_3$ of weight $2,4,4,6$ 
and a form $\chi_{19}$ of
weight $19$ with the ideal of relations generated by the relation (\ref{weight38rel}).
\end{theorem}

\begin{theorem}\label{RingGamma12}
The ring of scalar-valued modular forms on $\Gamma_1[2]$ is generated by
forms 
$s_1,s_2,\alpha, D_1,D_2, s_3$ 
of weight $2, 4, 4, 4, 4$ and $6$ 
and by a form $\chi_7$ in weight $7$. The ideal of relations is
generated by the relation (\ref{weight8rel}) in weight $8$, 
the relation (\ref{weight12rel})  in weight $12$ and the relation (\ref{weight14rel}) in weight $14$.
\end{theorem}

The relations are given explicitly below.
Theorem \ref{scalarformsonGamma0} is due to Ibukiyama, see 
\cite{Aoki,AokiIbukiyama}, but we give here an independent proof.
\begin{proof}
The group $\Gamma_0[2]/\Gamma[2]\simeq \mathfrak{S}_4\times {\ZZ}/2{\ZZ}$ 
acts on the ring $R^{\rm ev}$, generated by 
$x_1,\ldots,x_5$, but it will now be convenient to choose $x_5-x_6=2x_5-x_1+x_2-x_3+x_4$ as the last generator.
Then $\mathfrak{S}_4$ acts on $x_1,\ldots,x_4$ by $x_i \mapsto x_{\sigma(i)}$
and ${\ZZ}/2{\ZZ}$ acts trivially on $x_1,\ldots,x_4$ and by $-1$ on 
$x_5-x_6$. 
The ring of invariants is the ring
$M_*^{\rm ev}(\Gamma_0[2])$,  while the ring of invariants under the
subgroup $({\ZZ}/2{\ZZ})^3=\Gamma_1[2]/\Gamma[2]$ is the ring
$M_*^{\rm ev}(\Gamma_1[2])$.

The ring of invariants of the subring generated by $x_1,\ldots,x_4$
is generated by the $\frak{S}_4$ elementary symmetric functions 
$s_1,s_2,s_3$ 
and $s_4$ in these $x_i$.  A further invariant is $\alpha=(x_5-x_6)^2$.
We now find eight forms of weight $8$, namely
$s_4, s_3 s_1,s_2^2, s_2s_1^2, s_1^4,
\alpha s_2, \alpha s_1^2, \alpha^2$
and as we know that $\dim M_{0,8}(\Gamma_0[2])=7$ we find one linear relation.
Since all these forms live in $M_{0,8}(\Gamma[2])$ 
this must be (a multiple of) 
the Igusa quartic relation expressing $s_4$ in the other forms.
To make this explicit, note that $\vartheta_1^2\vartheta_2^2\vartheta_3^2\vartheta_4^2$
is in $M_{0,4}(\Gamma_1[2])$, and it equals $(-s_1^2+4\, s_2+\alpha)/8$
as one checks.
We thus see that
\begin{equation}\label{Igusaquartic6}
64\, s_4=(-s_1^2+4\, s_2+\alpha)^2\, . 
\end{equation}
There can be no further  relations because
the ideal of relations among the $x_1,\ldots,x_4,x_5-x_6$ is generated
by the Igusa quartic. So $M_*^{\rm ev}(\Gamma_0[2])$ contains a subring
generated by $s_1, s_2, \alpha$ and $s_3$ 
with Hilbert function $(1-t^8)/(1-t^2)(1-t^4)^2(1-t^6)$, and this is the
Hilbert function of $M_*^{\rm ev}(\Gamma_0[2])$, see Prop.\ \ref{Multipl}. 
Therefore there can be no further relations and we
found the ring $M_*^{\rm ev}(\Gamma_0[2])$.

\smallskip
For $M_*^{\rm ev}(\Gamma_1[2])$ we look at the invariants under 
$({\ZZ}/2{\ZZ})^3$. The $s[2,1]$-subspace of $M_{0,4}(\Gamma_1[2])$ 
has a basis $D_1,D_2$ with
$$
D_1=(x_1-x_2)(x_3-x_4)
\quad
{\rm{and}}
\quad
D_2=(x_1-x_3)(x_2-x_4)\, .
$$
Since the form $D_1^2-D_1D_2+D_2^2$ is $\mathfrak{S}_3$-invariant 
(and equal to $s_2^2-3\, s_1s_3+12\, s_4$) we
have using (\ref{Igusaquartic6}) a relation in weight $8$
\begin{equation}\label{weight8rel}
16(D_1^2-D_1D_2+D_2^2) = 3 \alpha^2-6(s_1^2-4\, s_2) \alpha
+3\, s_1^4-24\, s_1^2s_2-48\, s_1s_3+64 s_2^2\,  . 
\end{equation}
The expression
$$
C= -\vartheta_5^2 \vartheta_6^2 \cdots \vartheta_{10}^2
=\frac{1}{2} ((x_1x_3-x_2x_4)(x_5+x_6)+s_1 x_5x_6)
$$
defines an element of $M_{0,6}(\Gamma_1[2])$ 
(but with a non-trivial character on $\Gamma_0[2]$) 
and thus 
can be expressed polynomially in $s_1^3, s_1 s_2, s_3, \alpha s_1,
D_1 s_1$  and $D_2 s_1$.
It satisfies the relation 
\begin{equation}\label{weight12rel}
C^2= x_5 \cdots x_{10},
\end{equation}
where $x_5\cdots x_{10}$ is $\mathfrak{S}_3$-invariant.
We thus find a subring of $M_*^{\rm ev}(\Gamma_1([2])$ generated over 
$M_*^{\rm ev}(\Gamma_0[2])$ by $D_1$ and $D_2$ and we have two  algebraic
relations, one of degree $8$ and one of degree $12$ given by (\ref{weight8rel}) and (\ref{weight12rel}). 
We can have no third independent algebraic relation because 
there are no algebraic relations among
$s_1, s_2, \alpha$ and $s_3$. 
The Hilbert function of this subring is 
$(1-t^8)(1-t^{12})/(1-t^2)(1-t^4)^4(1-t^6)$ 
and coincides with the  Hilbert function of $M_*^{\rm ev} (\Gamma_1[2])$. 
This shows that we found the ring $M_*^{\rm ev}(\Gamma_1[2])$. 

\smallskip
We can construct a cusp form of weight $7$ in the $s[1^3]$-subspace
of $M_{0,7}(\Gamma_1[2])$, namely
$$
\chi_7= \chi_5 (x_6-x_5)\, .
$$
Since we have $\chi_5^2=-2^{14}\chi_{10}$ we find a relation in weight $14$:
\begin{equation}\label{weight14rel}
\chi_7^2=-2^{14}\, \chi_{10} \, \alpha\, .
\end{equation}

Furthermore, we have the square root of the discriminant
$$
\delta=(x_1-x_2)\cdots (x_3-x_4)
$$
which is a modular form in the $s[1^3]$-subspace of $M_{0,12}(\Gamma_1[2])$.
We thus find a cusp form $\chi_{19}= \chi_7\delta$ in $S_{0,19}(\Gamma_0[2])$.
It satisfies the relation
\begin{equation}\label{weight38rel}
\chi_{19}^2= -2^{14}(x_1-x_2)^2 \cdots (x_3-x_4)^2 \chi_{10} (x_5-x_6)^2\, .
\end{equation}
We now show that each modular form of odd weight on $\Gamma_1[2]$ is 
divisible by $\chi_7$. In fact, such a form $f$ is also a modular form on
$\Gamma[2]$, hence is divisible by $\chi_5$ as a modular form on $\Gamma[2]$.
Next, we show that $f$ also vanishes on the component of the Humbert surface
defined by $x_5-x_6=0$. For this we look at the action of a representative
of the element $\iota=(12)(34)(56)$ and observe that $f/\chi_5$ changes sign
under this action, hence vanishes on the locus where $x_5=x_6$. So
$M_{0,k}(\Gamma_1[2])=\chi_7 M_{0,k-7}(\Gamma_1[2])$ for odd $k$.
But by a similar argument any odd weight modular 
form on $\Gamma_0[2]$ also vanishes
on the other components of the Humbert surface $H_4$, and hence is divisible by
$\delta$.
\end{proof}
\begin{remark}
Note that we have $M_{0,2k}({\Gamma_1[2]})^{s[1^3]}=\delta \, M_{0,2k-12}(\Gamma_0[2])$.
\end{remark} 
\begin{remark}
Ibukiyama constructed $\chi_{19}$ as a Wronskian, see \cite{AokiIbukiyama}.
\end{remark}

\end{section}
\begin{section}{The Action of the Fricke Involution}\label{FrickeInvolution}
We start by computing the action of the Fricke involution on the modular
forms on $\Gamma_1[2]$.
Recall that the Fricke involution $W_2$ given by formula (\ref{Fricke})
acts on ${\mathcal A}_2[\Gamma_1[2]]$ and
${\mathcal A}_2[\Gamma_0[2]]$ and thus induces an action on modular forms
via $f \mapsto W_2(f)=f\vert_{j,k} W_2$ (where we sometimes omit the indices $j,k$). 

\begin{lemma}\label{W2onxi} 
The transformation formula for the 
$x_i=\vartheta_i^4$ ($1 \leq i \leq 4$) under $W_2$ is:
\begin{align*}
x_{1}\vert_{0,2}\, W_2&=(\vartheta_1^2+\vartheta_2^2+\vartheta_3^2+\vartheta_4^2)^2/4
\\
x_{2}\vert_{0,2}\, W_2&=(\vartheta_1^2-\vartheta_2^2+\vartheta_3^2-\vartheta_4^2)^2/4
\\
x_{3}\vert_{0,2}\, W_2&=(\vartheta_1^2+\vartheta_2^2-\vartheta_3^2-\vartheta_4^2)^2/4
\\
x_{4}\vert_{0,2}\, W_2&=(\vartheta_1^2-\vartheta_2^2-\vartheta_3^2+\vartheta_4^2)^2/4.
\end{align*}

\end{lemma}
\begin{proof}
Setting $T=2\tau$, we get by definition
$$
(\vartheta_i^4 \vert_{0,2}\, W_2)(T/2)=\det(-T/\sqrt{2})^{-2}\vartheta_i^4(-T^{-1})
=4\det(-T)^{-2}\vartheta_i^4(-T^{-1}).
$$
This expression is closely related to the transformation formula of the $\vartheta_i^4$
under the element
$
J=
\left(
\begin{smallmatrix}
0 & 1_2\\
-1_2 & 0
\end{smallmatrix}
\right)
$
which reads
\[
\vartheta_{J\cdot\left[\begin{smallmatrix}\mu  \\ \nu
\end{smallmatrix}\right]}^4(-\tau^{-1})=
\kappa(J)^4\,
\det(-\tau)^{2}
\vartheta_{\left[
\begin{smallmatrix}\mu \\ \nu
\end{smallmatrix}\right]}^4
(\tau)
\]
since $8\, \phi(\left[\begin{smallmatrix} \mu \\ \nu \end{smallmatrix}\right],J)=2\,  \mu  \nu^t \in 2\ZZ$.
We know that $\kappa(J)^4=\pm1$ and we can determine its value by using
$\vartheta_1^4=\vartheta_{\left[
\begin{smallmatrix} 0 &  0\\ 0 & 0
\end{smallmatrix}\right]}^4$ whose characteristic is fixed by $J$
and evaluating the latter equation at $\tau=i1_2=-\tau^{-1}$, getting
$\kappa(J)^4=1$. Taking into account the action of $J$ on the characteristics
we thus find 
$$
(\vartheta_{i}^4\vert_{0,2}\, W_2)(\tau)=4\, \vartheta_{w(i)}^4(2\tau)\, ,
$$
where $[w(1),\ldots,w(10)]=[1,5,7,9,2,8,3,6,4,10]$. We now use Riemann's bilinear relations
(\ref{bilinear}) to see
\begin{align*}
\vartheta_1^2(\tau)&=\vartheta_1^2(2\tau)+\vartheta_5^2(2\tau)+\vartheta_7^2(2\tau)+\vartheta_9^2(2\tau)\\
\vartheta_2^2(\tau)&=\vartheta_1^2(2\tau)-\vartheta_5^2(2\tau)+\vartheta_7^2(2\tau)-\vartheta_9^2(2\tau)\\
\vartheta_3^2(\tau)&=\vartheta_1^2(2\tau)+\vartheta_5^2(2\tau)-\vartheta_7^2(2\tau)-\vartheta_9^2(2\tau)\\
\vartheta_4^2(\tau)&=\vartheta_1^2(2\tau)-\vartheta_5^2(2\tau)-\vartheta_7^2(2\tau)+\vartheta_9^2(2\tau)\\
\end{align*}
from which the result follows. 
\end{proof}
By looking at the values of the $x_i$ at the fifteen cusps for $\Gamma[2]$ we
derive easily the action on the $0$-dimensional and $1$-dimensional
cusps of $\Gamma_1[2]$. We use the notation of Section \ref{Preliminaries}. 
\begin{corollary}\label{W2oncusps}
The action of $W_2$ on the cusps of ${\mathcal A}_2[\Gamma_1[2]]^*$ is as follows\footnote{Here the letters $\alpha,\ldots,\epsilon$, $A,\ldots,F$ refer to the figure in Section
\ref{Preliminaries}}
$$
\begin{aligned}
W_2(\gamma)=\gamma, &&  W_2(\alpha)=\epsilon,  && W_2(\beta)=\delta, \\
W_2(A)=F, && W_2(B)=E, &&  W_2(C)=D. \\
\end{aligned}
$$
\end{corollary}
The action of $W_2$ on the cusps of $\Gamma_0[2]$ can be deduced immediately from this.

Using Lemma \ref{W2onxi} we find that $s_1|_{0,2}W_2= s_1$ and similarly
\begin{align*}
s_{2}\vert_{0,4}\, W_2&=3\, s_1^2/8-s_2/2-3\vartheta_1^2\vartheta_2^2\vartheta_3^2\vartheta_4^2 \, .
\end{align*}
Since $\vartheta_1^2\vartheta_2^2\vartheta_3^2\vartheta_4^2
\in M_{0,4}(\Gamma_1[2])$, we can express it in our basis and get
\[
\vartheta_1^2\vartheta_2^2\vartheta_3^2\vartheta_4^2=-s_1^2/8+s_2/2+\alpha/8\, ,
\]
where $\alpha$ denotes the modular form $(x_5-x_6)^2$ introduced in Section
\ref{gamma1andgamma0} and
thus
\begin{align*}
s_{2}\vert_{0,4}\, W_2&=3\, s_1^2/4-2\, s_2-3\alpha/8\, .
\end{align*}
As $W_2$ is an involution, we get its action on $\alpha$ using the last equation:
\begin{equation}\label{W2alpha}
\alpha\vert_{0,4}\, W_2=-2\, s_1^2+8\, s_2+2\alpha\, .
\end{equation}
We can refine it as follows.
\begin{lemma}\label{W2xi}
We have 
$(x_5-x_6)\vert_{0,2}\, W_2=4\, \vartheta_1\vartheta_2\vartheta_3\vartheta_4$.
\end{lemma}
\begin{proof}
We know that 
$(\vartheta_1\vartheta_2\vartheta_3\vartheta_4\vert_{0,2}\, W_2)(\tau)=
4\,(\vartheta_1\vartheta_5\vartheta_7\vartheta_9)(2\tau)
$
but we also know that
$$
\vartheta_5^2(\tau)=2\, (\vartheta_1\vartheta_5+\vartheta_7\vartheta_9)(2\tau)
\quad {\rm and}\quad
\vartheta_6^2(\tau)=2\, (\vartheta_1\vartheta_5-\vartheta_7\vartheta_9)(2\tau)
$$
and this implies $\vartheta_1\vartheta_2\vartheta_3\vartheta_4\vert_{0,2}\, W_2=
(x_5-x_6)/4$ and thus the lemma since $W_2$ is an involution.
\end{proof}

We summarize the results.

\begin{proposition}
The action of the involution $W_2$ on the generators is given by $s_1\vert{W_2}=s_1$, 
$s_2\vert{W_2}=3\, s_1^2/4 -2\, s_2-3\, \alpha/8$, $\alpha\vert{W_2}=-2\, s_1^2+8\, s_2 +2\, \alpha$
and $D_1\vert{W_2}=D_2$. Furthermore $s_3\vert{W_2}= s_3+ s_1^3/8 -s_1s_2/2-s_1\alpha/16$, 
$\chi_7\vert{W_2}=\chi_7$ and $\chi_{19}\vert{W_2}=- \chi_{19}$.
\end{proposition}
\begin{remark}
The trace of the action of $W_2$ on the space $M_{0,4}(\Gamma_1[2])$ is equal
to $1$.
\end{remark}
\end{section}
\begin{section}{${\mathcal A}_2[\Gamma_1[2]]^*$ and the Igusa Quartic}\label{Gamma1andIgusaQuartic}
In his study of moduli of Enriques surfaces Mukai found that the Satake compactification
of ${\mathcal A}_2[\Gamma_1[2]]^*$ is isomorphic to the Igusa quartic, see \cite{Mukai}.
He showed this using the
geometry. We give an independent proof of this using modular forms.
We will show that the scalar-valued 
modular forms of weight divisible by $4$ define an embedding of
${\mathcal A}_2[\Gamma_1[2]]$ into projective space and that the closure of the image is the Igusa quartic.

We know that the ring of modular forms on $\Gamma[2]$ 
is generated by the modular forms
$x_1,x_2,x_3,x_4$ and $\xi=x_5-x_6$ of weight $2$. These satisfy the relation
\begin{equation}\label{Igusaquartic4}
(s_1^2-4\, s_2-\xi^2)^2=64\, s_4 
\end{equation}
as we know from (\ref{Igusaquartic6}), but as follows also from 
comparing equation (\ref{W2alpha}) and Lemma (\ref{W2xi}).

We define the following modular forms in $M_{0,4}(\Gamma_1[2])$:
$$
\begin{aligned}
&X_1=(x_1+x_2+x_3+x_4)^2, \quad X_2=(x_1-x_2+x_3-x_4)^2, \\
&X_3=(x_1+x_2-x_3-x_4)^2, \quad X_4=(x_1-x_2-x_3+x_4)^2  \\
\end{aligned}
$$
and
$$
\eta= -16\, \vartheta_1^2\vartheta_2^2\vartheta_3^2\vartheta_4^2=
2(s_1^2-4\, s_2-\xi^2).
$$
\begin{proposition}
The modular forms $X_1,X_2,X_3,X_4$ and $\eta$ generate $M_{0,4}(\Gamma_1[2])$.
\end{proposition}
\begin{proof}
These forms lie in $M_{0,4}(\Gamma[2])$, are invariant under $({\ZZ}/2{\ZZ})^3$ and linearly
independent as one readily sees, cf Thm.\ \ref{RingGamma12}.
\end{proof}
Let $\gamma_i$ be the $i$th elementary symmetric function in the $X_1,\ldots,X_4$. 
Then one checks that 
$$
\eta^4-2(\gamma_1-4\gamma_2) \eta^2 +(\gamma_1^2-4\gamma_2)^2-64 \, \gamma_4=0
$$
since by equation (\ref{Igusaquartic4}) we have $\eta^2=2^8 \, s_4$.
This means that $X_1,\ldots,X_4,\eta$ satisfy the equation
\begin{equation}\label{Igusaquartic5}
(\gamma_1^2-4\gamma_2-\eta^2)^2=64 \gamma_4
\end{equation}
which is the same as (\ref{Igusaquartic4}) and thus defines the Igusa quartic.

It is easy to see that the ideal of relations among the $x_i$ intersected with $\oplus_k M_{0,4k}(\Gamma_1[2])$
is generated by relation (\ref{Igusaquartic5}), hence the $X_1,\ldots,X_4,\eta$ generate a subring with
Hilbert function $(1-t^{16})/(1-t^4)^5$ and since this equals the 
Hilbert function of $\oplus_k M_{0,4k}(\Gamma_1[2])$
we have the structure of $\oplus_k M_{0,4k}(\Gamma_1[2])$.

\begin{corollary}\label{Mukairesult} (Mukai, \cite{Mukai}) The Satake compactification ${\mathcal A}_2[\Gamma_1[2]]^*$ is isomorphic
to the Igusa quartic. 
\end{corollary}

It follows that there is an action of $\frak{S}_6$
on $\oplus_k M_{0,4k}(\Gamma_1[2])$. This action does not preserve the set of boundary
components, as ${\mathcal A}_1[\Gamma_1[2]]^*$ has only six $1$-dimensional boundary components
and $\frak{S}_6$ acts transitively on the set of $15$ singular lines. Therefore a large part of
the automorphism group of ${\mathcal A}_1[\Gamma_1[2]]^*$ is not modular (i.e.\ not
induced by elements of ${\rm Sp}(4,{\QQ})$).  To see this action
we now define the modular forms 
$$
\begin{aligned}
& X_5= (\eta+X_1-X_2+X_3-X_4)/2, \,  &X_6=(-\eta+X_1-X_2+X_3-X_4)/2,\\
& X_7=(\eta+X_1+X_2-X_3-X_4)/2, \, &X_8=(-\eta+X_1+X_2-X_3-X_4)/2,\\
& X_9=(\eta+X_1-X_2-X_3+X_4)/2, \, &X_{10}=(-\eta+X_1-X_2-X_3+X_4)/2 \,. \\
\end{aligned}
$$
We have
$$
X_5=4\, x_7x_8, \, X_7=4\, x_5x_6, \, X_9=4\, x_9x_{10},
$$
and 
$$
X_6=4\, (\vartheta_1^2\vartheta_2^2+\vartheta_3^2\vartheta_4^2)^2, \,
X_8=4\, (\vartheta_1^2\vartheta_3^2+\vartheta_2^2\vartheta_4^2)^2, \,
X_{10}=4\, (\vartheta_1^2\vartheta_4^2+\vartheta_2^2\vartheta_3^2)^2.
$$
These ten $X_i$ generate formally a representation $s[2^3]+s[2,1^4]$ and satisfy linear relations
of type $s[2,1^4]$ as the $x_i$ do. 
They satisfy the quartic relation 
$(\sum X_i)^2-4\, \sum X_i^4=0$.

The action of $W_2$ on the ten $X_i$ is given by $X_i \mapsto X_{w(i)}$
with $(w(1),\ldots,w(10))$ given by $(1,6,8,10,7,2,5,3,9,4)$.
The action of $W_2$ on $\eta$ is
$$
\eta\vert{W_2}=(X_1-X_2-X_3-X_4+\eta)/2 \, .
$$
\smallskip

\begin{construction}\label{constructionU}
We view the $X_i$ as the analogues for $\Gamma_1[2]$ 
of the $x_i=\vartheta_i^4$ for $\Gamma[2]$. 
We can also define modular forms
with a character on $\Gamma_1[2]$ that play a role analogous to the role 
that the theta squares $\vartheta_i^2$ 
play for $\Gamma[2]$ as follows. 

\begin{equation}
\begin{aligned}\label{Urelation}
&U_1=(x_1+x_2+x_3+x_4), \quad
U_2=(x_1-x_2+x_3-x_4)\\
&U_3=(x_1+x_2-x_3-x_4), \quad
U_4=(x_1-x_2-x_3+x_4)\\
&U_5=2\,\vartheta_5^2\vartheta_6^2, \quad
U_7=2\, \vartheta_7^2\vartheta_8^2,\quad
U_9=2\, \vartheta_9^2\vartheta_{10}^2, \\
&U_6=2\, (\vartheta_1^2\vartheta_2^2+\vartheta_3^2\vartheta_4^2),\quad
U_8=2\, (\vartheta_1^2\vartheta_3^2+\vartheta_2^2\vartheta_4^2),\quad
U_{10}=2\, (\vartheta_1^2\vartheta_4^2+\vartheta_2^2\vartheta_3^2). \\
\end{aligned}
\end{equation}
The $45$ modular forms $U_iU_j$ of weight $4$ with character on $\Gamma_1[2]$
satisfy  equations like 
$$
U_1U_2-U_3U_4=U_7U_8, \,
U_1U_3-U_2U_4=U_5U_6,\,
U_1U_4-U_2U_3=U_9U_{10}\, .
$$  
We shall use them later to construct vector-valued modular forms on $\Gamma_1[2]$.
\end{construction}

\begin{remark}
The automorphism group of the Igusa quartic is $\frak{S}_6$. 
This implies that $\frak{S}_6$ acts on the ring 
$R_{(4)}=\oplus_k M_{0,4k}(\Gamma_1[2])$.
But not all automorphisms preserve the boundary ${\mathcal A}_2[\Gamma_1[2]]^*-{\mathcal A}_2[\Gamma_1[2]]$, hence not all automorphisms are induced
by an action on $\frak{H}_2$ as we saw above.

On the other hand we have a natural action of the  subgroup 
$\frak{G}$ generated by $\frak{S}_3$ and $W_2$ on
${\mathcal A}_2[\Gamma_1[2]]$, where $\frak{S}_3= 
\Gamma_0[2]/\Gamma_1[2]$ is a subquotient of
$\frak{S}_6=\Gamma/\Gamma[2]$. The group $\frak{S}_3$ 
is generated by the two elements
$X'$ and $Y'$ given in (\ref{AprimeBprime}). 
To express this action on $R_{(4)}$ we choose as 
generators the modular forms $Y_i$ defined by $X_i=Y_{a(i)}+Y_{b(i)}+Y_{c(i)}$
with $(a(i),b(i),c(i))$ given as in Section \ref{IgusaQuartic}. One then calculates
the induced action.

\begin{lemma}
The action of $X'$ (resp.\ $Y'$, resp.\ $W_2$) 
on the generators $Y_i$ ($i=1,\ldots,6$) of $M_{0,4}(\Gamma_1[2])$
is given by
$(Y_1,\ldots,Y_6) \mapsto (Y_1,Y_2,Y_6,Y_4,Y_5,Y_3)$ 
(resp.\ $(Y_1,Y_2,Y_6,Y_3,Y_5,Y_4)$, resp.\ $(Y_5,Y_2,Y_3,Y_6,Y_1,Y_4)$).
\end{lemma}
 Since $\frak{G}$ is a group of automorphisms
of the ring $R_{(4)}$ it acts by automorphisms on the Igusa quartic
and it can be viewed as the subgroup of {\sl modular} automorphisms
of ${\mathcal A}_2[\Gamma_1[2]]^*$
(i.e. induced by an action of elements of ${\rm Sp}(4, \overline{\QQ})$ on $\frak{H}_2$) of  $\frak{S}_6$. 
It is the subgroup of permutations that preserve the 
partition $\{2\}\sqcup \{1,5\} \sqcup \{3,4,6\}$ of $\{1,\ldots,6\}$.

Finally, we give the fixed point locus of the Fricke involution. 
\begin{lemma}
In the model of the Igusa quartic given by \ref{Igusaquartic2} 
the fixed point locus of $W_2$ 
is given by the equations $y_1=y_5$ and $y_4=y_6$.
It consists of a singular line and a conic section and two isolated fixed points.
\end{lemma}
\begin{proof} The action is given by the permutation $(y_1\,  y_5)(y_4\,  y_6)$.
A fixed point is either of the form $(1:0:0:0:\pm 1: -1: \mp 1)$ or $(a:b:c:d:a:d)$
with $2a+b+c+2d=0$ and in the latter case 
the Igusa quartic equation (\ref{Igusaquartic2}) factors as 
a double line and a quadric.
\end{proof} 
\end{remark}

\end{section}
\begin{section}{Dimension Formulas for Vector-Valued Modular Forms on
$\Gamma[2]$}\label{dimensions}

We now give formulas for the dimension of the spaces of
vector-valued modular forms $M_{j,k}(\Gamma[2])$ and 
$S_{j,k}(\Gamma[2])$.
These formulas can be proved using the Hirzebruch-Riemann-Roch formula
or the Selberg trace formula. In fact, a recent paper by Wakatsuki
\cite{wakatsuki} proves the formula for $S_{j,k}(\Gamma[2])$ for $k\geq 5$
using the Selberg trace formula.

Since the group $\Gamma[2]$ contains $-1_4$ it follows that $M_{j,k}(\Gamma[2])=(0)$
for all odd $j$. Furthermore, we have $M_{j,k}(\Gamma[2])=S_{j,k}(\Gamma[2])$ for odd $k$.

\begin{theorem} For $k\geq 3 $ odd and $j\geq 2$ even (or for $k \geq 5$ odd and $j=0$) we have
$$
\begin{aligned}
\dim M_{j,k}(\Gamma[2])=\dim S_{j,k}(\Gamma[2])=
{1 \over 24} \big[ 2(j+1)\, k^3 + 3(j^2-2j-8)\, k^2 + &\\
+(j^3-9j^2-42 j+118)\, k +(-2j^3-9j^2+& 152 j-216)\big] .\\
\end{aligned}
$$
For $k\geq 4$ even and $j \geq 2$ even we have
$$
\begin{aligned}
\dim M_{j,k}(\Gamma[2])= {1 \over 24} \big[
2(j+1)\, k^3+ 3(j^2-2j+2)\, k^2+ (j^3-9j^2-12j+28)\, k &\\
+(-2j^3-9j^2+182 j - 336) \big] . \\
\end{aligned}
$$
Furthermore, for $k\geq 0$ even we have
$$
\dim M_{0,k}(\Gamma[2])=
{(k+1)(k^2+2\, k + 12)\over 12}
$$
and $\dim M_{0,k+5}(\Gamma[2])=\dim M_{0,k}(\Gamma[2])$ for $k\geq 0$ even.
\end{theorem}

\begin{remark}
As we shall see in the next section for $k\geq 4$ even and $j+k\geq 6$ we have
$$
\dim M_{j,k}(\Gamma[2])= \dim S_{j,k}(\Gamma[2]) + 15(j+k-4)/2\, .
$$
\end{remark}

We can rewrite these formulas in the form of a generating series.
\begin{theorem}\label{ThmGenFunction}
The generating function for the dimension of $M_{j,k}(\Gamma[2])$ for 
fixed even $j\geq 2$ and $k\geq 3$ is given as 
$$
\sum_{k \geq 3} \dim M_{j,k}(\Gamma[2])\, t^k = 
\frac{ \sum_{i=3}^{12} a_i \, t^i}{(1-t^2)^5}
$$
with $a_n=a_n(j)$ given by

\bigskip
\vbox{
\centerline{\def\quad{\hskip 0.3em\relax}
\vbox{\offinterlineskip
\hrule
\halign{&\vrule#& \quad \hfil#\hfil \strut \quad  \cr
height2pt&\omit&&\omit&&\omit&&\omit& \cr
& $n$ && $a_n$ && $n$ && $a_n$ & \cr
\noalign{\hrule}
height2pt&\omit&&\omit&&\omit&&\omit& \cr
& $3$ && $(j-2)(j-3)(j-4)/24$ && $4$ && $j \, (2\, j^2+3\, j+166)/24$ &\cr
& $5$ && $(-j^3+33\, j^2-44\, j+72)/12 $ && $6$ && $-(j-1)(j^2-4\, j+80)/4$ & \cr
& $7$ && $(-10\, j^2+25\, j -20)/2$ && $8$ && $j^3/4-7j^2/2+63j/2-46$ & \cr
& $9$ && $(j^3+39\, j^2-172\, j+120)/12$ && $10$ && $-j^3/12+11j^2/4-71j/3+36$ & \cr
& $11$ && $(-j^3-15\, j^2+106 \, j -120)/24$ && $12$ && $-5j^2/8+25j/4-10$ & \cr
height2pt&\omit&&\omit&&\omit&&\omit& \cr
} \hrule}
}}

\end{theorem}

\begin{remark}\label{Euler=0}
Note that we have for Theorem \ref{ThmGenFunction} the identities 
$a_3+a_5+a_7+a_9+a_{11}=0$ and $a_4+a_6+a_8+a_{10}+a_{12}=0$;
see Section \ref{ModuleStructure} for an explanation.
\end{remark}
\end{section}
\begin{section}{Representations of $\frak{S}_6$ on Eisenstein Spaces}
As a result of \cite{BFvdG} we can calculate the representation of
the group $\mathfrak{S}_6$ on the spaces $S_{j,k}(\Gamma[2])$ algorithmically for
$j+k\geq 5$ assuming the conjectures there.
This yields very helpful information for determining
the structure of the modules ${\mathcal M}_j=\oplus_k M_{j,k}(\Gamma[2])$ and $\Sigma_j=\oplus_k S_{j,k}(\Gamma[2])$ and agrees in all cases
we considered with
the dimension formulas for $M_{j,k}(\Gamma[2])$ and $S_{j,k}(\Gamma[2])$. 
Moreover, for small
weights the $\mathfrak{S}_6$-representation can be determined by 
combining the dimension formula with the cohomological calculations 
from \cite{BFvdG} using point counting over finite fields or by using the
module structure over $R^{\rm ev}$. In view of this it will be 
useful to know the representation of $\frak{S}_6$ on the subspaces
of the spaces of modular forms for the groups $\Gamma[2], \Gamma_1[2]$ 
and $\Gamma_0[2]$ generated by Eisenstein series. 
We will denote the orthogonal complement of the space 
$S_{j,k}(G)$ in $M_{j,k}(G)$ w.r.t.\ the Petersson product
for $G=\Gamma[2], \Gamma_1[2]$ or $\Gamma_0[2]$ by $E_{j,k}(G)$.
\begin{remark}
We have $E_{j,k}(\Gamma[2])=(0)$ if $k$ is odd.
\end{remark}

The Eisenstein subspace $E_{j,k}(\Gamma[2])$ of $M_{j,k}(\Gamma[2])$
is also a representation space of $\mathfrak{S}_6$.
By using Siegel's operator for one of the $15$ boundary components
of ${\mathcal A}_2[\Gamma[2]]$
it maps to the space of
cusp forms $S_{j+k}(\Gamma(2))\cong S_{j+k}(\Gamma_0(4))$ where $\Gamma(2)$
and $\Gamma_0(4)$ are the usual congruence subgroups of ${\rm SL}(2,{\ZZ})$.
The dimension of $S_{2r}(\Gamma(2))$ equals $r-2$ for $r\geq 3$
and is zero otherwise. The space $S_r(\Gamma(2))$ is a representation space
for the symmetric group $\mathfrak{S}_3= {\rm SL}(2,{\ZZ})/\Gamma(2)$. 
The stabilizer in $\frak{S}_6$ of
one $1$-dimensional boundary component  is a group $H$ of order $48$
and this group acts on the $1$-dimensional boundary component via its quotient
$\mathfrak{S}_3$. 

As a representation space of $\mathfrak{S}_3$ the vector space
$S_{2r}(\Gamma(2))$ is of the form
$$
{\rm Sym}^r (s[2,1])- \begin{cases} s[2,1] & r=1 \\
s[3]+s[2,1] & r>1 \\\end{cases}
$$
because the ring of modular forms on $\Gamma(2)$ is generated
by two modular forms of weight~$2$ that form an irreducible representation $s[2,1]$ and the space
of Eisenstein series is a representation space
$s[3]+s[2,1]$ except in weight $2$,
where it is a $s[2,1]$. We have
$$
{\rm Sym}^r(s[2,1])=(1+[r/6]+\epsilon)\, s[3]+[(r+2)/3]\, s[2,1]+
([(r+3)/6]+\epsilon')\, s[1^3]
$$
with $\epsilon=-1$ if $k\equiv 1 \bmod 6$ and $\epsilon'=-1$ 
if $k \equiv 4 \bmod 6$ and $\epsilon=0$ and $\epsilon'=0$ else.
The representation of $\frak{S}_6$ on the Eisenstein subspace of $M_{j,k}(\Gamma[2])$ is
thus
$$
{\rm Ind}_{H}^{\mathfrak{S}_6}\left({\rm Sym}^{(j+k)/2} (s[2,1])-s[3]-s[2,1] \right)
$$
for
$j+k\geq 4$. We have 
$$
\begin{aligned}
{\rm Ind}_H^{\frak{S}_6}(s[3])&=s[6]+s[5,1]+s[4,2] \\
{\rm Ind}_H^{\frak{S}_6}(s[2,1])&=s[4,2]+s[3,2,1]+s[2^3]\\
{\rm Ind}_H^{\frak{S}_6}(s[1^3])&=s[3,1^3]+s[2,1^4] .\\
\end{aligned}
$$
\begin{proposition} For $k\geq 2$
the space $E_{j,k}(\Gamma[2])$ as a representation space of $\frak{S}_6$
equals
$$
E_{j,k}(\Gamma[2])=
{\rm Ind}_{H}^{\mathfrak{S}_6} 
\left({\rm Sym}^{k} (s[2,1])-s[3]-s[2,1] 
\right) +
\begin{cases} 
s[6]+s[4,2]+s[2^3] & j=0, \\
0 & j\geq 2. \\
\end{cases}
$$
\end{proposition}

\begin{corollary} For $k\geq 2$
the space $ E_{j,2k}(\Gamma_1[2])$ as a representation of $\frak{S}_3$ equals
$$
\begin{cases} 
a_k(s[3]+s[2,1])-2\, s[2,1] & j=0, \\
b_{j,k}(s[3]+s[2,1]) & j\geq 2,\\
\end{cases}
$$
where $a_k=k$ if $k$ is odd and $a_k=k+1$ if $k$ is even
and $b_{j,k}= j/2+k-3$ if $j/2+k$ is odd and $j/2+k-2$ if $j/2+k$ is even.
\end{corollary}

\begin{corollary}
For $k\geq 2$  we have $\dim E_{0,2k}(\Gamma[2])=15(k-1)$. Moreover,
$$\dim E_{0,2k}(\Gamma_1[2])= 6\, [k/2] -1 \quad {\rm and} \quad
\dim E_{0,2k}(\Gamma_0[2])= 2\, [k/2] +1.
$$
For $j\geq 2$ and $k\geq 2$
we have $\dim E_{j,2k}(\Gamma[2])= 15(j/2+k)$. Moreover,
$$\dim E_{j,2k}(\Gamma_1[2])=3\, b_{j,k} \quad {\rm and}
\quad \dim E_{j,2k}(\Gamma_0[2])= b_{j,k}.
$$
\end{corollary}
\end{section}
\begin{section}{Dimension Formulas for Vector-valued Modular Forms on 
$\Gamma_1[2]$}\label{DimFormulasGamma1}
We now give dimension formulas for the space of modular forms and 
cusp forms of weight $(j,k)$ on the group $\Gamma_1[2]$; that is, 
we give the generating functions
$$
\sum_{k\geq 3, \rm odd} \dim S_{j,k}(\Gamma_1[2])\, t^k
\qquad {\rm and} \qquad
\sum_{k\geq 4, \, \rm even} \dim M_{j,k}(\Gamma_1[2]) \, t^k\, .
$$
These results can be deduced from the action of $\frak{S}_6$ on the spaces 
$S_{j,k}(\Gamma[2])$ assuming the conjectures of \cite{BFvdG}.  
Alternatively, they can be obtained by applying the holomorphic Lefschetz 
formula and are then not conditional on the conjectures of \cite{BFvdG}.

We start with the scalar-valued ones ($j=0$). The generating function 
of $R^{\rm ev}(\Gamma_1[2])$ is computed using the ring structure
given in Theorem \ref{RingGamma12} to be
$$
\frac{(1-t^{8})(1-t^{12})}{(1-t^2)(1-t^4)^4(1-t^6)}\, .
$$
\begin{theorem}\label{dimGamma12}
For $j>0$  we have
$$
\sum_{k\geq 3 , \rm odd}  \dim S_{j,k}(\Gamma_1[2]) \, t^k =
\frac{\sum_{i=1}^{12} a_{2i+1} t^{2i+1}}{(1-t^2)(1-t^4)^4(1-t^6)}
$$
with the vector $[a_3,a_5,\ldots,a_{25}]$ of coefficients  $a_i=a_i(j)$ for 
$j \equiv 0 (\bmod \, 4)$ equal to
$1/192$ times
$$
\begin{matrix}
[j^3-18j^2+104j-192, 2j^3+30j^2-104j+192, -2j^3+126j^2-184j+960, &\cr
-7j^3-24j^2+688j-576, -2j^3-252j^2+704j-1344, 8j^3-132j^2-704j+384, &\cr
8j^3+180j^2-1472j+1344, -2j^3+240j^2-400j+384, -7j^3-18j^2+1048j-1536, &\cr
-2j^3-138j^2+680j-576, 2j^3-18j^2-200j+768, j^3+24j^2-160j+192] . &\cr
\end{matrix}
$$
For $j\equiv 2 (\bmod 4)$ the coefficient vector $[a_3,a_5,\ldots,a_{25}]$ 
of the numerator is equal to
$1/192$ times
$$
\begin{matrix}
[j^3-18 j^2+92 j-120, 2 j^3+30 j^2-104 j+72, -2 j^3+126 j^2-136 j+552,\cr
 -7 j^3- 24 j^2+700 j-288, -2 j^3-252 j^2+632 j-432, 8 j^3-132 j^2-752 j+432, \cr
 8 j^3+180 j^2-1424 j+336, -2 j^3+240 j^2-328 j-288, -7 j^3-18 j^2+1036 j-984, \cr
 -2 j^3- 138 j^2+632 j+72, 2 j^3-18 j^2-200 j+648, j^3+24 j^2-148 j] . \cr
\end{matrix}
$$
For even $j\geq 2$ the generating function for even $k$ has the shape
$$
\sum_{k\geq 4 , \rm even} \dim M_{j,k}(\Gamma_1[2])\, t^k =
\frac{\sum_{i=1}^{12} a_{2i+2} \, t^{2i+2}}{(1-t^2)(1-t^4)^4(1-t^6)}
$$ 
with $[a_4,a_6, \ldots, a_{26}]$ for $j\equiv 0 (\bmod 4)$ being equal to
$1/96$ times 
$$
\begin{matrix}
[j^3-3 j^2+140 j, j^3+21 j^2+68 j+96, -3 j^3+45 j^2-372 j+864, \cr
 -4 j^3-36 j^2- 56 j, 2 j^3-114 j^2+592 j-2016, 6 j^3-30 j^2+192 j-960, \cr
 2j^3+102 j^2-656 j+ 1440, -4 j^3+96 j^2-632 j+1920, -3 j^3-27 j^2+324 j-288,
\cr
 j^3-63 j^2+572 j-1440, j^3-3 j^2-28 j, 12 j^2-144 j+384]. \cr
\end{matrix}
$$
For even $j\geq 2$ the generating function 
$\sum_{k\geq 4, \, \rm even} \dim M_{j,k}(\Gamma_1[2])\, t^k $ is of the same shape
with the coefficients $[a_4,a_6, \ldots, a_{26}]$  
for $j\equiv 2 (\bmod 4)$ being equal to $1/96$ times
$$
\begin{matrix}
[{j}^{3}-3\,{j}^{2}+116\,j-228,{j}^{3}+21\,{j}^{2}+68\,j+540, 
-3\,{j}^{3}+45\,{j}^{2}-276\,j+1068, \cr
-4\,{j}^{3}-36\,{j}^{2}-32\,j-816, 2\,{j}^{3}-114\,{j}^{2}+448\,j-1992,
6\,{j}^{3}-30\,{j}^{2}+96\,j-408, \cr
2\,{j}^{3} +102\,{j}^{2}-560\,j+1848,-4\,{j}^{3}+96\,{j}^{2}-488\,j+1296,
-3\,{j}^{3}-27\,{j}^{2}+300\,j-852, \cr
{j}^{3}-63\,{j}^{2}+476\,j-804,{j}^{3}-3\,{j}^{2}-28\,j+156,12\,{j}^{2}
-120\,j+192] . \cr
\end{matrix}
$$
\end{theorem}
\begin{remark}
We observe the following remarkable coincidences. For $k$ even we have:
$$
\begin{aligned}
\dim M_{0,k}(\Gamma[2])&=\dim M_{0,2k}(\Gamma_1[2]), \\
\dim M_{2,k}(\Gamma[2])&=\dim M_{2,2k}(\Gamma_1[2]),\\
\dim S_{2,k+1}(\Gamma[2])&=\dim S_{2,2k+1}(\Gamma_1[2]). \\
\end{aligned}
$$
\end{remark}
To explain two of these dimensional coincidences recall 
that the modular forms of weight $2$ embed the moduli space
${\mathcal A}_2[\Gamma[2]]$ into projective space ${\PP}^4$ and that the closure
of the image is the quartic given by equations 
(\ref{Igusaquartic00}) and (\ref{Igusaquartic0}) 
and is isomorphic to the Satake compactification. The hyperplane bundle of the Igusa
quartic is the anti-canonical bundle; in fact, for a group $\Gamma'\subset {\rm Sp}(4,{\ZZ})$
acting freely on $\frak{H}_2$ the canonical bundle is given by $\lambda^3$ with $\lambda$
the line bundle corresponding to the factor of automorphy $\det(c\tau+d)$ for
a matrix $(a,b;c,d) \in {\rm Sp}(2g,{\ZZ})$. But if the group does not act freely
we have to correct this; in the case at hand, the map 
$\frak{H}_2 \to {\mathcal A}_2[\Gamma[2]]$
is ramified along the ten components of Humbert surface $H_1$ of invariant $1$. The corrected formula is then
$$
K_{{\mathcal A}_2[\Gamma[2]]}= 3\lambda -5\lambda=-2\lambda
$$
where the $5$ comes from $10/2$ with $10$ being the weight of the modular form $\chi_{10}$
defining $H_1$. 
In the preceding section we showed that
${\rm Proj}(\oplus_k M_{0,4k}(\Gamma_1[2]))$ is the Igusa quartic. 
This fits because the map ${\mathcal A}_2[\Gamma[2]]\to
{\mathcal A}_2[\Gamma_1[2]]$ is ramified along the component of the Humbert
surface of invariant $4$ given by the vanishing of $x_5-x_6$. 
Namely, the action of $({\ZZ}/2{\ZZ})^3$ on the $15$ components
of $H_4$ on ${\mathcal A}_2[\Gamma[2]]$ has one orbit of length $1$, 
three orbits of length $2$, and one orbit of length $8$ 
and the orbit of length $1$ is the fixed point locus. 
We thus find
$$
K_{{\mathcal A}_2[\Gamma_1[2]]}= 3\lambda -(5+2)\lambda=-4\lambda.
$$
Note that by the Koecher principle a section of $\lambda^n$ is a modular form.
No holomorphicity conditions at infinity are required.

So the anti-canonical map of ${\mathcal A}_2[\Gamma_1[2]]$ is given by the modular forms of
weight $4$ for $\Gamma_1[2]$. We conclude
$$
M_{0,2k}(\Gamma[2])\cong M_{0,4k}(\Gamma_1[2])\, .
$$

A modular form of weight $(2,2k)$ on $\Gamma[2]$ defines a section
of $T^{\vee}\otimes H^k$ with $T^{\vee}$ the cotangent bundle and $H$ the hyperplane bundle 
on the smooth locus of the Igusa quartic minus the Humbert surface $H_1$. By a local calculation
one sees that such a section extends over $H_1$. We thus see that
$$
M_{2,2k}(\Gamma[2])=H^0({\mathcal A}_2[\Gamma[2]],{\rm Sym}^2{\EE}\otimes \det{\EE}^{2k}),
$$
with ${\EE}$ the Hodge bundle on ${\mathcal A}_2[\Gamma[2]]$ 
(corresponding to the automorphy factor $c\tau+d$).

Similarly, a modular form of weight $(2,4k)$ on $\Gamma_1[2]$ defines a section
of $T^{\vee}\otimes H^k$ with $T^{\vee}$ the cotangent bundle 
and $H$ the hyperplane bundle on the smooth locus of the Igusa quartic 
minus the Humbert surface $H_1$ and one component
of the Humbert surface $H_4$. By a local calculation
one sees that such a section extends to the smooth locus of
${\mathcal A}_2[\Gamma_1[2]]$. 
By the Koecher principle it defines a modular form
holomorphic on all of ${\mathcal A}_2[\Gamma_1[2]]$. 
We thus see that we get an isomorphism
$$
M_{2,2k}(\Gamma[2]) \cong M_{2,4k}(\Gamma_1[2]).
$$  
\end{section}
\begin{section}{Constructing Vector-valued modular forms using brackets}
We now move to constructing vector-valued modular forms.
One way to construct these is by using
so-called Rankin-Cohen brackets. We recall the definition
of the Rankin-Cohen bracket of two Siegel modular forms
and its basic properties.

Let $F$ and $G$ be two modular forms of weight
$(0,k)$ and $(0,l)$ on some subgroup $\Gamma^{\prime}$ of ${\rm Sp}(4,{\ZZ})$.
The Rankin-Cohen bracket of $F$ and $G$
is defined by the formula
$$
[F,G](\tau)=\frac{1}{2\pi i}\left(k \, F\frac{dG}{d\tau}-l\, G\frac{dF}{d\tau}
  \right)(\tau),
$$
where
$$
\frac{dF}{d\tau}(\tau)=
\left(
\begin{matrix}
{\partial F }/{\partial \tau_{11}} & \frac{1}{2}{\partial F}/{\partial \tau_{12}} \\
\frac{1}{2}{\partial F}/{\partial \tau_{12} } & {\partial F}/{\partial \tau_{22}}
\end{matrix}
\right)(\tau)\, .
$$
We refer also to \cite{Satoh, Ibukiyama3,vanDorp}.
(Note that Satoh's definition of the bracket in \cite{Satoh} differs from ours:
$ [F,G]_{({}{\rm Satoh}{})}=-\frac{1}{kl}\,[F,G]$.)
The main fact about this bracket is the following.

\begin{proposition}
If $F_i\in M_{0,k_i}(\Gamma',\chi_i)$, with $\chi_i$ a character
or a multiplicative system on $\Gamma'$,
then $[F_1,F_2]\in M_{2,k_1+k_2}(\Gamma',\chi_1\chi_2)$.
\end{proposition}

\noindent Thus the bracket defines a bilinear operation:
$$
M_{0,k_1}(\Gamma',\chi_1)\times M_{0,k_2}(\Gamma',\chi_2) \to M_{2,k_1+k_2}(\Gamma',\chi_1\chi_2)
$$
satisfying the following properties
\begin{enumerate}
\item[i)] $[F,G]=-[G,F]$
\item[ii)] $F[G,H]+G[H,F]+H[F,G]=0$
\item[iii)] $[FG,G]=G[F,G]$.
\end{enumerate}

\bigskip

We give some examples.
\begin{example}\label{Hij}
As we saw in Section \ref{ThetaSquares}
any pair $(m_i,m_j)$ of odd theta characteristics
with $1\leq i < j \leq 6$ determines a quadratic relation
between squares of theta constants; for example, for the pair $(1,2)$ we have
$$
\vartheta_{1}^2 \vartheta_{3}^2 -
\vartheta_{2}^2\vartheta_{4}^2-\vartheta_{5}^2\vartheta_{6}^2=0.
$$
This implies the following relation between brackets:
$$
[\vartheta_{1}^2 \vartheta_{3}^2,\vartheta_{2}^2\vartheta_{4}^2]
=-[\vartheta_{1}^2 \vartheta_{3}^2,\vartheta_{5}^2\vartheta_{6}^2]=
-[\vartheta_{2}^2\vartheta_{4}^2,\vartheta_{5}^2\vartheta_{6}^2],
$$
and by direct computation we also have
$$
[\vartheta_i^2\vartheta_j^2,\vartheta_k^2\vartheta_l^2]=
8\, \vartheta_i^2  \vartheta_j  \vartheta_k  \vartheta_l^2
[\vartheta_j,\vartheta_k] + 8\,
\vartheta_i  \vartheta_j^2  \vartheta_k^2  \vartheta_l
[\vartheta_i,\vartheta_l]\, .
$$
We thus can associate to a pair $(i,j)$ (of odd theta characteristics)
a form $H_{ij}$ defined by, e.g.\
$$
H_{12}=[\vartheta_{1}^2\vartheta_{3}^2,\vartheta_{2}^2\vartheta_{4}^4]=
-[\vartheta_{1}^2\vartheta_{3}^2,\vartheta_{5}^2\vartheta_{6}^2]=
-[\vartheta_{2}^2\vartheta_{4}^2,\vartheta_{5}^2\vartheta_{6}^2]
$$
(up to an ambiguity of signs)
and in this way using the action of $\frak{S}_6$ we obtain $15$ forms $H_{ij}$
with $1\leq i < j \leq 6$
in $M_{2,4}(\Gamma[2])$.
\end{example}

\begin{example}\label{Hijprime}
In analogy with Example \ref{Hij} we can use the 
relation \ref{Urelation} and the analogues $U_iU_j$ 
from  Construction \ref{constructionU}
to construct $15$ modular forms $H'_{ij}$
$$
H_{12}'=
[U_1U_2,U_3U_4]=[U_1U_2,U_7U_8]=-[U_3U_4,U_7U_8]
\in M_{2,8}(\Gamma_1[2]).
$$
\end{example}

\begin{remark}
We might also consider the brackets $[\Theta[\mu],\Theta[\nu]]$ 
of the theta constants of second order
which lie in $M_{2,1}(\Gamma[2,4])$. 
\end{remark}
\end{section}
\begin{section}{Gradients of Odd Theta Functions}
Another way of constructing vector-valued Siegel modular forms 
is by using the gradients of the six odd theta functions.
The (transposed) gradients
$$
G_i^t= 
(\partial\vartheta_{m_i}/\partial z_1,\partial\vartheta_{m_i}/\partial z_2)
\qquad 1\leq i \leq 6
$$
define sections of the vector bundle ${\EE}\otimes \det({\EE})^{1/2}$ 
on the group $\Gamma[4,8]$ with ${\EE}$ the Hodge bundle, see Section 
\ref{Preliminaries}. 
In other words they are
vector-valued modular forms of weight $(1,1/2)$ 
on the subgroup $\Gamma[4,8]$. 
We identify ${\rm Sym}^j({\EE})$ with the $\frak{S}_j$-invariant 
subbundle of ${\EE}^{\otimes j}$.
We consider expressions
of the form

\begin{equation}\label{One}
{\rm Sym}^j(G_{i_1},\ldots,G_{i_j})
\, \vartheta_{r_1} \cdots \vartheta_{r_l},
\end{equation}
where ${\rm Sym}^j(G_{i_1},\ldots,G_{i_j})$ is the projection of the section of
${\EE}^{\otimes j} \otimes {\det}({\EE})^{j/2}$ 
onto its $\frak{S}_j$-invariant subbundle.
We abbreviate ${\rm Sym}^a(G_i,\ldots,G_i)$ with $G_i$ occurring $a$ times by
${\rm Sym}^a(G_i)$ and ${\rm Sym}^a(G_1,\ldots,G_1,\ldots,G_6, \ldots, G_6)$ with
$G_i$ occuring $a_i$ times and $a=\sum_i a_i$ is abbreviated by
${\rm Sym}^a(G_1^{a_1},\ldots,G_6^{a_6})$.

\begin{remark}\label{symconvention}
If $V\simeq {\CC}^2$ is a ${\CC}$-vector space with ordered basis $e_1,e_2$ then we 
shall use the ordered basis 
$e_1^{\otimes(n-i)} \otimes e_2^{\otimes i}$ for $i=0,\ldots,n$
for  ${\rm Sym}^n(V)$ .
\end{remark}
We ask when the expression   (\ref{One})
gives rise to a vector-valued Siegel modular form on $\Gamma[2]$ 
as opposed to only on $\Gamma[4,8]$.

We shall write the $j+l$ theta characteristics occurring in 
(\ref{One}) as
a $4\times (j+l)$-matrix $M$ where each characteristic is written as a length $4$
column. We first write the odd theta characteristics, then the even ones.
A similar problem involving polynomials in the theta constants was considered by Igusa and Salvati Manni,
\cite[Corollary of Theorem 5]{Igusa2} and \cite [Equation 20]{SM}. One finds in an analogous manner:

\begin{proposition}\label{descendlevel2}
The expression (\ref{One}) gives a modular form in $M_{j,(l+j)/2}(\Gamma[2])$
if and only if the matrix $M$ satisfies $M\cdot M^t\equiv 0 \bmod 4$. If we write each of the $j+l$ characteristics in $M$ as
$\left(\epsilon_1^{(i)} \epsilon_2^{(i)} \epsilon_3^{(i)} \epsilon_4^{(i)}\right)^t$, then these conditions can be written equivalently as
\begin{enumerate}
\item[i)]  $\sum_{i=1}^{j+l} \epsilon_a^{(i)} \equiv 0 \bmod 4$ for any $1\le a\le 4$,
\item[ii)]  $\sum_{i=1}^{j+l} \epsilon_a^{(i)}\epsilon_b^{(i)}\equiv 0 \bmod2$ for any $1\le a<b\le 4$.
\end{enumerate}
\end{proposition}

We also want to know the action of $\frak{S}_6$. For this we have the following lemma.

\begin{lemma}
The action of $X$ (resp.\ $Y$) on the gradients $G_i$
for $i=1,\ldots,6$ of the odd theta functions
is given by
$$ \rho(X)=\left(
\begin{matrix}
0 & 1 & 0 & 0 & 0 & 0\\
1 & 0 & 0 & 0 & 0 & 0 \\
0 & 0 & \zeta & 0 & 0 & 0 \\
0 & 0 & 0 & \zeta & 0 & 0 \\
0 & 0 & 0 & 0 & \zeta & 0 \\
0 & 0 & 0 & 0 & 0 & \zeta \\
\end{matrix}\right)
\qquad
\rho(Y)=
\left(\begin{matrix}
0 & 0 & 0 & 0 & 0 & \zeta \\
\zeta^6 & 0 & 0 & 0 & 0 & 0 \\
0 & \zeta^7 & 0 & 0 & 0 & 0 \\
0 & 0 & \zeta^6 & 0 & 0 & 0 \\
0 & 0 & 0 & 1 & 0 & 0 \\
0 & 0 & 0 & 0 & 1 & 0 \\
\end{matrix} \right)\, .
$$
\end{lemma}
We give examples of modular forms constructed in this way;
a number of these will be used later.
\begin{example}\label{2-4-forms}
We take
$$
F={\rm Sym}^6(G_1,\ldots,G_6) \, .
$$
This is a modular form of weight $(6,3)$ on $\Gamma[2]$; 
it is $\frak{S}_6$-anti-invariant and necessarily a
cusp form. The space $S_{6,3}(\Gamma[2])$ is $1$-dimensional and
generated by $F$. The product $F\, \chi_5$ generates the $1$-dimensional
space $S_{6,8}(\Gamma)$ of level $1$. A form in this space was constructed
by Ibukiyama in \cite{Ibukiyama1}, cf.\ \cite{FvdG}. He used theta 
functions with pluri-harmonic coefficients.
\end{example}
\begin{example}\label{Gij}
We consider
$$
G_{12}= {\rm Sym}^2(G_1,G_2)
\, \vartheta_{1} \cdots \vartheta_{6} \quad \in M_{2,4}(\Gamma[2]).
$$
We can vary this construction by taking
for any pair $G_i,G_j$ of different gradients of theta functions with odd characteristics
 the six even $n_j$ that are complementary to the four
that correspond to a pair of odd ones via Lemma \ref{thetacharlemma}.
The modular forms constructed in this way form a representation of $\frak{S}_6$ that is $s[3,1^3]+s[2,1^4]$.

\begin{remark}\label{boundaryvsoddpair}
The restriction of $G_{12}$ to the $1$-dimensional 
boundary components of ${\mathcal A}_2[\Gamma[2]]^*$ vanishes on $14$ of those, while it is a 
multiple of the unique cusp form 
$(\vartheta_{00}\vartheta_{01}\vartheta_{10})^4$ on $\Gamma_0(2)$ 
times the vector $(1,0,0)$ on the remaining boundary component.
(Note that $(1,0,0)$ is a highest weight vector of our representation.) 
This gives the correspondence between the fifteen boundary components
and unordered pairs of odd theta characteristics. See also Lemma \ref{geometryvsthetachar}.
\end{remark}
As a variation, consider
$$
G_{11}= {\rm Sym}^2(G_1)
\, \vartheta_{1}^2 \vartheta_{4}^2\vartheta_{6}^2 \, .
$$
Also this is a modular form of weight $(2,4)$ on $\Gamma[2]$ and its orbit 
under $\frak{S}_6$ spans the representation $s[2,1^4]$, see Example 
\ref{ffourtuple} below.
\end{example}
\begin{example}
We have
$$
{\rm Sym}^2(G_1,G_2)
\, \vartheta_{2}^2 \vartheta_{4}^2 \vartheta_{7}
\vartheta_{8}\vartheta_{9}\vartheta_{10} \in S_{2,5}(\Gamma[2])\, .
$$
For fixed $(i,j)$ there are three choices for the factor
$\vartheta_a^2 \vartheta^2_b \vartheta_c\vartheta_d\vartheta_e \vartheta_f$
so that ${\rm Sym}^2(G_i,G_j)$ times this factor
is a modular form of weight $(2,5) $ on $\Gamma[2]$. 
Formally we find a representation $s[3^2]\oplus s[3,2,1]\oplus s[2^2,1^2]$, 
but we know $S_{2,5}(\Gamma[2])=s[2^2,1^2]$. We thus find relations, 
for example
$$
\vartheta_1\vartheta_7\vartheta_{10}\, {\rm{Sym}}^2(G_1,G_2)-
\vartheta_4\vartheta_5\vartheta_{9}\, {\rm{Sym}}^2(G_1,G_4)+
\vartheta_2\vartheta_6\vartheta_{8}\, {\rm{Sym}}^2(G_1,G_6)=0
$$
This identity shows that
$$ 
{\rm{Sym}}^2(G_1,G_2)\wedge {\rm{Sym}}^2(G_1,G_4)\wedge {\rm{Sym}}^2(G_1,G_6)=0
$$
which gives using the $\frak{S}_6$-action
\begin{equation}\label{Two}
 {\rm{Sym}}^2(G_i,G_j)\wedge {\rm{Sym}}^2(G_i,G_k)\wedge 
{\rm{Sym}}^2(G_i,G_l)=0. 
\end{equation}

\end{example}

\begin{example} We have
$$
{\rm Sym}^2(G_1) \,
\vartheta_{2}^2 \vartheta_{4}^2\vartheta_{5}^2\vartheta_{9}^2\vartheta_{10}^2 
\in M_{2,6} (\Gamma[2]).
$$
We can build $72$ modular forms of this type, $12$ for each ${\rm Sym}^2(G_i)$
and one can show that these generate a representation $s[3,2,1]+s[3,1^3]$.

Similarly, we have
$$
{\rm Sym}^2(G_1,G_2)
\, \vartheta_{7}^3
\vartheta_{8}^3\vartheta_{9}^3\vartheta_{10}^3 \in S_{2,7}(\Gamma[2])\, .
$$
Finally,
$$
{\rm Sym}^4(G_1) \in M_{4,2}(\Gamma[2]), \quad
{\rm Sym}^4(G_1,G_2,G_3,G_4) \,
\vartheta_{5}\vartheta_{6}\vartheta_{7}\vartheta_{8}
\in M_{4,4}(\Gamma[2]),
$$
and
$$
{\rm Sym}^4(G_1,G_1,G_2,G_2) \,
\vartheta_i^2 \vartheta_j^2 \in M_{4,4}(\Gamma[2])
\quad \hbox{\rm for $(i,j)=(1,3), (2,4)$ and $(5,6)$} \, .
$$
\end{example}
\begin{example}\label{ffourtuple}
For a given $1\leq i \leq 6$ there are ten triples $(a,b,c)$ such that
$$
f[i;a,b,c]= {\rm Sym}^2(G_i) \vartheta_a^2 \vartheta_b^2 \vartheta_c^2 
$$
lies in $M_{2,4}(\Gamma[2])$. The relations (5.1)
in Section \ref{ThetaSquares}
imply obvious relations among these forms and using these we are
reduced to four different triples 
for each $i$; e.g.\ for $i=1$ we have the four forms
$f[1; 1,2,5], f[1;1,4,6], f[1;2,3,6], f[1;3,4,5]$. In total we get
$24$ forms that form a $\frak{S}_6$-representation 
$s[3,1^3]+s[2^2,1^2]+s[2,1^4]$. But since $M_{2,4}(\Gamma[2])=s[3,1^3]+s[2,1^4]$
we have a space $s[2^2,1^2]$ of relations and these relations are 
generated by the $\frak{S}_6$-orbit of the relation
$$
f[3;2,3,8]-f[1;3,4,5]+f[2;3,4,6]-f[6;3,4,10] =0 \, .
$$
\end{example}

\begin{example}
The form 
${\rm{Sym}}^6(G_1^3,G_2^3)
\vartheta_{7}\vartheta_{8}\vartheta_{9}\vartheta_{10}$ is a cusp form of weight
$(6,5)$ on $\Gamma[2]$ and has an orbit of fifteen elements, 
generating formally a representation $s[5,1]+s[4,1^2]$; 
its contribution to $S_{6,5}(\Gamma[2])$
is $s[4,1^2]$, thus giving a $s[5,1]$ of relations.
\end{example}
\begin{example} 
We have in $M_{8,4}(\Gamma[2])$ the form
${\rm{Sym}}^8(G_1^4,G_2^4)
\vartheta_{7}\vartheta_{8}\vartheta_{9}\vartheta_{10}$. Its 
$\frak{S}_6$-orbit generates formally the representation $s[6]+s[5,1]+s[4,2]$.
We also have the six forms ${\rm Sym}^8(G_i^8) \in M_{8,4}(\Gamma[2])$
that generate a representation $s[6]+s[5,1]$ and 
$\sum_i {\rm Sym}^8(G_i^8) \in M_{8,4}(\Gamma)$ and it is not zero since
its image under the Siegel operator is 
$$
2\pi^8 (\vartheta_{00}^8\vartheta^8_{01} \vartheta^8_{11})(\tau_{11}) (1,0,\ldots,0)^t\, . 
$$
\end{example}

\begin{example}\label{s321genofS45}
The form ${\rm Sym}^4(G_1,G_2,G_3^2)\vartheta_7^3 \vartheta_8 \vartheta_9
\vartheta_{10}$ is a cusp form in $S_{4,5}(\Gamma[2])$ that generates 
a $s[3,2,1]$ representation in this space.
\end{example} 
\end{section}
\begin{section}{Identities between Gradients of Odd Theta Functions and Even Theta Constants}\label{Identities}
The fact that we have two ways of constructing modular forms and
that we can decompose the spaces where these forms live as $\frak{S}_6$-representations, easily leads to
many identities. In this section we give a number of such identities, and in
some sense these can be seen as generalizations
of Jacobi's famous derivative formula for genus $1$
$$
\frac{\partial \vartheta_{11}}{\partial z}|_{z=0}=
-\pi \, \vartheta_{00}\vartheta_{01}\vartheta_{10}
$$
to vector-valued modular forms of genus $2$. 
For generalizations to scalar-valued modular forms we refer to
\cite{Fay, Igusa3, Grant, grsm}.

To motivate the fact that such an identity for vector-valued modular exists, we recall some of the results of  \cite{grsm}. Indeed, consider Riemann's bilinear addition formula (\ref{bilinear}) for the case when the characteristic $\left[\begin{smallmatrix}\mu\\ \nu\end{smallmatrix}\right]$ is odd. Differentiating this identity with respect to $z_i$ and $z_j$, and evaluating at $z=0$, one obtains \cite[Lemma 4]{grsm}:
$$
 2\left.\frac{\partial\vartheta_{\left[\begin{smallmatrix}\mu  \\ \nu\end{smallmatrix}\right]}(\tau,z)}{\partial z_i}\cdot \frac{\partial\vartheta_{\left[\begin{smallmatrix}\mu  \\ \nu\end{smallmatrix}\right]}(\tau,z)}{\partial z_i}\right|_{z=0}=\left.
 \sum\limits_{\sigma\in (\ZZ/2\ZZ)^2} (-1)^{\sigma\cdot\nu}\Theta[\sigma](\tau)\frac{\partial\Theta[\sigma+\mu](\tau,z)}{\partial z_i\partial z_j}\right|_{z=0}.
$$
Using the heat equation for the theta function, the second order 
$z$-derivative in the right-hand side can be rewritten as a constant factor times the $\tau$-derivative. 
By summing over $\nu$ with coefficient $(-1)^{\nu \cdot \alpha}$
the Rankin-Cohen brackets on the right can be written as linear
combinations of expressions on the left: the result is \cite[Lemma 5]{grsm},
an identity between a quadratic expression in the gradients and a
combination of Rankin-Cohen  brackets.

\smallskip

We now give an identity between the two types of vector-valued modular forms
that we constructed. To rule out any ambiguities of notation we fix
the coordinates by putting
$$
{\rm{Sym}}^2(G_i,G_j)=
\left[
\begin{smallmatrix}
G_i^{(1)}G_j^{(1)} \\
G_i^{(1)}G_j^{(2)}+G_i^{(2)}G_j^{(1)} \\
G_i^{(2)}G_j^{(2)}
\end{smallmatrix}
\right]
\quad {\rm{for}}
\quad
G_i=
\left[
\begin{smallmatrix}
G_i^{(1)} \\
G_i^{(2)}
\end{smallmatrix}
\right] ;
$$
we also write the bracket, which is given as  $2\times 2$ matrix-valued, 
as vector-valued via
$$
[f,g]=
\left[
\begin{smallmatrix}
a & b\\
b & c
\end{smallmatrix}
\right] \mapsto
\left[
\begin{smallmatrix}
a \\
2b\\
c
\end{smallmatrix}
\right].
$$

\begin{lemma}\label{G-H-identity}
The following identity holds for modular forms in $M_{2,2}(\Gamma[2,4])$:
$$
  {\rm Sym}^2(G_1,G_1)\vartheta_{1}^2=2\, \pi^2  
\left([\vartheta_{2}^2,\vartheta_{5}^2]+
  [\vartheta_{4}^2,\vartheta_{6}^2]+[\vartheta_{8}^2,\vartheta_{9}^2]\right);
$$
it yields similar identities under the action of $\frak{S}_6$.
Moreover, for all $1\leq i < j \leq 6$ we have the identity
$G_{ij}= -\pi^2 \, H_{ij}$ in $M_{2,4}(\Gamma[2])$. (Here the $G_{ij}$
are defined in Example \ref{Gij} and the $H_{ij}$ in Example \ref{Hij}.)
\end{lemma}
For example, for $(i,j)=(1,2)$ we have
$$
{\rm Sym}^2(G_1,G_2)
\vartheta_{1}\vartheta_{2}\vartheta_{3}\vartheta_{4}\vartheta_{5}\vartheta_{6} =-\pi^2 \,
[\vartheta_{1}^2\vartheta_{3}^2,\vartheta_{2}^2\vartheta_{4}^2] \, .
$$

\begin{proof}
The space $M_{2,4}(\Gamma[2])$ is generated by the fifteen forms $G_{ij}$, 
because we know that $\dim M_{2,4}(\Gamma[2])=15$ and the fifteen $G_{ij}$
are linearly independent by Remark \ref{boundaryvsoddpair}. 
By comparing Fourier coefficients we then find the relation 
$$
f[1;1,2,5]=G_{12}+G_{15}=-\pi^2 \, (H_{12}+H_{15})
$$
with $f[1;1,2,5]$ defined in Example \ref{ffourtuple}, 
$H_{12}=-[\vartheta_2^2\vartheta_4^2,\vartheta_5^2\vartheta_6^2]$ and
$H_{15}=-[\vartheta_2^2\vartheta_8^2,\vartheta_5^2\vartheta_9^2]$.
Dividing by $\vartheta_2^2 \vartheta_5^2$ gives the desired identity
in $M_{2,2}(\Gamma[2,4])$. 
The second identity also follows by comparing Fourier coefficients.
\end{proof}

We end with a question:
\begin{question}
Is the algebra $\oplus_{j,k} M_{j,k}(\Gamma[4,8])$ generated over 
the ring $\oplus_{k \in {\ZZ}} M_{0,k}(\Gamma[4,8])$ by the 
$[\vartheta_a,\vartheta_b]$?
Is the algebra $\oplus_{j,k} M_{j,k}(\Gamma[2,4])$ generated over the ring $\oplus_{k \in {\ZZ}} M_{0,k}(\Gamma[2,4])$ by the brackets $[\Theta[\mu],\Theta[\nu]]$?
\end{question}
\end{section}
\begin{section}{Wedge Products}\label{wedges}
In this section we calculate some triple wedge products of modular forms of weight $(2,4)$ that give information
about the vanishing loci of these modular forms.
We start by looking at a triple wedge product of the form
$$
{\rm Sym}^2(G_{i_1},G_{j_1})\wedge  {\rm Sym}^2(G_{i_2},G_{j_2})\wedge {\rm Sym}^2(G_{i_3},G_{j_3}),
$$
where we recall that the $G_i$'s are the gradients of odd theta functions.
A direct computation, by writing out the summands of this wedge product, and
matching the individual terms, shows that it is equal to the sum of two triple
products of Jacobian determinants, for example to
$$D(i_1,i_2)\cdot D(j_1,i_3)\cdot D(j_2,j_3)+D(i_1,j_3)\cdot D(j_1,j_2)\cdot D(i_2,i_3),$$
where $D(a,b)=G_a\wedge G_b$ are the usual Jacobian nullwerte.
By a generalized Jacobi's derivative
formula (see \cite{Igusa3,grsm}) each such Jacobian determinant is a product of four theta constants with
characteristics, and thus we obtain an expression for such a triple wedge product
as an explicit  degree $12$ polynomial in theta constants with characteristics. 
\begin{proposition}
We have the following identities in $S_{0,15}(\Gamma[2])$:
$$
G_{12}\wedge G_{34} \wedge G_{56}= 
\pi^6 \, \chi_5
\vartheta_{1}^4
\vartheta_{2}^4
\vartheta_{3}^4
\vartheta_{4}^4
(\vartheta_{5}^4- \vartheta_{6}^4)\,= -\pi^6 \chi_7 x_1x_2x_3x_4 ;
$$
$$
G_{12}\wedge G_{13} \wedge G_{45}=\pi^6\frac{\chi_5^3\, \vartheta_1^2 \, 
\vartheta_4^2\, \vartheta_6^2}
{\vartheta_2^2\, \vartheta_7^2\, \vartheta_9^2}
$$
and
$$
G_{12}\wedge G_{13}\wedge G_{14}=0\, .
$$
\end{proposition}

Since we know that the zero divisors of the $\vartheta_i^4$ are the components
of $H_1$ we deduce:
\begin{corollary}
The modular form $G_{12}$ does not vanish outside the Humbert surface~$H_1$.
\end{corollary}
\smallskip
A calculation using the Fourier-Jacobi expansion of the theta constants 
shows that $G_{12}$ vanishes on $6$ components of $H_1$
and does not vanish identically on the other $4$ as one sees by using the group action. 
On the component given by $\tau_{12}=0$ it equals 
$$
\pi^2 \, \vartheta_{00}^4\vartheta_{01}^4(\tau_{11}) \otimes 
\vartheta_{00}^4\vartheta_{01}^4\vartheta_{10}^4(\tau_{22}) \cdot 
\left(\begin{smallmatrix} 0 \\ 0 \\  1 \\ \end{smallmatrix} \right) \, .
$$

\end{section}
\begin{section}{Bounds on the Module Generators}\label{ModuleStructure}
We now turn to the module structure of the $R^{\rm ev}$-modules
$$
{\mathcal M}_j^{\epsilon}=\oplus_{k\equiv \epsilon
\bmod  2} M_{j,k}(\Gamma[2]) \qquad \hbox{ for $\epsilon= 0$ and $\epsilon=1$}
$$
and similar ones where the $M_{j,k}$ are replaced by spaces of cusp forms 
$S_{j,k}$.
First note that 
$R^{\rm ev}=\oplus_{k \, \rm even} H^0({\mathcal A}_2[\Gamma[2]],L^k)$ 
with $L$ the determinant of the Hodge bundle. 
By the Koecher principle we
have 
$H^0({\mathcal A}_2[\Gamma[2]],L^k)
=H^0(\tilde{{\mathcal A}}_2[\Gamma[2]],L^k)$ 
with $\tilde{\mathcal A}_2[\Gamma[2]]$ the standard toroidal compactification. 
Similarly, we have 
$$
M_{j,k}(\Gamma[2])=H^0({\mathcal A}_2[\Gamma[2]], {\rm Sym}^j({\EE})\otimes L^k)
=H^0(\tilde{{\mathcal A}}_2[\Gamma[2]], {\rm Sym}^j({\EE})\otimes L^k).
$$
Note that the Hodge bundle ${\EE}$ extends to the toroidal compactification and $L$ extends to
the Satake compactification.
The $R^{\rm ev}$-modules ${\mathcal M}_j^{\epsilon}$
for $\epsilon= 0, 1$ are of the form 
$$
\oplus_k \, H^0(X, F\otimes L^k)
$$
over $R^{\rm ev}=\oplus_k H^0(X,L^k)$ and as $L$ is an ample line bundle
on the Satake compactification (but only nef on $\tilde{\mathcal A}_2[\Gamma[2]]$) they
are finitely generated (cf.\ \cite{Lazarsfeld}, p.\ 98--100). 
Recall that we have 
$$
R^{\rm ev}= {\CC}[u_0,\ldots,u_4]/(f)
$$
with $f$ a homogeneous polynomial of degree $4$ in the $u_i$. We set
$$T={\CC}[u_0,\ldots,u_4].
$$ 
The group $\mathfrak{S}_6$ 
acts on it; the action on the space of homogeneous polynomials of degree $1$ 
is the irreducible representation $s[2^3]$. So we may write $T$ as
the symmetric algebra 
${\rm Sym}^{\ast}s[2^3]$ and $R^{\rm ev}$ as the (virtual) $T$-module 

\begin{equation}\label{virtualT}
R^{\rm ev} = T-T(-4)\, .
\end{equation}

We can view ${\mathcal M}_j^{\epsilon}$ as a $T$-module and 
then a theorem of Hilbert (\cite{Peeva}, p.\ 56) 
tells us that it is of finite presentation. 
But because of (\ref{virtualT}) it is then not of
finite presentation when viewed as a module over $R^{\rm ev}$ 
and we get (infinite) periodicity.
The fact that it is also a $R^{\rm ev}$-module implies 
that its Euler characteristic is zero as stated in 
Remark \ref{Euler=0}.

By what was observed in Section \ref{DimFormulasGamma1} the situation is similar for modules of modular forms on $\Gamma_1[2]$ over
the ring $\oplus_k M_{0,4k}(\Gamma_1[2])$. 
Again we can view the modules as modules over a 
polynomial ring in five variables (with a non-modular $\frak{S}_6$-action).

In order to determine the structure of these modules it is useful 
to have bounds on the weight of generators and relations of these modules. 
Here the notion of Castelnuovo-Mumford 
regularity applies. We refer to \cite{Lazarsfeld}, I, pp.\ 90 ff. 
Let $F_{r,s}$ be the vector bundle ${\rm Sym}^r (\EE) \otimes \det({\EE})^s$
and $F_{r,s}^{\prime}={\rm Sym}^r (\EE) \otimes \det({\EE})^s \otimes O(-D)$, 
where $D$ is the divisor at infinity of the toroidal compactification 
$X=\tilde{\mathcal A}_2[\Gamma[2]]$. So the sections of $F_{j,k}$
are the modular forms of weight $(j,k)$ on $\Gamma[2]$ and those of $F^{\prime}_{j,k}$ the
cusp forms of weight $(j,k)$ on $\Gamma[2]$. 
We consider the modules
$$
{\mathcal M}_j^{\epsilon}= \oplus_{k\equiv \epsilon \bmod 2} \Gamma(X,F_{j,k})
\qquad {\rm and} \qquad
\Sigma_j^{\epsilon} = \oplus_{k\equiv \epsilon \bmod 2} \Gamma(X,F_{j,k}^{\prime}).
$$
Here we can consider these as modules over the polynomial ring ${\CC}[u_0,\ldots,u_4]$. 

Recall that one calls a vector bundle $F$ $m$-regular in the sense 
of Castelnuovo-Mumford
with respect to an ample line bundle ${\mathcal L}$ if
$$
H^i(X,F\otimes {\mathcal L}^{m-i})=0 \quad \hbox{\rm for $i>0$}.
$$
The relevance of this notion is that it implies
\begin{enumerate}
\item{} $F$ is generated by its global sections
\item For $k\geq 0$ the natural maps 
$$
H^0(X,{\mathcal L}^{k}) \otimes H^0(X, F\otimes {\mathcal L}^{m}) 
\to H^0(X, F\otimes {\mathcal L}^{m+k})
$$
are surjective.
\end{enumerate}

In our case we will apply this to the case 
${\mathcal L}=\det({\EE})^2$ and $F=F_{j,r}$ or $F^{\prime}_{j,r}$
for some $j$ and small $r$. However, since ${\mathcal L}$ is ample only on
${\mathcal A}_2[\Gamma[2]]$ and nef on $\tilde{\mathcal A}_2[\Gamma[2]]$
one needs to adapt these notions slightly. The main point is that
by the Koecher Principle the sections of $F_{j,r}$ on
${\mathcal A}_2[\Gamma[2]]$ automatically extend to sections over all
of $\tilde{\mathcal A}_2[\Gamma[2]]$. The cohomological mechanism (cf.\
\cite{Lazarsfeld}, Vol.\ I, proof of Thm.\ 1.8.3.)
thus works the same way.

Note that we have Serre duality
$$
H^i(X,F_{j,k})^{\vee}=H^{3-i}(X, F_{j,3-j-k}\otimes O(D)), \qquad
H^i(X,F_{j,k}^{\prime})^{\vee}=H^{3-i}(X, F_{j,3-j-k}).
$$
So a necessary condition for $F_{j,0}^{\prime}$ (resp.\ $F_{j,1}^{\prime}$) 
being $m$-regular with respect to $\det({\EE})^2$ is that 
$H^3(X,F_{j,0}^{\prime}\otimes \det({\EE})^{2m-6})=0$ 
and by Serre duality this gives
$$
M_{j,9-2m-j}(\Gamma[2])=(0) \qquad
\hbox{\rm (resp.\ $M_{j,8-2m-j}(\Gamma[2])=(0)$).}
$$
So the dimension formulas give restrictions on the regularity. 
A bound on the regularity gives bounds on the weights
of generators, see 
for example \cite{Lazarsfeld}, Vol.\ I, Thm.\ 1.8.26.

We give here two results on the regularity.

\begin{proposition}\label{regularity2-1}
The vector bundle $F_{2,1}^{\prime}=
{\rm Sym}^2({\EE}) \otimes \det({\EE})\otimes O(-D)$ 
is $3$-regular with respect to $\det({\EE})^2$.
\end{proposition}
\begin{proof}
We have to prove the vanishing of $H^1(X,F_{2,5}^{\prime})$, $H^2(X,F_{2,3}^{\prime})$ and $H^3(X,F_{2,1}^{\prime})$.
By Serre duality the vanishing of the $H^3$ comes down to the non-existence of modular forms of weight $(2,0)$.
The cohomology $H^1(X,F^{\prime}_{2,5})$ occurs as the first step of the Hodge filtration of the
compactly supported cohomology $H_c^4({\mathcal A}_2[\Gamma[2]],\VV_{4,2})$, see \cite{Getzler,BFvdG2}. 
Here ${\VV}_{k,l}$ is a local system defined in \cite{BFvdG2}. In fact, the Hodge filtration 
on $H^i_c({\mathcal A}_2[\Gamma[2]],{\VV}_{k,l})$ has the steps 
$H^{i}(X,F^{\prime}_{k-l,-k})$, $H^{i-1}(X,F^{\prime}_{k+l+2,-k})$, $H^{i-2}(X,F^{\prime}_{k+l+2,1-l})$
and $H^{i-3}(X,F^{\prime}_{k-l,l+3})$. 
Since ${\VV}_{4,2}$ is a regular
local system the $H^4_c$ consists only of Eisenstein cohomology 
by results of Saper and Faltings, cf.\ \cite{Faltings,Saper}. 
By Eisenstein cohomology we mean the kernel of the natural map 
$H^{\bullet}_c({\mathcal A}_2[\Gamma[2]],\VV_{k,l})
\to H^{\bullet}({\mathcal A}_2[\Gamma[2]],\VV_{k,l})$.
This cohomology is known by results of Harder (see \cite{Harder,vdG2}) and does not contain a contribution of this type. 
Harder dealt with the case of level $1$, but the results can easily be extended to the case of level $2$, cf.\ also 
\cite{vdG2, BFvdG}. The cohomology $H^2(X,F_{2,3}^{\prime})$
occurs in $H^5_c({\mathcal A}_2[\Gamma[2]],\VV_{2,0})$ 
and again in the Eisenstein cohomology.
But this contribution is zero, see \cite{Harder, BFvdG2}.
\end{proof}

\begin{proposition}\label{regularity2-0}
The vector bundle $F_{2,0}={\rm Sym}^2({\EE})$ is $3$-regular 
with respect to $\det({\EE})^2$.
\end{proposition}
\begin{proof}
Now we have to show the vanishing of $H^1(X,F_{2,4})$, $H^2(X,F_{2,2})$ and $H^3(X,F_{2,0})$.
Instead of compactly supported cohomology we 
now look at the Hodge filtration of 
$H^i({\mathcal A}_2[\Gamma[2]],{\VV}_{k,l})$ with the  steps
$H^{i}(X,F_{k-l,-k})$, $H^{i-1}(X,F_{k+l+2,-k})$, $H^{i-2}(X,F_{k+l+2,1-l})$ and $H^{i-3}(X,F_{k-l,l+3})$.
The space $H^1(X,F_{2,4})$ occurs in $H^4({\mathcal A}_2[\Gamma[2]],{\VV}_{3,1})$. 
Again this is Eisenstein cohomology, 
i.e.\ occurs in the cokernel of the natural map 
$H^{\bullet}_c \to H^{\bullet}$ and it vanishes. 
Similarly, $H^2(X,F_{2,2})$ occurs in 
$H^5({\mathcal A}_2[\Gamma[2]],{\VV}_{3,1})$.
For $H^3(X,F_{2,0})$ we take the Serre dual $H^0(X,F_{2,1}\otimes O(D))$. But any section of this
on ${\mathcal A}_2[\Gamma[2]]$ extends by the Koecher principle to a modular form of weight $(2,1)$
and thus vanishes; indeed, it is automatically a cusp form and if 
$S_{2,1}(\Gamma[2])\neq (0)$ we land by multiplying 
with $\psi_4 \in M_{0,4}(\Gamma)$
in $S_{2,5}(\Gamma[2])$ which is the $\frak{S}_6$-representation $s[2^2,1^2]$, and hence $S_{2,1}(\Gamma[2])$ is a $s[2^2,1^2]$ too;
then $\chi_5 S_{2,1}(\Gamma[2])$ is a $s[4,2]$; but $S_{2,6}(\Gamma[2])$ does not contain a $s[4,2]$.
For another argument see the proof of Lemma \ref{S21S23}.
\end{proof}
\end{section}
\begin{section}{The Module ${\Sigma}_2(\Gamma[2])$}
In this section we determine the structure of the $R^{\rm ev}$-module of cusp forms
$$
\Sigma_2 = \Sigma_2(\Gamma[2])=
\oplus_{k=0, k \, {\rm odd}}^{\infty} S_{2,k}(\Gamma[2])\, .
$$
We construct 
modular forms $\Phi_i$ for $i=1,\ldots, 10$ in the first 
non-zero summand $S_{2,5}(\Gamma[2])$ of $\Sigma_2$ 
by setting
$$
\Phi_i=  [x_i,\chi_5]/x_i=
 [\vartheta_i^4,\chi_5]/\vartheta_i^4 = 4 \, [\vartheta_i, \vartheta_1 \ldots
\hat{\vartheta}_i \ldots \vartheta_{10}] \, .
$$
\begin{remark}
Some Fourier coefficients of $\Phi_1$ are given in Section  \ref{FourierExpansions}.
Eigenvalues of Hecke operators acting on the space $S_{2,5}(\Gamma[2])$
were calculated in \cite{BFvdG}.
\end{remark}
The main result in this section is the following.

\begin{theorem}\label{ThmSigma2}
The ten modular forms $\Phi_i$ generate the $R^{\rm ev}$-module 
$\Sigma_2(\Gamma[2])$.
\end{theorem}

\begin{remark}
As a module over the polynomial ring $T$ in five variables the module
$\Sigma_2(\Gamma[2])$ is generated by the $\Phi_i$ with relations of type
$s[1^6]$ in weight $(2,5)$, $s[5,1]$ in weight $(2,7)$ and type $s[3^2]$ in weight $(2,9)$ and a syzygy in weight $(2,11)$. But over the ring of modular forms
of even weight, which we recall is $T-T(-4)$, 
it is not of finite presentation and this pattern of (virtual) 
generators and relations 
is repeated indefinitely (modulo~$8$).
\end{remark}

Before giving the proof we sketch its structure. We can
calculate the action of $\frak{S}_6$ on the spaces $S_{2,k}(\Gamma[2])$ of modular forms
(assuming the conjectures of \cite{BFvdG}) for small $k$. This suggests that
there are $9$ generators in weight $(2,5)$. We construct these forms
and show (directly, not using the conjectures of \cite{BFvdG})
that these forms generate $S_{2,k}(\Gamma[2])$ over the ring of 
even weight scalar-valued modular forms for $k \leq 13$. We also calculate
the relations up to weight $13$.
We then use the bound on the Castelnuovo-Mumford regularity 
of the module $\Sigma_2$ over $T$ and this shows that there 
are no further relations between
our purported generators and a comparison of generating functions shows
that we found the whole module $\Sigma_2$.
Thus the result is independent of the conjectures in \cite{BFvdG}.

We begin by giving a table for the decomposition of $S_{2,k}(\Gamma[2])$ as a
$\mathfrak{S}_6$-representation for small odd $k$. 
At the end of this section we shall prove that $S_{2,k}(\Gamma[2])=(0)$
for $k=1$ and $k=3$.

\begin{footnotesize}
\smallskip
\vbox{
\bigskip\centerline{\def\quad{\hskip 0.6em\relax}
\def\quod{\hskip 0.5em\relax }
\vbox{\offinterlineskip
\hrule
\halign{&\vrule#&\strut\quod\hfil#\quad\cr
height2pt&\omit&&\omit&&\omit&&\omit&&\omit&&\omit&&\omit&&\omit&&\omit&&\omit && \omit && \omit &\cr
&$S_{2,k}\backslash P $ && $[6]$ && $[5,1]$ && $[4,2]$ && $[4,1^2]$ && $[3^2]$ && $[3,2,1]$ && $[3,1^3]$ && $[2^3]$ && $[2^2,1^2]$ && $[2,1^4]$&& $[1^6]$ &\cr
\noalign{\hrule}
&$S_{2,5}$ && $0$ && $0$ && $0$ && $0$ && $0$ && $0$ && $0$ && $0$ && $1$ && $0$ && $0$ &\cr
&$S_{2,7}$ && $0$ && $0$ && $0$ && $1$ && $1$ && $1$ && $0$ && $0$ && $1$ && $0$ && $0$ &\cr
&$S_{2,9}$ && $0$ && $1$ && $0$ && $2$ && $1$ && $2$ && $1$ && $0$ && $3$ && $1$ && $1$ &\cr
&$S_{2,11}$ && $0$ && $2$ && $1$ && $4$ && $3$ && $5$ && $2$ && $0$ && $4$ && $1$ && $1$ &\cr
&$S_{2,13}$ && $0$ && $2$ && $2$ && $6$ && $5$ && $9$ && $4$ && $1$ && $8$ && $2$ && $1$ &\cr
} \hrule}
}}
\end{footnotesize}

Besides the $\Phi_i$ we can construct the weight $(2,5)$ forms
$$
\phi_{ij}=( \prod_{k\neq i,j} \vartheta_k )  
[\vartheta_{i},\vartheta_{j}] \qquad (1\leq i, j \leq 10)\, .
$$
To check that the $\phi_{ij}$ are modular forms on $\Gamma[2]$
one can use (an analogue of) Proposition \ref{descendlevel2}.
Clearly $\phi_{ii}=0$ and $\phi_{ij}=-\phi_{ji}$. Furthermore, these satisfy
$\phi_{ij}+\phi_{jk}+\phi_{ki}=0$.
One sees easily that
$ \Phi_i= 4\, \sum_{j=1}^{10} \phi_{ij} $.
We also have the relations
$ \phi_{ij}=\phi_{1j}-\phi_{1i} \quad\hbox{\rm and} \quad
 \phi_{1i}=(1/40)(\Phi_1-\Phi_i)
$
and one thus obtains the relation
\begin{equation}\label{sumPhi}
\sum_{i=1}^{10} \Phi_i=0\, .
\end{equation}

To prove Theorem \ref{ThmSigma2}  
we begin by analyzing the $\mathfrak{S}_6$ action on the $\Phi_i$.

\begin{lemma}
The ten forms $\Phi_i$ generate the $9$-dimensional
$s[2^2,1^2]$-isotypic subspace of $S_{2,5}(\Gamma[2])$ and
satisfy the relation $\sum_{i=1}^{10} \Phi_i =0$.
\end{lemma}
\begin{proof}
We calculate the action of $\frak{S}_6$ on the $\Phi_i$ ($i=1,\ldots,10$) 
and find that it is formally a representation $s[2^2,1^2]+s[1^6]$. 
The $s[1^6]$ corresponds to the relation $(\ref{sumPhi})$. 
Since the $\Phi_i$ are non-zero these 
must generate an irreducible representation
$s[2^2,1^2]$ in $S_{2,5}(\Gamma[2])$.
Alternatively, by restricting to the components of the Humbert surface
$H_1$ one can also check 
that $\Phi_i$ for $i=1,\ldots,9$
are linearly independent, cf.\ the proof of Lemma \ref{S21S23}.
\end{proof}

We now give the proof of Theorem \ref{ThmSigma2}. We first show that the $\Phi_i$
generate $S_{2,7}(\Gamma[2])$ and $S_{2,9}(\Gamma[2])$ 
and we find relations there.



We obtain a relation in weight $(2,7)$ as follows.
A linear relation between the $x_i$ like
$x_1-x_4-x_6-x_7=0$
implies by linearity of the bracket a relation
$ [x_1,\chi_5]-[x_4,\chi_5]-[x_6,\chi_5]-[x_7,\chi_5]=0$
and we can rewrite it as
$$
x_1\, \Phi_1- x_4 \, \Phi_4-x_6\, \Phi_6 - x_7 \,  \Phi_7=0.
$$
Since the relations among the ten $x_i$ generate 
an irreducible representation $s[2,1^4]$, we get in this way a 
space $s[2,1^4]\otimes s[1^6]=s[5,1]$ of relations
between the $\Phi_i$ over $R^{\rm ev}$ in weight $(2,7)$.

One can check that the projections of the space generated by the $x_i\Phi_j$ 
to the $s[4,1^2]$, $s[3^2]$, $s[3,2,1]$ and
$s[2^2,1^2]$-part do give non-zero modular forms, 
and comparing this to the decomposition of $S_{2,7}(\Gamma[2])$ 
into irreducible $\frak{S}_6$ representations shows that the $\Phi_i$ 
generate $S_{2,7}(\Gamma[2])$ over the ring of scalar-valued modular forms. 
Now $M_{0,2}(\Gamma[2])=s[2^3]$ and $S_{2,5}(\Gamma[2])=s[2^2,1^2]$ and 
comparing the two representations
$$
\begin{aligned}
s[2^3] \otimes s[2^2,1^2]&=
s[5,1]+s[4,1^2]+s[3^2]+s[3,2,1]+s[2^2,1^2]\\
S_{2,7}(\Gamma[2])&=s[4,1^2]+s[3^2]+s[3,2,1]+s[2^2,1^2]\\
\end{aligned}
$$ 
it follows that we must have an irreducible representation
 $s[5,1]$ of relations, which is just
the space given above. 

In a similar way we compare the representations in weight $(2,9)$. 
We find that $S_{2,5}(\Gamma[2]) \otimes M_{0,4}(\Gamma[2])$ 
equals as an $\frak{S}_6$ representation
the representation of $S_{2,9}(\Gamma_2[2])$ plus $s[3^2]+s[3,2,1]+s[2^2,1^2]$;
the contribution $s[3,2,1]+s[2^2,1^2]$ to this excess 
comes from $M_{0,2}(\Gamma[2])\otimes s[5,1]$ where $s[5,1]$ are 
the relations in weight $(2,7)$.
We are thus left with a relation space $s[3^2]$ in weight $(2,9)$. Indeed, by calculating the projections
we check that $M_{0,4}\otimes S_{2,5} \to S_{2,9}$ is surjective and
the explicit relations can be computed by projection on the $s[3^2]$-subspace. 
Since the coefficients are not simple we refrain from giving these.

We can check again by projection on the isotopic subspaces of $S_{2,11}$ that
$M_{0,6}\otimes S_{2,5} \to S_{2,11}$ is surjective. We thus find a syzygy of type
$s[1^6]$ in weight $(2,11)$. By the result on the Castelnuovo-Mumford regularity
of Proposition \ref{regularity2-1} there can be no further relations. Therefore the $\Phi_i$ generate
a submodule of $\Sigma_2(\Gamma[2])$ with Hilbert function
$$
\frac{9\, t^5- 5 \, t^7-5\, t^9+t^{11}}{(1-t^2)^5}.
$$
Since this coincides with the generating series given in Section
\ref{dimensions}, the $\Phi_i$ must generate the whole module.
This completes the proof of the theorem.

\smallskip
\begin{remark}
If we work over the function field ${\mathcal F}$ of 
${\mathcal A}_2[\Gamma[2]]$ and
consider the module $\Sigma_2 \otimes {\mathcal F}$ of
meromorphic sections of ${\rm Sym}^2({\EE})$ with ${\EE}$ the Hodge bundle,
then the submodule $F$ generated by the $\Phi_i$ has rank at least three since
the wedge product $\Phi_1\wedge \Phi_2 \wedge \Phi_3$ does not vanish
identically. Using the relation $\sum_{i=1}^{10} \Phi_i=0$ and 
the five relations of weight $(2,7)$ 
$$
\begin{aligned}
x_6\Phi_6=& x_1\Phi_1-x_2\Phi_2+x_3\Phi_3-x_4\Phi_4-x_5\Phi_5, \\
x_7\Phi_7 =& x_2\Phi_2-x_3\Phi_3+x_5\Phi_5, \quad
x_8\Phi_8 = x_1\Phi_1-x_4\Phi_4-x_5\Phi_5, \\
x_9\Phi_9 =& -x_3\Phi_3+x_4\Phi_4+x_5\Phi_5, \quad
x_{10}\Phi_{10} = x_1\Phi_1-x_2\Phi_2-x_5\Phi_5 \\
\end{aligned}
$$
we see that $F$ is generated by $\Phi_i$
with $i=1,\ldots,4$. Indeed, after inverting the $x_i$ we can eliminate
$\Phi_6,\ldots,\Phi_{9}$ and then using $\ref{sumPhi}$ and 
$x_{10}\Phi_{10} = x_1\Phi_1-x_2\Phi_2-x_5\Phi_5$ we can also eliminate 
$\Phi_5$. Using the relations of type $s[3^2]$ in weight $(2,9)$
we can eliminate $\Phi_4$ too and 
 reduce the generators of $F$ to $\Phi_1,\Phi_2,\Phi_3$. 
We refrain from giving the explicit relation. So outside the
zero divisor of the wedge $\Phi_1 \wedge \Phi_2 \wedge \Phi_3$
the forms $\Phi_1,\Phi_2$ and $\Phi_3$ 
generate the bundle ${\rm Sym}^2({\EE})\otimes \det({\EE})^5$.
\end{remark}

We now give some wedges of the $\Phi_i$ that give information about the
vanishing loci of the $\Phi_i$.

\begin{proposition}\label{wedgePhi}
We have in $S_{0,18}(\Gamma[2])$ the identity
\[
\Phi_1 \wedge \Phi_2 \wedge \Phi_3=
25\,\chi_5^2(x_6-x_5)(3\, \vartheta_5^2\vartheta_6^2\vartheta_7^2\vartheta_8^2\vartheta_9^2\vartheta_{10}^2+
\vartheta_1^2\vartheta_2^2\vartheta_3^2(\vartheta_6^2\vartheta_8^2\vartheta_9^2-\vartheta_5^2\vartheta_7^2\vartheta_{10}^2))/8\, .
\]

\end{proposition}

The proposition is proved by brute force by computing a basis of the space of modular forms involved.
We know that $\Phi_1\wedge \Phi_2\wedge \Phi_3$ is divisible by $\chi_5^2$ and $x_5-x_6$, hence
the quotient is a form $f_6$ of weight $6$, and actually a cusp form.
In $S_{0,6}(\Gamma[2])$, a representation of type $s[2^3]$,
 there are five linearly independent 
modular forms $g_i$ ($i=1,\ldots,5$)
that are products of squares of six theta constants with characteristics
given by
$$
[1,2,3,5,7,10],[1,2,3,6,8,9],[1,2,4,5,8,10],[1,3,4,5,8,9],[5,6,7,8,9,10]
$$
respectively.
By computing the Fourier expansion of $f_6$ and of the latter forms $g_1,\ldots,g_5$, we get the proposition.

Using the same method, we find
$$
\begin{aligned}
\Phi_1 \wedge \Phi_2 \wedge \Phi_4=&
25\,\chi_5^2(x_6-x_5)(g_1+g_2-2\, g_3-3\, g_5)/8, \\
\Phi_1 \wedge \Phi_3 \wedge \Phi_4=
& 25\,\chi_5^2(x_6-x_5)(g_1+g_2-2\, g_4+3\, g_5)/8 \\
\Phi_2 \wedge \Phi_3 \wedge \Phi_4=
& 25\,\chi_5^2(x_6-x_5)(-g_1+g_2+2\, g_3-2\, g_4-g_5)/8. \\
\end{aligned}
$$

We now prove that $S_{2,1}(\Gamma[2])$ and $S_{2,3}(\Gamma[2])$ are both zero.

\begin{lemma}\label{S21S23}
We have $S_{2,1}(\Gamma[2])=(0)$ and $S_{2,3}(\Gamma[2])=(0)$.
 \end{lemma}
\begin{proof}
Since multiplication by $x_1$ is injective
the vanishing of $S_{2,3}(\Gamma[2])$  implies the vanishing of $S_{2,1}(\Gamma[2])$. 
We thus have to prove that $S_{2,3}(\Gamma[2])=(0)$.
The injectivity of multiplication by $x_1$
applied to  $S_{2,3}(\Gamma[2])$ and $\dim S_{2,5}(\Gamma[2])=9$ implies that 
$\dim S_{2,3}(\Gamma[2])\leq 9$. But since not every $\Phi_i$ is
divisible by $x_1$, as follows from calculating the restriction to the 
components of the Humbert surface $H_1$:
$$
\Phi_i(\left( \begin{matrix} \tau_{11} & 0 \\ 0 & \tau_{22} \end{matrix} \right) )
= c_i \, (\vartheta_{00}\vartheta_{01}\vartheta_{10})^4(\tau_{11}) \otimes
(\vartheta_{00}\vartheta_{01}\vartheta_{10})^4(\tau_{22}) \, 
\left( \begin{matrix} 0  \\ 2\\  0 \\ \end{matrix} \right)
$$
with $c_i=-1/4$ for $i=1,\ldots,9$ and $c_{10}=9/4$.
We see that $\dim S_{2,3}(\Gamma[2])< 9$.
 By multiplication by $\psi_4$ and $\psi_6$ in $M_{0,4}(\Gamma)$ and 
$M_{0,6}(\Gamma)$  we land in $S_{2,7}(\Gamma[2])$ and $S_{2,9}(\Gamma[2])$ 
and by inspection we see that the only irreducible representations in 
common in $S_{2,7}(\Gamma[2])$ and $S_{2,9}(\Gamma[2])$
are of type $s[4,1^2]$, $s[3^2]$,  $s[3,2,1]$ and $s[2^2,1^2]$ , 
so for dimension reasons we must 
have $S_{2,3}(\Gamma[2])=s[3^2]$ if it is non-zero. 
If $S_{2,3}(\Gamma[2])=s[3^2]$  we find that $S_{2,3}(\Gamma_1[2])=s[1^3]$ as an $\frak{S}_3$-representation
and since $M_{0,2}(\Gamma_1[2])=s[3]$ we find a representation $s[1^3]$ in $S_{2,5}(\Gamma_1[2])$.
But we know that $S_{2,5}(\Gamma_1[2])=(0)$. Therefore $S_{2,3}(\Gamma[2])=(0)$.
\end{proof}

\end{section}
\begin{section}{The Module $\Sigma_2(\Gamma_1[2])$}
Recall that the ring $R^{\rm ev}(\Gamma[2])=\oplus_k M_{0,2k}(\Gamma[2])$ 
is abstractly isomorphic to the ring $R'=\oplus_k M_{0,4k}(\Gamma_1[2])$.

We therefore look at the following two $R'$-modules
$$
\Sigma^{1}=\Sigma^{1}(\Gamma_1[2])=\oplus_k S_{2,4k+1}(\Gamma_1[2])
\quad {\rm and} \quad 
\Sigma^3=\Sigma^{3}(\Gamma_1[2])=\oplus_k S_{2,4k+3}(\Gamma_1[2]) \,
$$
As before, we can consider these as modules over a polynomial ring in
five variables as well as over the ring of scalar-valued modular forms.

\begin{theorem}\label{ThmSigma1Gamma1}
The module $\Sigma^1(\Gamma_1[2])$ is generated over the ring $R^{\prime}$
by the nine cusp forms of weight $(2,9)$
generating a $\frak{S}_3$ representation $3 \, s[2,1] + 3\, s[1^3]$.
The module $\Sigma^3(\Gamma_1[2])$
is generated by the $4$ modular forms of weight $(2,7)$
forming a $\frak{S}_3$-representation $s[2,1]+2\, s[1^3]$
and the two pairs of  forms of weight $(2,11)$ each forming a
$\frak{S}_3$-representation $s[2,1]$.
\end{theorem}

\begin{remark}
The generating functions for the dimensions of the graded pieces of
these modules are
$$
\frac{9\, t^9-5\, t^{13}-5\, t^{17}+t^{21}}{(1-t^4)^5}\, 
\quad{\rm and} \quad
\frac{4\, t^7+4 \, t^{11}- 8\, t^{15}}{(1-t^4)^5}\, .
$$
This follows now from the results on $\Sigma_2(\Gamma[2])$.
\end{remark}

\noindent
{\sl Proof of Thm.\ \ref{ThmSigma1Gamma1}}
In order to construct the generators explicitly, we look at the eigenspaces
of the action of $\Gamma_1[2]/\Gamma[2]=({\ZZ}/2{\ZZ})^3$; this group
is generated by $(12),(34)$ and $(56) \in \frak{S}_6$. So for a triple $\epsilon$ of signs we have a corresponding eigenspace 
$M_{j,k}^{\epsilon} \subset M_{j,k}(\Gamma[2])$.
We have maps
$$
M_{0,k_1}^{\epsilon} \times M_{j,k_2}^{\epsilon} \to M_{j,k_1+k_2}(\Gamma_1[2]),
\qquad
M_{0,k_1}^{\epsilon} \times M_{0,k_2}^{\epsilon} \to M_{2,k_1+k_2}(\Gamma_1[2]),
$$
given by the product $(f,g) \mapsto fg$, resp.\  by the Rankin-Cohen bracket
$(f,g) \mapsto [f,g]$.

For example, a form in the $s[1^6]$-part of $M_{0,k}(\Gamma[2])$ gives rise
to a form in $M_{0,k}^{---}$; so as soon as this $s[1^6]$-part is not empty 
we get forms in $M_{2,k+5}(\Gamma_1[2])$ by taking the bracket 
$\psi \mapsto [\psi,\chi_5]$. Using this idea we can construct forms
in the following way.

As the four generators in weight $(2,7)$ one can take: $F_1=(x_5-x_6)
(\Phi_1+\Phi_2+\Phi_3+\Phi_4)$, and 
$$F_2=(x_5+x_6)(\Phi_5-\Phi_6), \quad 
F_3=(x_7+x_8)(\Phi_7-\Phi_8), \quad F_4=(x_9+x_{10})(\Phi_9-\Phi_{10}),
$$
with $F_1$, $F_2+F_3+F_4$ generating $2s[1^3]$ and $F_1-F_3$ and $F_2-F_3$
generating  a $s[2,1]$.

To construct cusp forms of weight $(2,9)$ we take  $A_1=s_1F_2$, $A_2=s_1F_3$
and $A_3=s_1F_4$ and 
$$
\begin{aligned}
A_4= (x_5+x_6)(x_7+x_8)(\Phi_9-\Phi_{10}), \quad & A_7= (x_5+x_6)\xi (\Phi_1-\Phi_{2}+\Phi_3-\Phi_4), \\   
A_5= (x_7+x_8)(x_9+x_{10})(\Phi_5-\Phi_6), \quad & A_8= (x_9+x_{10})\xi (\Phi_1-\Phi_{2}-\Phi_3+\Phi_4),\\
A_6=(x_5+x_6)(x_9+x_{10})(\Phi_7-\Phi_8), \quad & A_9= (x_7+x_8)\xi (\Phi_1+\Phi_{2}-\Phi_3-\Phi_4). \\
\end{aligned}
$$

Finally, to construct generators of weight $(2,11)$ we consider
$$
\begin{aligned}
L_1=(x_1^3+x_2^3-x_3^3-x_4^3)(\Phi_7-\Phi_8), \quad &
M_1=\xi(x_7+x_8)(x_9+x_{10})(\Phi_1-\Phi_2+\Phi_3-\Phi_4),\\
L_2=(x_1^3-x_2^3-x_3^3+x_4^3)(\Phi_9-\Phi_{10}),\quad &
M_2=\xi(x_5+x_6)(x_7+x_{8})(\Phi_1-\Phi_2-\Phi_3+\Phi_4),\\
L_3=(x_1^3-x_2^3+x_3^3-x_4^3)(\Phi_5-\Phi_6),\quad &
M_3=\xi(x_5+x_6)(x_9+x_{10})(\Phi_1+\Phi_2-\Phi_3-\Phi_4). \\
\end{aligned}
$$
The generators that we need are the nine forms $A_i$ of weight $(2,9)$ generating a representation $3s[2,1]+3s[1^3]$, the four modular forms $F_i$ of weight $(2,7)$ generating a representation $s[2,1]+2s[1^3]$ and the two pairs
$L_3-L_1,L_3-L_2$ and $M_1-M_2,M_1-M_3$
of weight $(2,11)$, each generating a representation $s[2,1]$. 
One checks that these forms generate up to weight $(2,19)$ and
that the Castelnuovo-Mumford regularity is bounded by $3$.
This finishes the proof.
\end{section}
\begin{section}{The Module ${\mathcal M}_2(\Gamma[2])$ and its $\Gamma_1[2]$-analogue.}
In this section we determine the structure of the $R^{\rm ev}$-module
$$
{\mathcal M}_2 = {\mathcal M}_2(\Gamma[2])=
 \oplus_{k=0}^{\infty} M_{2,2k}(\Gamma[2])\, .
$$
In Example \ref{Gij} we constructed $15$ modular forms $G_{ij}$
in $M_{2,4}(\Gamma[2])$; recall that these are proportional to the $H_{ij}$. 
As shown in Remark (\ref{boundaryvsoddpair}) 
these are linearly independent.

\begin{theorem} The $R^{\rm ev}$-module ${\mathcal M}_2$ is generated 
by the fifteen modular forms $G_{ij}$.
\end{theorem}

As in the preceding section the $\frak{S}_6$ action is an essential tool
for proving this theorem.
We list the representations involved.

\begin{footnotesize}
\smallskip
\vbox{
\bigskip\centerline{\def\quad{\hskip 0.6em\relax}
\def\quod{\hskip 0.5em\relax }
\vbox{\offinterlineskip
\hrule
\halign{&\vrule#&\strut\quod\hfil#\quad\cr
height2pt&\omit&&\omit&&\omit&&\omit&&\omit&&\omit&&\omit&&\omit&&\omit&&\omit && \omit && \omit &\cr
&$M_{2,k}\backslash P$ && $[6]$ && $[5,1]$ && $[4,2]$ && $[4,1^2]$ && $[3^2]$ && $[3,2,1]$ && $[3,1^3]$ && $[2^3]$ && $[2^2,1^2]$ && $[2,1^4]$&& $[1^6]$ &\cr
\noalign{\hrule}
&$M_{2,4}$ && $0$ && $0$ && $0$ && $0$ && $0$ && $0$ && $1$ && $0$ && $0$ && $1$ && $0$ &\cr
&$M_{2,6}$ && $0$ && $0$ && $1$ && $0$ && $0$ && $2$ && $1$ && $1$ && $0$ && $0$ && $0$ &\cr
&$M_{2,8}$ && $0$ && $0$ && $2$ && $1$ && $0$ && $3$ && $3$ && $2$ && $1$ && $2$ && $0$ &\cr
&$M_{2,10}$ && $1$ && $2$ && $5$ && $3$ && $0$ && $5$ && $5$ && $3$ && $2$ && $3$ && $0$ &\cr
&$M_{2,12}$ && $0$ && $1$ && $7$ && $4$ && $1$ && $11$ && $8$ && $6$ && $4$ && $4$ && $0$ &\cr
} \hrule}
}}
\end{footnotesize}
The generating function for $\dim M_{2,k}$ with $k\geq 4$ even is
\begin{equation}\label{X1}
\frac{ 15\, t^4-19\, t^6+ 5 \, t^8 -t^{10}}{(1-t^2)^{5}}\, .
\end{equation}

We know already that the $15$ forms $G_{ij}$ 
generate a $\frak{S}_6$-representation $s[3,1^3]+s[2,1^4]$.

\begin{lemma}\label{generatorsM2}
The $15$ forms $G_{ij}$ with $1 \leq i < j \leq 6$ 
form a basis of the space $M_{2,4}(\Gamma[2])$.
\end{lemma}

\begin{remark} By using the forms 
$F_{ij}=[x_i,x_j]=[\vartheta_{i}^4,\vartheta_{j}^4]$
we find in this way the $\wedge^2 s[2^3]=s[3,1^3]$-part of $M_{2,4}$ as the $x_i$ generate a
$s[2^3]$.
The
$F_{ij}$ can be expressed in the $G_{ij}$, for example
$$
F_{12}=\frac{1}{\pi^2}(-G_{12}+G_{56}-G_{15}-G_{26})\, .
$$
\end{remark}

Now $M_{2,6}$ decomposes as 
$s[4,2]+2\, s[3,2,1]+s[3,1^3]+s[2^3]$ as a representation space
for $\frak{S}_6$ and $M_{0,2}=s[2^3]$ and since
\begin{equation}\label{X2}
\begin{aligned}
s[2,1^4]\otimes s[2^3]&=s[4,2]+ s[3,2,1] \\ 
s[3,1^3]\otimes s[2^3] &=s[4,2]+s[4,1^2]+
s[3,2,1]+s[3,1^3]+s[2^3]\\
\end{aligned} 
\end{equation}
we expect to find relations of
type $s[4,2]+s[4,1^2]$. 
One checks that $M_{0,2} \otimes M_{2,4}$
generates $M_{2,6}$. We get a $s[4,1^2]$ of relations of the form
\begin{equation}\label{X3}
x_i\, F_{jk}-x_j\, F_{ik} +x_k \, F_{ij}=0\, . 
\end{equation}
These relations follow immediately from the Jacobi identity for brackets.
The relations of type $s[4,2]$ either come from the vanishing of a $s[4,2]$
in the right-hand sides of (\ref{X2}) or from an identification of a copy of $s[4,2]$
in these right-hand sides. The latter is the case. We give an example of such a relation:
$$
\begin{aligned}
x_1(2G_{23}-G_{25}+G_{35}+G_{56})-x_2(G_{24}+G_{45})-x_3(G_{13}-G_{15})-x_5G_{26}& \\
+x_8(G_{36}+G_{56})-x_9(G_{34}-G_{45})+x_{10}(G_{12}-G_{15})=0\, . & \\
\end{aligned}
$$
Denote the left-hand-side of the relation (\ref{X3}) by $R_{ijk}$.
Then we have the syzygy
\begin{equation}\label{X4}
x_i\, R_{jkl} - x_j \, R_{ikl} + x_k \, R_{ijl} - x_l \, R_{ijk}=0\, .
\end{equation}
and it generates an irreducible representation $s[3^2]$ of relations in weight $(2,8)$.

In a similar way we expect a syzygy of type $s[1^6]$ in weight $(2,10)$.
Write $R_{ijkl}$ for the left-hand-side of (\ref{X4}). Then we have
\begin{equation}\label{X5}
x_1 \, R_{2345} - x_2 \, R_{1345}
+x_3 \, R_{1245} - x_4 \, R_{1235} + x_5 \, R_{1234}=0\, . 
\end{equation}
This is a $S_6$-anti-invariant syzygy in weight $(2,10)$.
By using the result on the regularity \ref{regularity2-0} we can derive now as
we did above that we cannot have more relations.
The $G_{ij}$ thus generate a submodule of
${\mathcal M}_{2}^{\rm ev}$ with Hilbert function given by (\ref{X1}).
Since this coincides with the generating function of our module ${\mathcal M}_{2}^{\rm ev}$ we have found our module.
This proves the theorem.

\smallskip
Since the forms
$H_{ij}^{\prime}$ defined in Example \ref{Hijprime} 
satisfy similar relations we can deduce in a 
completely analogous way the following theorem.

\begin{theorem}
The fifteen modular forms $H_{ij}^{\prime}\in M_{2,8}(\Gamma_1[2])$ from 
Example \ref{Hijprime} generate the module $\oplus_k M_{2,4k}(\Gamma_1[2])$
over the module $\oplus_k M_{0,4k}(\Gamma_1[2])$.
\end{theorem}

\end{section}
\begin{section}{Other Modules}
\begin{subsection}{The module ${\mathcal M}_4$}

We treat the $R^{\rm ev}$-module ${\mathcal M}_4=\oplus_k M_{4,2k}(\Gamma[2])$.

\begin{theorem}\label{ThmM4}
The module ${\mathcal M}_4$ over $R^{\rm ev}$ is generated by 
six modular forms of weight $(4,2)$ generating a representation $s[2,1^4]$, 
fifteen modular forms of weight $(4,4)$ generating a representation 
$s[2,1^4]$ and five modular forms of weight $(4,4)$ generating a  
representation $s[2^3]$.
\end{theorem}

The proof is similar to the cases given above. First we look where the
generators should appear, we then construct these and check
that these generate $M_{4,k}$ for small $k$, and then use the
bound on the Castelnuovo-Mumford regularity to bound the weight of 
the generators and relations. 
 
The $\frak{S}_6$ representations for small $k$ are as follows.

\begin{footnotesize}
\smallskip
\vbox{
\bigskip\centerline{\def\quad{\hskip 0.6em\relax}
\def\quod{\hskip 0.5em\relax }
\vbox{\offinterlineskip
\hrule
\halign{&\vrule#&\strut\quod\hfil#\quad\cr
height2pt&\omit&&\omit&&\omit&&\omit&&\omit&&\omit&&\omit&&\omit&&\omit&&\omit && \omit && \omit &\cr
&$M_{4,k}\backslash P$ && $[6]$ && $[5,1]$ && $[4,2]$ && $[4,1^2]$ && $[3^2]$ && $[3,2,1]$ && $[3,1^3]$ && $[2^3]$ && $[2^2,1^2]$ && $[2,1^4]$&& $[1^6]$ &\cr
\noalign{\hrule}
&$M_{4,2}$ && $0$ && $0$ && $0$ && $0$ && $0$ && $0$ && $0$ && $0$ && $0$ && $1$ && $0$ &\cr
&$M_{4,4}$ && $0$ && $0$ && $1$ && $0$ && $0$ && $1$ && $0$ && $1$ && $0$ && $1$ && $0$ &\cr
&$M_{4,6}$ && $0$ && $0$ && $2$ && $0$ && $0$ && $3$ && $2$ && $2$ && $1$ && $2$ && $0$ &\cr
&$M_{4,8}$ && $1$ && $2$ && $5$ && $2$ && $0$ && $6$ && $4$ && $3$ && $2$ && $4$ && $0$ &\cr
&$M_{4,10}$ && $1$ && $2$ && $8$ && $4$ && $1$ && $12$ && $8$ && $6$ && $5$ && $6$ && $0$ &\cr
&$M_{4,12}$ && $2$ && $5$ && $14$ && $8$ && $3$ && $20$ && $13$ && $9$ && $8$ && $8$ && $0$ &\cr
} \hrule}
}}
\end{footnotesize}

The generating series is
$$
\sum_{k \in 2{\ZZ}_{>0}} \dim M_{4,k} t^k= \frac{5\, t^2+10\, t^4-10 \, t^6 -10 \, t^8 +5\, t^{10}}{(1-t^2)^5}\, .
$$
We also give the cusp forms:

\begin{footnotesize}
\smallskip
\vbox{
\bigskip\centerline{\def\quad{\hskip 0.6em\relax}
\def\quod{\hskip 0.5em\relax }
\vbox{\offinterlineskip
\hrule
\halign{&\vrule#&\strut\quod\hfil#\quad\cr
height2pt&\omit&&\omit&&\omit&&\omit&&\omit&&\omit&&\omit&&\omit&&\omit&&\omit && \omit && \omit &\cr
&$S_{4,k}\backslash P$ && $[6]$ && $[5,1]$ && $[4,2]$ && $[4,1^2]$ && $[3^2]$ && $[3,2,1]$ && $[3,1^3]$ && $[2^3]$ && $[2^2,1^2]$ && $[2,1^4]$&& $[1^6]$ &\cr
\noalign{\hrule}
&$S_{4,4}$ && $0$ && $0$ && $0$ && $0$ && $0$ && $0$ && $0$ && $0$ && $0$ && $1$ && $0$ &\cr
&$S_{4,6}$ && $0$ && $0$ && $1$ && $0$ && $0$ && $2$ && $1$ && $1$ && $1$ && $1$ && $0$ &\cr
&$S_{4,8}$ && $0$ && $1$ && $3$ && $2$ && $0$ && $5$ && $3$ && $2$ && $2$ && $3$ && $0$ &\cr
&$S_{4,10}$ && $1$ && $2$ && $6$ && $4$ && $1$ && $10$ && $7$ && $4$ && $5$ && $5$ && $0$ &\cr
&$S_{4,12}$ && $1$ && $4$ && $11$ && $8$ && $3$ && $18$ && $12$ && $7$ && $8$ && $7$ && $0$ &\cr
} \hrule}
}}
\end{footnotesize}
The generating function is
$$
\sum_{k\in {\ZZ}_{\geq 2}} \dim S_{4,2k} t^{2k} = \frac{5\, t^4+45\, t^6-95\, t^8 +55\, t^{10}-10\, t^{12}}{(1-t^2)^5}\, .
$$
Using the map $S_{4,2} \times M_{0,2} \to S_{4,6}$ we see that $S_{4,2}=(0)$.
In weight $(4,2)$ we find a space of Eisenstein series $s[2,1^4]$ of dimension $5$
instead of the usual $s[2^3]+s[2,1^4]$. We now construct generators for our module.
We expect generators $s[2,1^4]$ in weight $(4,2)$, of type $s[2^3]+s[2,1^4]$
in weight $(4,4)$, relations of type $s[6]+s[4,2]$ both in weight
$(4,6)$ and $(4,8)$ and a syzygy of type $s[2^3]$ in weight $(4,10)$. That is what we shall find.

\begin{proposition}
The forms $E_i={\rm Sym}^4(G_i)$ for $i=1,\ldots,6$
are modular forms of weight $(4,2)$ and satisfy the $s[1^6]$-type
relation $E_1-E_2-E_3+E_4-E_5+E_6 =0$.
They generate the space $M_{4,2}=M_{4,2}^{s[2,1^4]}$.
\end{proposition}
Here our convention is that if $G_i=[a,b]^t$ then ${\rm Sym}^4(G_i)=[a^4,4a^3b,6a^2b^2,4ab^3,b^4]^t$.
The following lemma is proved by a direct calculation.

\begin{lemma}\label{wedgeE}
We have $E_1\wedge E_2 \wedge \ldots \wedge E_5=
-96\, \pi^{6}\, \chi_{5}^4$. 
\end{lemma}

\begin{corollary}
Every linear relation of the form $\sum f_iE_i=0$ 
with $f_i \in R^{\rm ev}$  is a multiple of
$E_1-E_2-E_3+E_4-E_5+E_6=0$. The $E_i$ generate a submodule of ${\mathcal M}_4$
with generating function $5\, t^2/(1-t^2)^5$.
\end{corollary}
\begin{proof}
If  $\sum_i f_i E_i=0$ is a relation not in the ideal generated by $E_1-E_2-E_3+E_4-E_5+E_6$ then over the function field
${\mathcal F}$ of ${\mathcal A}_2[\Gamma[2]]$ we can eliminate $E_6$ and $E_5$ and then the wedge would be zero
contradicting Lemma (\ref{wedgeE}).
\end{proof}

To construct the forms in the $s[2,1^4]$ space of $M_{4,4}(\Gamma[2])$ we consider the $15$ modular forms in the $\frak{S}_6$-orbit of
$$
D_{1234}={\rm Sym}^4(G_1,G_2,G_3,G_4)
\vartheta_{5}\vartheta_{6}
\vartheta_{7}\vartheta_{8} \in M_{4,4}(\Gamma[2])
$$
The formal representation is of type $s[3,1^3]+s[2,1^4]$, but these forms satisfy an irreducible representation 
of type $s[3,1^3]$ of relations generated by
$$
4\, D_{1234}-D_{1235}-D_{1236}-D_{1245}-D_{1246}-D_{1345}-D_{1346}-D_{2345}-D_{2346}=0\, .
$$
These forms are cusp forms and generate the space of cusp forms $S_{4,4}=S_{4,4}^{s[2,1^4]}$.
A basis is given by the forms $D_{1256}, D_{1345}, D_{1346}, D_{1356}$ and $D_{3456}$ as follows from the
fact that
$$
D_{1256} \wedge D_{1345} \wedge  D_{1346} \wedge D_{1356} \wedge D_{3456}= -\pi^{20} \chi_5^6.
$$
In fact, we find many linear identities between these forms. Simplifying one of those leads to an identity like
$$
\begin{aligned}
&{\rm{Sym}}^4(G_1,G_2,G_3,G_4)\vartheta_6\vartheta_7\vartheta_8= \\
& \qquad {\rm{Sym}}^4(G_1,G_3,G_4,G_5)\vartheta_3\vartheta_4\vartheta_9+
{\rm{Sym}}^4(G_1,G_3,G_4,G_6)\vartheta_1\vartheta_2\vartheta_{10}\, .
\end{aligned}
$$

Finally we construct generators in the $s[2^3]$-part of $M_{4,4}(\Gamma[2])$. We consider expressions
$$
K_{i,j,k,l}={\rm Sym}^4(G_i,G_i,G_j,G_j)\, \vartheta^2_k \vartheta^2_l
$$
for appropriate quadruples $(i,j,k,l)$. For example we take $K_{1,2,1,3} \in M_{4,4}(\Gamma[2])$. These modular
forms satisfy many relations, e.g.,
$$
K_{1,2,1,3}-K_{1,2,2,4}-K_{1,2,5,6}=0 \qquad \hbox{\rm due to} \qquad 
\vartheta_1^2\vartheta_3^3-\vartheta_2^2\vartheta_4^2-\vartheta_5^2\vartheta_6^2 =0\, .
$$
We find $30$ such forms in the $\frak{S}_6$-orbit and as it turns out these are linearly independent and 
generate a $s[4,2]+s[3,2,1]+s[2^3]$ subspace of $M_{4,4}$.
If $p$ denotes the projection on the $s[2^3]$-subspace the five forms $R_1=p(K_{1,2,1,3})$, $R_2=p(K_{1,2,2,4})$, $R_3=p(K_{1,3,1,10})$, $R_4=p(K_{1,3,4,9})$ and
$R_5=p(K_{1,4,2,10})$ form a basis of $s[2^3]$-subspace of $M_{4,4}(\Gamma[2])$ as a calculation shows.
As it turns out their wedge is zero since these satisfy a ($\frak{S}_6$-anti-invariant) relation
$$
x_2\, R_1-(x_2 + x_5)\, R_2 -(x_2 - x_4)\, R_3 -(x_1-x_2- x_5)\, R_4 +(x_2-x_3+x_5)\, R_5=0 .
$$

We now prove the theorem. One can show that this module is $3$-regular 
in the sense of Castelnuovo-Mumford as in Section \ref{ModuleStructure}. 
Then one checks that
these generators generate the spaces $M_{4,2k}$ for $k\leq 3$. 
By \cite{Lazarsfeld}, Thm.\ 1.8.26 this suffices.
This finishes the proof.

\end{subsection}
\begin{subsection}{The module $\Sigma_4$}

Another case is $\Sigma_4^{\rm odd}(\Gamma[2])$, where the representations are as follows.

\begin{footnotesize}
\smallskip
\vbox{
\bigskip\centerline{\def\quad{\hskip 0.6em\relax}
\def\quod{\hskip 0.5em\relax }
\vbox{\offinterlineskip
\hrule
\halign{&\vrule#&\strut\quod\hfil#\quad\cr
height2pt&\omit&&\omit&&\omit&&\omit&&\omit&&\omit&&\omit&&\omit&&\omit&&\omit && \omit && \omit &\cr
&$S_{4,k}\backslash P$ && $[6]$ && $[5,1]$ && $[4,2]$ && $[4,1^2]$ && $[3^2]$ && $[3,2,1]$ && $[3,1^3]$ && $[2^3]$ && $[2^2,1^2]$ && $[2,1^4]$&& $[1^6]$ &\cr
\noalign{\hrule}
&$S_{4,3}$ && $0$ && $0$ && $0$ && $0$ && $0$ && $0$ && $0$ && $0$ && $0$ && $0$ && $0$ &\cr
&$S_{4,5}$ && $0$ && $0$ && $0$ && $0$ && $0$ && $1$ && $0$ && $0$ && $1$ && $1$ && $0$ &\cr
&$S_{4,7}$ && $0$ && $1$ && $1$ && $1$ && $1$ && $3$ && $1$ && $0$ && $2$ && $1$ && $0$ &\cr
&$S_{4,9}$ && $0$ && $2$ && $2$ && $3$ && $2$ && $6$ && $3$ && $1$ && $5$ && $3$ && $1$ &\cr
&$S_{4,11}$ && $0$ && $4$ && $5$ && $7$ && $4$ && $12$ && $5$ && $2$ && $8$ && $4$ && $1$ &\cr
} \hrule}
}}
\end{footnotesize}
We expect generators in weight $(4,5)$ of type $s[3, 2, 1] + s[2^2,1^2] +
s[2,1^4]$; relations of type 
$s[5,1]+s[4,2]+s[4,1^2]+s[3,2,1]$ 
in weight $(4,7)$ and a generator of type $s[3,1^3]$ in weight $(4,9)$ and 
then periodic if viewed as a module over the ring of even 
weight scalar-valued modular forms on $\Gamma[2]$.
For the generators of the $s[3,2,1]$ part of $S_{4,5}(\Gamma[2])$ we refer to Example \ref{s321genofS45} and we invite the reader to construct
the remaining ones; in fact,  $S_{4,5}(\Gamma[2])$ is generated by thirty 
cusp forms of the following shape
$$
F_{abcd}= \frac{\chi_5}{\vartheta_a\vartheta_b\vartheta_c\vartheta_d}
{\rm Sym}^2([\vartheta_a,\vartheta_b],[\vartheta_c,\vartheta_d])
$$
for appropriate quadruples $(a,b,c,d)$ of distinct 
integers between $1$ and $10$.
\end{subsection}
\end{section}
\begin{section}{Modular Forms of Level One}\label{levelone}
We can use our constructions and results to obtain modular forms of level $1$.
Note that the modules of vector-valued modular forms of level $1$ for
$j=2,4$ and $6$ were determined by Ibukiyama, Satoh and van Dorp, see
\cite{Ibukiyama1,Ibukiyama2,Satoh, vanDorp}. Ibukiyama used theta series
with harmonic coefficients.
Here is a list of all cases where $\dim S_{j,k}(\Gamma)=1$ for $k\geq 4$.

\begin{footnotesize}
\smallskip
\vbox{
\bigskip\centerline{\def\quad{\hskip 0.6em\relax}
\def\quod{\hskip 0.5em\relax }
\vbox{\offinterlineskip
\hrule
\halign{&\vrule#&\strut\quod\hfil#\quad\cr
height2pt&\omit&&\omit&&\omit&&\omit&&\omit&&\omit&&\omit&&\omit&\cr
&  $j$ && $k $ &&&& $j$ && $k$ &&  && $j$ && $k$  &\cr
\noalign{\hrule}
&  $0$ && $10, 12, 14,35, 39, 41, 43 $ &&&& $10$ && $9,11$ &&  && $20$ && $5$  &\cr
&  $2$ && $14, 21, 23, 25$ &&&& $12$ && $6, 7$ &&  && $24$ && $4$  &\cr
&  $4$ && $10, 12, 15, 17$ &&&& $14$ && $7$ &&  && $28$ && $4$  &\cr
&  $6$ && $8, 10, 11, 13$ &&&& $16$ && $6, 7$ &&  && $30$ && $4$  &\cr
&  $8$ && $8, 9, 11$ &&&& $18$ && $5, 6$ &&  && $34$ && $4$  &\cr
} \hrule}
}}
\end{footnotesize}

We know that $\dim S_{j,2}(\Gamma)=0$ for $j=2,\ldots,10,14$.

In all cases we can write down an explicit form generating the space.
For $j=0$ we know the generators by Igusa's description of the ring of modular forms.
We give a number of these generators below, but note that all can be obtained
from theta series with spherical coefficients for the $E_8$ lattice.

\begin{example}
\begin{enumerate}
\item{}  The form
$
\sum_{i=1}^{10} \chi_5 \, \vartheta_i^8 \, \Phi_i
$
generates the space $S_{2,14}(\Gamma)$.
This form is a multiple of the Rankin-Cohen bracket
$[E_4,\chi_{10}]$ that occurs in the work of Satoh \cite{Satoh}.
\item{} The forms $[E_4,E_6,\chi_{10}]$ and $[E_4,E_6,\chi_{12}]$
generate $S_{2,21}(\Gamma)$ and $S_{2,23}(\Gamma)$, see \cite{Ibukiyama2}.
\item{} The form $\sum_{i=1}^{10} {\rm Sym}^2(\Phi_i)$ generates the space
$S_{4,10}(\Gamma)$.
\item{} The form $A=\chi_5 \, {\rm Sym}^6(G_1,\ldots,G_6)$ generates
$S_{6,8}(\Gamma)$. A candidate generator for $S_{6,13}(\Gamma)$ is
$\{ E_4,A\}$, where we use the notation of \cite{vanDorp}.
\item{} The form ${\rm Sym}^{12}(G_1,G_1,\ldots,G_6,G_6)$ generates
$S_{12,6}(\Gamma)$.
\end{enumerate}

\end{example}
\end{section}
\begin{section}{Some Fourier Expansions}\label{FourierExpansions}
We give in two tables a few Fourier coefficients of 
$\Phi_1 \in S_{2,5}(\Gamma[2])$ and of $D_{1234} \in S_{4,4}(\Gamma[2])^{s[2,1^4]}$. 
We write the Fourier series as
$$
\sum_{a,b,c} A(a,b,c)\, e^{\pi i(a \tau_{11} +b \tau_{12}+ c \tau_{22})}= \sum_{a,c} \gamma(a,c) P(a,c)\, q_1^a q_2^c
$$
where the first sum runs over the triples $(a,b,c)$ of integers with $b^2-4ac<0$. For a fixed pair $(a,c)$ we collect
the coefficients of $q_1^a q_2^c=e^{\pi i(a\tau_{11}+c\tau_{22})}$ in the form of a vector of Laurent polynomials $\gamma(a,c) P(a,c)$ in $r=\exp{\pi i\tau_{12}}$
with $\gamma(a,c)$ an integer. Note that we have
$P(c,a)$ equals $P(a,c)$ read in retrograde order: $P(c,a)_i=P(a,c)_{n-i}$ with $n=3$ (for $\Phi_1$)  or $n=5$ (for $D_{1234}$).
\vfill \eject
\begin{footnotesize}
\smallskip
\vbox{
\bigskip\centerline{\def\quad{\hskip 0.6em\relax}\def\quod{\hskip 0.5em\relax }
\vbox{\offinterlineskip
\hrule\halign{&\vrule#&\strut\quod\hfil#\quad\cr
height2pt&\omit&&\omit && \omit & \cr
& $[a,c]$ && $ \gamma(a,c) $ && $P(a,c)$ & \cr
height2pt&\omit&&\omit && \omit & \cr
\noalign{\hrule}
& $[1,1]$ && $64$ && $\begin{matrix} -r+1/r \\  -r-1/r \\ -r+1/r\\ \end{matrix}$ & \cr
\noalign{\hrule}
height2pt&\omit&&\omit && \omit & \cr
& $[2,1]$ && $1280$ && $\begin{matrix} r-1/r \\ 0 \\ 0\\ \end{matrix} $ & \cr
\noalign{\hrule}
height2pt&\omit&&\omit && \omit & \cr
& $[2,2]$ && $1280$ && $\begin{matrix} r^3-3r+3/r-1/r^3\\  2r^3-2r-2/r+2/r^3 \\ r^3-3r+3/r-1/r^3 \\ \end{matrix}$  & \cr
\noalign{\hrule}
height2pt&\omit&&\omit && \omit & \cr
& $[3,1]$ && $64$ && $\begin{matrix} 3r^3-13r+13/r-3/r^3\\ 3r^3+9r+9/r+3/r^3\\ r^3+9r-9/r-1/r^3 \\ \end{matrix}$ & \cr
\noalign{\hrule}
height2pt&\omit&&\omit && \omit & \cr
& $[3,2] $ && $1280$ && $\begin{matrix} -4r^3+12r-12/r+4/r^3\\ -4r^3+4r+4/r-4/r^3\\ -3r^3-3r+3/r+3/r^3\\ \end{matrix} $ & \cr
\noalign{\hrule}
height2pt&\omit&&\omit && \omit & \cr
& $[3,3] $ && $64$ && $\begin{matrix} -13r^{5}+121r^3-250r+250/r-121/r^3+13/r^{5} \\ -35r^{5}+121r^3-230r-230/
r+121/r^3-35/r^{5}\\ -13r^{5}+121r^3-250r+250/r-121/r^3+13/r^{5} \end{matrix}$ & \cr
\noalign{\hrule}
height2pt&\omit&&\omit && \omit & \cr
& $[4,1]$ && $1280$ && $\begin{matrix}   -r^3-5r+5/r+1/r^3\\ 0\\ 0 \\  \end{matrix} $ & \cr 
\noalign{\hrule}
height2pt&\omit&&\omit && \omit & \cr
& $[4,2]$  && $1280$ && $\begin{matrix} -r^{5}+5r^3-10r+10/r-5/r^3+1/r^{5}\\ -2r^{5}-10r^3+12r+12/r-10/r^3-2/r^{5} \\ -r^{5}-3r^3+14r-14/r+3/r^3+1/r^{5} \\ \end{matrix} $ & \cr
\noalign{\hrule}
height2pt&\omit&&\omit && \omit & \cr
& $[4,3] $ && $1280$ && $\begin{matrix} 5r^{5}+19r^3+14r-14/r-19/r^3-5/r^{5} \\ 20r^{5}-28r^3+8r+8/r-28/r^3+
20/r^{5} \\ 12r^{5}-28r^3+24r-24/r+28/r^3-12/r^{5} \\ \end{matrix} $ & \cr
\noalign{\hrule}
height2pt&\omit&&\omit && \omit & \cr
& $[4,4]$ && $1280$ && $ \begin{matrix}
-5r^{7}+19r^{5}-25r^3+15r-15/r+25/r^3-19/r^{5}+5/r^{7}\\ 
-10r^{7}+66r^{5}- 10 r^3-46 r-46/r-10/r^3+66/r^{5}-10/r^{7}\\ 
-5r^{7}+19r^{5}-25r^3+15r-15/r+25/r^3-19/r^{5}+5/r^{7} \\ \end{matrix} $ & \cr
\noalign{\hrule}
height2pt&\omit&&\omit && \omit & \cr
& $[5,1]$ && $64$ && 
$\begin{matrix} -5r^3+145r-145/r+5/r^3\\ -27r^3-27r-27/r-27/r^3 \\ -9r^3-27r+27/r+9/r^3\\
 \end{matrix}$ & \cr
\noalign{\hrule}
height2pt&\omit&&\omit && \omit & \cr
& $[5,2]$ && $1280$ && 
$\begin{matrix} 8r^5-8r^3-16r+16/r+8/r^3-8/r^5 \\ 8r^5+16r^3-24r-24/r+16/r^3+8
/r^5\\ 3r^5+14r^3-3r+3/r-14/r^3-3/r^5 \\ \end{matrix} $ & \cr
\noalign{\hrule}
height2pt&\omit&&\omit && \omit & \cr
& $[5,3]$ && $64$ && $ \begin{matrix} -5r^7-270r^5+190r^3-745r+745/r-190/r^3+270/r^5+5/r^7 \\
 17r^7-270r^5+242r^3+659r+659/r+242/r^3-270/r^5+17/r^7\\ 
13r^7-250r^5+242r^3+217r-217/r-242/r^3+250/r^5-13/r^7 \\ \end{matrix} $ & \cr
\noalign{\hrule}
height2pt&\omit&&\omit && \omit & \cr
& $[6,1]$ && $1280$ && $\begin{matrix} 
5r^3-3r+3/r-5/r^3 \\ 0 \\ 0 \\ \end{matrix} $ & \cr
\noalign{\hrule}
height2pt&\omit&&\omit && \omit & \cr
& $[6,2]$ && $1280$ && $\begin{matrix} 
-13 r^5+5 r^3+50 r-50/r-5/r^3+13/r^5 \\ 2 r^5+18 r^3-20 r-20/r+18/r^3+2/r^5 \\ 3 r^5-3 r^3-6 r+6/r+3/r^3-3/r^5 \\
  \end{matrix} $ & \cr
\noalign{\hrule}
}  \hrule}
}}
\end{footnotesize}

\begin{footnotesize}
\smallskip
\vbox{
\bigskip\centerline{\def\quad{\hskip 0.6em\relax}\def\quod{\hskip 0.5em\relax }
\vbox{\offinterlineskip
\hrule\halign{&\vrule#&\strut\quod\hfil#\quad\cr
height2pt&\omit&&\omit && \omit & \cr
& $[a,c]$ && $ \gamma(a,c) $ && $P(a,c)$ & \cr
height2pt&\omit&&\omit && \omit & \cr
\noalign{\hrule}
& $[1,1]$ && $256$ && $\begin{matrix} 0 \\ 0\\  1 \\ 0 \\ 0\\ \end{matrix}$ & \cr
height2pt&\omit&&\omit && \omit & \cr
\noalign{\hrule}
& $[1,3]$ && $512$ && $\begin{matrix} 0\\ 0\\ -1/r -4 -r \\ 2/r -2r \\ -2/r +4 -2r\\ \end{matrix}$ & \cr
\noalign{\hrule}
& $[1,5]$ && $256$ && $\begin{matrix} 0\\ 0\\ 1/r^2 +16/r+20+16r+r^2  \\ -4/r^2-32/r+32 r+4r^2 \\ 4/r^2+16/r-40+16r+4r^2 \\ \end{matrix}$ & \cr
\noalign{\hrule}
& $[1,7]$ && $1024$ && $\begin{matrix} 0\\ 0\\ -2/r^2 -9/r-9 r-2 r^2  \\ 8/r^2+18/r-18 r-8r^2 \\ -6/r^2+6/r+6r-6 r^2 \\ \end{matrix}$ & \cr
\noalign{\hrule}
& $[3,3]$ && $1024$ && $\begin{matrix} 4/r^2 -4/r-4 r+4 r^2  \\ -10/r^2 +8/r-8r +10 r^2 \\ -11/r^2-8/r+30-8r+11 r^2 \\ -10/r^2+8/r -8r+10r^2 \\
4/r^2-4/r-4r+4r^2\\  \end{matrix}$ & \cr
\noalign{\hrule}
& $[3,5]$ && $1024$ && $\begin{matrix} -6/r^3-2/r^2+8/r+8r-2r^2-6r^3  \\ 24/r^3-18/r +18r-24 r^3 \\ -39/r^3+26/r^2-68 /r-68r+26 r^2-39r^3 \\ 
+30/r^3-8/r^2+46/r-46r +8r^2-30 r^3  \\
-6/r^3 -16/r^2+22/r+22r-16r^2-6r^3 \\  \end{matrix}$ & \cr
\noalign{\hrule}
& $[3,7]$ && $1024$ && $\begin{matrix} 4/r^4+8/r^3-16/r^2+16/r -24 +16r-16r^2 +8r^3+4r^4  \\ -22/r^4-32/r^3+56/r^2-16/r+16r-56r^2+32r^3+22r^4 \\
47/r^4+16/r^3-2/r^2+32/r+78+32r-2r^2+16r^3+47r^4\\ -42/r^4+32/r^3-96/r^2-144/r+144r+96r^2-32r^3+42r^4 \\
12/r^4 -24/r^3 +48 /r^2+48/r-168+48r+48r^2-24r^3+12r^4 \\  \end{matrix}$ & \cr  
\noalign{\hrule}
& $[5,5]$ && $256$ && $\begin{matrix} 60/r^4+64/r^3+72/r^2-96/r-200-96r+72r^2+64r^3+60r^4\\
-324/r^4-720/r^2+567/r-576r+720r^2+324r^4 \\
525/r^4-128/r^3+936/r^2-192/r+634-192r+936r^2-128r^3+525r^4 \\
-324/r^4-720/r^2+576/r-576r+720r^2+324r^4 \\
60/r^4+64/r^3+72/r^2-96/r-200-96r+72r^2+64r^3+60r^4\\ \end{matrix}$ & \cr 
\noalign{\hrule}
}  \hrule}
}}
\end{footnotesize}

\end{section}
\newpage
\begin{appendix}

\begin{section}{Correction to \lq\lq Igusa quartic and Steiner surfaces"}

\centerline{ by Shigeru Mukai}
\bigskip
This correction concerns the definition of the Fricke involution in \cite{Mukai}.
The paragraph before Theorem 2 in \cite{Mukai} 
is not precise enough to determine our Fricke involution of  $H_2/\Gamma_1(2)$.
In fact, the two explanations, the analytic and the moduli-theoretic one, conflict with each other.
It should read as follows:

\bigskip

\lq\lq The element  $\frac1{\sqrt{2}}\begin{pmatrix} 0& I_2 \\ -2I_2& 0 \end{pmatrix} \in 
{\rm Sp}(4, \mathbb R)$  belongs to
the normalizer of  $\Gamma_0(2)$, and induces an involution of the quotient  $H_2/\Gamma_0(2)$, which is called the Fricke involution.
Moduli-theoretically, the Fricke involution maps a pair $(A, G)$  to  $(A/G,  A_{(2)}/G)$.
We note that a 2-dimensional vector space $V$ is {\it almost} isomorphic to its dual $V^\vee$, 
or more precisely, we have canonically $V {\simeq} V^\vee \otimes \det V$.
Hence the quotient $A_{(2)}/G$  is canonically isomorphic to $G$ via the Weil pairing.
Therefore, the Fricke involution has a canonical lift on  $H_2/\Gamma_1(2)$, 
which we call the {\it (canonical) Fricke involution of} $H_2/\Gamma_1(2)$.
Our Fricke involution is the composite of  
$\frac1{\sqrt{2}}\begin{pmatrix} 0& I_2 \\ -2I_2& 0 \end{pmatrix}$ 
and the involution $\begin{pmatrix} J_2 & 0\\ 0 & J_2 \end{pmatrix}$, 
where we put $J_2 = \begin{pmatrix} 0& 1 \\ -1& 0 \end{pmatrix}$.
It commutes with each element of $\Gamma_0(2)/\Gamma_1(2) \simeq \mathfrak S_3$ and $H_2/\Gamma_1(2)$ has an action of  the product group $C_2 \times \mathfrak S_3$.
Two pairs $(A, G)$  and  $(A/G,  A_{(2)}/G)$  (in $H_2/\Gamma_1(2)$) are geometrically related to each other by Richelot's theorem.
See Remark~7.\rq\rq
\end{section}
\end{appendix}

\end{document}